\UseRawInputEncoding
\documentclass[pdflatex,sn-mathphys-num]{sn-jnl}

\usepackage{graphicx}%
\usepackage{multirow}%
\usepackage{amsmath,amssymb,amsfonts}%
\usepackage{amsthm}%
\usepackage{mathrsfs}%
\usepackage[title]{appendix}%
\usepackage{xcolor}%

\usepackage{textcomp}%
\usepackage{manyfoot}%
\usepackage{booktabs}%
\usepackage{algorithm}%
\usepackage{algorithmicx}%
\usepackage{algpseudocode}%
\usepackage{listings}%
\usepackage{pifont}
\usepackage{bbding}
\usepackage{enumerate}
\usepackage{subfigure}

\theoremstyle{thmstyleone}%
\newtheorem{theorem}{Theorem}
\newtheorem{lemma}{Lemma}
\newtheorem{proposition}{Proposition}%
\newtheorem{corollary}{Corollary}

\theoremstyle{thmstyletwo}%
\newtheorem{example}{Example}%
\newtheorem{remark}{Remark}%

\theoremstyle{thmstylethree}%
\newtheorem{definition}{Definition}%
\newtheorem{assumption}{Assumption}

\begin{document}

\title[
Stochastic Block Bregman Projection with Polyak-like Stepsize for Possibly Inconsistent Convex Feasibility Problems]{
	Stochastic Block Bregman Projection with Polyak-like Stepsize for Possibly Inconsistent Convex Feasibility Problems}


\author[1]{\fnm{Lu} \sur{Zhang}}\email{zhanglu21@nudt.edu.cn}
\author[1]{\fnm{Hongzhen} \sur{Chen}}\email{chenhongzhen25@nudt.edu.cn}
\author[1]{\fnm{Hongxia} \sur{Wang}}\email{wanghongxia@nudt.edu.cn}
\author*[1]{\fnm{Hui} \sur{Zhang}}\email{h.zhang1984@163.com}

\affil*[1]{\orgdiv{College of Science}, \orgname{National University of Defense Technology}, \orgaddress{\city{Changsha}, \postcode{410073}, \state{Hunan}, \country{China}}}


\abstract{
Stochastic projection algorithms for solving convex feasibility problems (CFPs) have attracted considerable attention due to their broad applicability. In this paper, we propose a unified stochastic bilevel reformulation for possibly inconsistent CFPs that combines proximity function minimization and structural regularization, leading to a feasible bilevel model with a unique and stable regularized solution. From the algorithmic perspective, we develop the stochastic block Bregman projection method with Polyak-like and projective stepsizes, which not only subsumes several recent stochastic projection algorithms but also induces new schemes tailored to specific problems. Moreover, we establish ergodic sublinear convergence rates for the expected inner function, as well as linear convergence in expectation to the inner minimizer set under a Bregman distance growth condition. In particular, the proposed Polyak-like stepsize ensures exact convergence in expectation for possibly inconsistent CFPs. Finally, numerical experiments demonstrate the effectiveness of the proposed method and its robustness to noise.
}




\maketitle
\section{Introduction}
\label{sec1}

The convex feasibility problem (CFP) seeks a point in the intersection of finitely many closed convex sets in Euclidean space $\mathbb{R}^n$.
Specifically, given closed convex sets $C_i\subset \mathbb{R}^n$ for $i=1,\ldots,m$, the CFP is formulated as
\begin{equation}
\label{eq1.1}
\operatorname{find}~x\in C=\cap_{i=1}^{m}C_i.
\end{equation}
CFPs arise across applied mathematics and physical sciences, with applications in machine learning \cite{bhattacharyya2004robust}, computerized tomography \cite{Herman1980,censor1987some}, and image recovery \cite{combettes1996convex,combettes1994inconsistent,combettes1999hard}.
The problem is termed consistent if $C\neq\emptyset$ and inconsistent otherwise.
In the consistent case, a classical approach is the projection onto convex sets (POCS) method \cite{bregman1965method}, which alternates orthogonal projections onto the individual sets $C_i$.
Projection methods were originally developed for linear equalities \cite{karczmarz1937angenaherte} and linear feasibility problems (LFPs) \cite{herman2008fast,agmon1954relaxation}, and were subsequently extended to split feasibility problems (SFPs) \cite{zhao2005several,censor2005multiple}, as well as parallel and block projection schemes \cite{censor1989,censor2001averaging}.

In many practical scenarios, however, the intersection $C$ may be empty due to inaccurate or imprecise information in estimation problems \cite{combettes2002use}, inaccurate constraints \cite{goldburg2003signal} and inadequate data modeling \cite{combettes2002methods}
in design problems. Specific examples are in the area of image recovery \cite{combettes1996convex} and computerized tomography \cite{Herman1980}.
As observed in \cite{combettes1994inconsistent}, the convergence behavior of classical projection algorithms in the inconsistent case is generally unsatisfactory, which motivates reformulations that seek meaningful approximate solutions.
A standard remedy replaces the hard feasibility requirement by minimizing an aggregate measure of constraint violation. In particular, \cite{combettes1994inconsistent} replaced (\ref{eq1.1}) by the weighted least-squares model
\begin{equation}
\label{eq1.2}
\min_{x\in\mathbb{R}^{ n}} F(x)=\frac{1}{2}\sum_{i=1}^{m}w_i\operatorname{dist}^2(x,C_i),
\end{equation}
where $\{w_i\}_{i=1}^{m}$ are strictly positive weights satisfying $\sum_{i=1}^{m}w_i=1$ and $\operatorname{dist}(x,C_i)=\inf_{y\in C_i}\|x-y\|$ denotes the Euclidean distance from $x$ to the set $C_i$. The function $F$ is referred to as a proximity function that measures the weighted average squared distance of $x$ to each constraint.
To solve (\ref{eq1.2}), \cite{combettes1994inconsistent} analyzed the parallel projection method (PPM) as a gradient descent method on $F$.
Similar approaches with proximity function minimization have been further developed in \cite{combettes1999hard,iiduka2016optimization}; see \cite{censor2018algorithms} for a comprehensive review of projection methods for inconsistent CFPs.

Recently, stochastic reformulations have become prominent for large-scale CFPs with large dimensions and a large number of convex sets. Following \cite{necoara2019randomized,necoara2022stochastic}, the CFP can be cast as minimizing an expected proximity function
\begin{equation}
\label{eq1.3}
\min_{x\in\mathbb{R}^{ n}} F(x)=\mathbb{E}[f_i(x)]=\sum_{i=1}^{m}\textbf{P}(i)f_i(x)
\end{equation}
with $f_i(x)=\frac{1}{2}\operatorname{dist}^2(x,C_i)$ and a probability distribution $\textbf{P}$ over $[m]$. This viewpoint leads to the stochastic block projection (SBP) method with linear convergence guarantees under bounded linear regularity in the consistent case \cite{necoara2019randomized,necoara2022stochastic}.
Although the convergence analysis of \cite{necoara2019randomized,necoara2022stochastic} focuses on $C\neq \emptyset$, the formulation (\ref{eq1.3}) is well suited for inconsistent CFPs.
In parallel, proximity function minimization based on Bregman distances has been studied for inconsistent CFPs \cite{byrne2001proximity,censor2004steered}.
While such models address inconsistency, they typically do not consider structural properties of the solution when the minimizer set is nonunique.
Therefore, an outer function $\psi:\mathbb{R}^n\rightarrow \mathbb{R}$ can be used to steer the iterations toward solutions with additional properties \cite{bregman1967relaxation,bauschke1997legendre,censor1994multiprojection,bauschke2003bregman}.
In the consistent setting, based on generalized Bregman projections with respect to $\psi$, \cite{lorenz2014linearized} presented a bilevel algorithmic framework
\begin{equation}
\label{eq1.5}
\min_{x\in\mathbb{R}^{n}}
\psi(x),~s.t.~x\in C=\cap_{i=1}^{m}C_i.
\end{equation}
A closely related strategy appears in \cite{beck2014first}, which aims to compute the minimal norm-like solution by minimizing a strongly convex outer function over the underlying solution set.

The stochastic formulation (\ref{eq1.3}) naturally extends to inconsistent CFPs through proximity function minimization but lacks structural priors, whereas the bilevel formulation (\ref{eq1.5}) incorporates regularization but is restricted to the consistent case.
These observations motivate us to propose the following stochastic bilevel reformulation for possibly inconsistent CFPs: 
\begin{equation}
\label{eq1.6}
\min_{x\in\mathbb{R}^{n}}
\psi(x),~s.t.~x\in \arg\min_{x\in\mathbb{R}^{n}} F(x)=\mathbb{E}[f_i(x)],
\end{equation}
where $\psi$ characterizes the structural property of the solution, each convex function $f_i:\mathbb{R}^n\rightarrow \mathbb{R}$ measures the deviation of $x$ from the set $C_i$ for $i=1,\ldots,m$. 
The formulation (\ref{eq1.6}) therefore selects a regularized solution from the minimizer set of the proximity function $F$, and the choices of $f_i$ are discussed in Example \ref{example1}.
To solve (\ref{eq1.6}), we propose the stochastic block Bregman projection (SBBP) method
\begin{equation}
\label{eq1.7}
\left\{
\begin{aligned}
x_{k+1}^*&=x_k^*-t_k\sum_{i\in J_k} w_{i,k}\nabla f_i(x_k),\\
x_{k+1}&=\nabla\psi^*(x_{k+1}^*).
\end{aligned}
\right.
\end{equation}
which provides a unified algorithmic framework that subsumes several state-of-the-art projection-based methods for CFPs, including \cite{schopfer2019linear,dai2024cut,necoara2022stochastic,yuan2022adaptively}.
To construct the stepsize $t_k$, we exploit two insights that have been established in recent literature.
First, by identifying SBBP as an instance of the stochastic mirror descent (SMD) method, we propose the decreasing mirror stochastic Polyak stepsize (DecmSPS) based on \cite{orvieto2022dynamics,bauschke2025boundedness}.
This stepsize guarantees exact convergence in expectation to the inner minimizer set $\bar{X}=\arg\min_{x\in\mathbb{R}^n}F(x)$ even when the CFP is inconsistent, overcoming the limitation that SMD equipped with other Polyak-like stepsizes converges only to a neighborhood of $\bar{X}$ without interpolation \cite{d2021stochastic}.
Second, via a cut-and-project paradigm \cite{dai2024cut,lorenz2014linearized}, we interpret SBBP as the Bregman projection onto a separating halfspace that separates the iterate and the optimal set $\bar{X}$, which leads to an \textsl{exact projective stepsize} and a \textsl{block adaptive extrapolated stepsize} for consistent CFPs.
We establish sublinear convergence under Assumptions \ref{assume.1} and \ref{assume.sampling}, and linear convergence under the Bregman distance growth condition (BDGC) imposed on $F$.



To demonstrate the versatility of SBBP, we specialize the framework to two representative problems, i.e., linear feasibility problems (LFPs) and split feasibility problems (SFPs).
For LFPs, our framework recovers the randomized iterated projection method \cite{leventhal2010randomized} and the randomized sparse Kaczmarz with averaging (RaSKA) \cite{tondji2023faster}. It also yields new sparse variants, including the randomized sparse Kaczmarz method with inequalities (\ref{eq6:1}) and RaSKA equipped with DecmSPS. These variants are suitable for finding sparse solutions of inconsistent LFPs; see Examples \ref{example3}--\ref{example4} and Remark \ref{remark5}--\ref{remark6}. For partition sampling with a fixed block size $|J_k|=\tau$, Corollary \ref{coro5..2} extends the optimal block size analysis of \cite{necoara2022stochastic} from orthogonal projections to the Bregman projections. Numerical experiments in Section \ref{sec6.1} verify the effectiveness of the proposed optimal block size and the benefits of Bregman projections for sparse recovery.
For SFPs, the randomized Bregman projection method \cite{schopfer2019linear} relies on consistency and geometric regularity of the underlying sets; our analysis is based on the stochastic reformulation of SFPs and a Bregman distance growth condition imposed on the inner function $F$, which allows SBBP to handle inconsistent SFPs. Numerical results in Section \ref{sec6.2} further show that mini-batching improves robustness under high noise levels.


\subsection{Contributions}

Specifically, the main contributions of this paper can be summarized as follows.

\begin{itemize}
\item \textbf{Stochastic bilevel reformulation.} We formulate possibly inconsistent CFPs as a feasible bilevel optimization problem. This reformulation moves the convergence analysis from the geometric regularity of the underlying sets to the Bregman distance growth condition imposed on the inner objective $F$, thereby providing a unified analytical framework that also accommodates inconsistent CFPs.
\item \textbf{Algorithmic framework.} We develop the stochastic block Bregman projection method with DecmSPS, an \textsl{exact projective stepsize}, and a \textsl{block adaptive extrapolated stepsize}. The proposed SBBP method extends the SBP method in \cite{necoara2019randomized,necoara2022stochastic} from the Euclidean setting to the Bregman setting, and extends the RBP method in \cite{schopfer2019linear} to possibly inconsistent problems. It not only unifies several existing randomized projection methods, but also induces new algorithmic variants tailored to specific problems (see Section \ref{sec5}).
\item \textbf{Convergence guarantees.} We establish ergodic sublinear convergence for the expected inner function and linear convergence in expectation to the inner minimizer set $\bar{X}$ for SBBP with DecmSPS, covering both consistent and inconsistent CFPs. For SBBP with projective stepsizes, analogous convergence guarantees are provided in the consistent case.
Under additional assumptions, SBBP further converges linearly in expectation to the unique solution of the bilevel problem (\ref{eq1.6}) in the consistent case.
In particular, for LFPs, the induced method from SBBP is a sparse generalization of the randomized iterated projections method \cite{leventhal2010randomized} together with convergence guarantees; for SFPs, the induced method from SBBP with DecmSPS guarantees exact convergence in the inconsistent case, thereby providing a different theoretical route from RBP \cite{schopfer2019linear} and extending the analysis to inconsistent SFPs. 
\end{itemize}

\subsection{Organization}
The remainder of the paper is organized as follows. Section \ref{sec2} introduces notation and reviews basic tools from convex analysis and Bregman projections. Section \ref{sec3} presents the SBBP framework along with several choices of stepsizes. Section \ref{sec4} establishes sublinear and linear convergence results for SBBP equipped with DecmSPS and projective stepsizes, respectively. Section \ref{sec5} develops two application examples and compares SBBP with existing methods. Section \ref{sec6} reports numerical results that illustrate the advantages of Bregman projection methods and the robustness of SBBP against noise. Finally, Section \ref{sec7} concludes the paper and outlines a few directions for future work; all proofs are collected in the Appendix.

\section{Preliminaries}
\label{sec2}
In this section, we summarize the notation and necessary preliminaries from convex analysis, Bregman distance theory, and probability theory.

\subsection{Notation}

We denote the set of integers $\{1,\ldots,m\}$ by $[m]$. For $t\in\mathbb{R}$, let $\lfloor t \rfloor$ denote the largest integer not exceeding $t$.
For a finite set $J$, define $|J|$ as the number of elements in $J$.
For a matrix $A\in\mathbb{R}^{m\times n}$, let $a_i^T$ denote the $i$-th row for $i\in [m]$, $A_J$ denote the submatrix formed by the rows of $A$ indexed by $J$, $A_{:,J}$ denote the submatrix formed by the columns of $A$ indexed by $J$, $\sigma_{\max}(A)$ denote the largest singular value of $A$, $\sigma_{\min}(A)$ denote the smallest singular value of $A$, $\|A\|_F$ denote the Frobenius norm, and $A^{\dagger}$ denote the Moore-Penrose pseudoinverse of $A$. For a vector $b\in\mathbb{R}^m$, $b_i$ denotes its $i$-th element for $i\in [m]$.
Unless otherwise stated, $\|\cdot\|$ denotes the $l_2$-norm.
The interior of a set $C$ is denoted by $\operatorname{int}(C)$.
The orthogonal projection of a point $x\in\mathbb{R}^n$ onto a closed convex set $C$ is denoted by $P_{C}(x)=\arg\min_{y\in C}\|x-y\|$.
The set of minimizers of $f:\mathbb{R}^n\rightarrow\mathbb{R}$ is denoted by $\arg\min_{x\in\mathbb{R}^n} f(x)$, the domain of $f$ is $\operatorname{dom}f=\{x\in\mathbb{R}^n\mid  f(x)<+\infty\}$. 

\subsection{Convex analysis tools}

We recall some well-known concepts of convex functions \cite{rockafellar1998variational}. 
Let $\psi:\mathbb{R}^n\rightarrow \mathbb{R}$ be a convex function. Define the subdifferential of $\psi$ at $x\in \mathbb{R}^n$ as
\begin{eqnarray}
\partial \psi(x):=\{x^*\in \mathbb{R}^n\mid \psi(y)\geq \psi(x)+\langle x^*,y-x\rangle, \forall y\in \mathbb{R}^n\}.\nonumber
\end{eqnarray}
The elements of $\partial \psi(x)$ are the subgradients of $\psi$ at $x$.
The Fenchel conjugate of $\psi$ is defined by
\begin{eqnarray}
\psi^*(x):=\sup_{z\in\mathbb{R}^{n}}\{\langle x,z\rangle-\psi(z)\}.\nonumber
\end{eqnarray}
The convex function $\psi:\mathbb{R}^n\rightarrow \mathbb{R}$ is $\mu$-strongly convex if there exists $\mu>0$ such that for all $x,y\in \mathbb{R}^n$ and $x^*\in \partial \psi(x)$ we have
\begin{eqnarray}
\psi(y)\geq \psi(x)+\langle x^*,y-x\rangle+\frac{\mu}{2}\|y-x\|^2.\nonumber
\end{eqnarray}
The convex function $\psi:\mathbb{R}^n\rightarrow \mathbb{R}$ is differentiable and has an $L$-Lipschitz continuous gradient if there exists $L>0$ such that for all $x,y\in \mathbb{R}^n$ we have
\[
\psi(y)\leq \psi(x)+\langle \nabla\psi(x),y-x\rangle+\frac{L}{2}\|y-x\|^2.
\]

\begin{lemma}
	Let $\psi:\mathbb{R}^n\rightarrow \mathbb{R}$ be $\mu$-strongly convex. Then its conjugate $\psi^*$ is differentiable and has a Lipschitz continuous gradient with modulus $1/\mu$, i.e.,
	$$\|\nabla \psi^*(x^*)-\nabla \psi^*(y^*)\|\leq \frac{1}{\mu}\|x^*-y^*\|,~\forall x^*,y^*\in\mathbb{R}^n.$$
	Moreover, the conditions $\psi(x)+\psi^*(x^*)=\langle x,x^*\rangle,~x^*\in\partial \psi(x)$, and $x\in\partial \psi^*(x^*)$ are equivalent.
\end{lemma}

\subsection{Bregman distance tools}
The Bregman distance is used as a tool to analyze and design the optimization algorithms.

\begin{definition}
	Let $\psi:\mathbb{R}^n\rightarrow \mathbb{R}$ be strongly convex. Then Bregman distance $D_{\psi}^{x^*}(x,y)$ between $x,y\in\mathbb{R}^n$ with respect to $\psi$ and a subgradient $x^*\in\partial \psi(x)$ is defined as
	$$D_{\psi}^{x^*}(x,y)=\psi(y)-\psi(x)-\langle x^*,y-x\rangle=\psi^*(x^*)-\langle x^*,y\rangle+\psi(y).$$
	If $\psi$ is differentiable then we have $\partial \psi(x)=\{\nabla \psi(x)\}$ and we simply write $D_{\psi}^{x^*}(x,y)=D_{\psi}(x,y)$.
\end{definition}

For strongly convex functions, the Bregman distance can be bounded by the Euclidean distance.
\begin{lemma}[Lemma 2.6, \cite{schopfer2019linear}]
	\label{lemma:2}
	Let $\psi:\mathbb{R}^n\rightarrow \mathbb{R}$ be $\alpha$-strongly convex. Then for all $x,y\in \mathbb{R}^n$ and $x^*\in \partial \psi(x), y^*\in \partial \psi(y)$, we have
	$$
	\frac{\alpha}{2}\|x-y\|^2\leq D_{\psi}^{x^*}(x,y)\leq \langle x^*-y^*,x-y\rangle\leq \|x^*-y^*\|\cdot\|x-y\|,
	$$
	and hence
	$D_{\psi}^{x^*}(x,y)=0$ if and only if $x=y$.
\end{lemma}

\begin{definition}
	Let $\psi:\mathbb{R}^n\rightarrow \mathbb{R}$ be $\alpha$-strongly convex and $C\subset \mathbb{R}^n$ be a nonempty closed convex set. The Bregman projection of $x$ onto $C$ with respect to $\psi$ and $ x^*\in \partial \psi(x)$ is the unique point $\Pi_{C}^{x^*}(x)\in C$ such that
	\begin{eqnarray}
	D_{\psi}^{x^*}(x,\Pi_{C}^{x^*}(x))=\min_{y\in C}D^{x^*}_\psi(x,y):=\operatorname{dist}_{\psi}^{x^*}(x,C)^2.
	\end{eqnarray}
\end{definition}

The Bregman projection can also be characterized by a variational inequality.

\begin{lemma}[Lemma 2.2, \cite{lorenz2014linearized}]
	\label{lemma2.1}
	Let $\psi:\mathbb{R}^n\rightarrow \mathbb{R}$ be $\mu$-strongly convex. Then a point $\hat{x}\in C$ is the Bregman projection of $x$ onto $C$ with respect to $\psi$ and $x^*\in\partial \psi(x)$ iff there is some $\hat{x}^{*}\in \partial \psi(\hat{x})$ such that one of the following equivalent conditions is fulfilled
	$$
	\langle \hat{x}^{*}-x^*,y-\hat{x}\rangle \geq 0,\forall y\in C,
	$$
	$$
	D_{\psi}^{\hat{x}^{*}}(\hat{x},y)\leq D_{\psi}^{x^{*}}(x,y)-D_{\psi}^{x^{*}}(x,\hat{x}),\forall y\in C.
	$$
	We call any such $\hat{x}^{*}$ an admissible subgradient for $\hat{x}=\Pi_{C}^{x^*}(x)$.
\end{lemma}

Note that for $\psi(x)=\frac{1}{2}\|x\|_2^2$, the Bregman projection $\Pi_{C}^{x^*}(x)$ reduces to the orthogonal projection onto $C$, i.e. $P_{C}(x)$.
The next lemma characterizes the Bregman projection onto halfspaces. 

\begin{lemma}[Lemma 2.4, \cite{lorenz2014linearized}]
	\label{lemma2.2}
	Let $\psi:\mathbb{R}^n\rightarrow \mathbb{R}$ be $\mu$-strongly convex, $A\in\mathbb{R}^{m\times n}$, $b\in\mathbb{R}^{m}$, $u\in\mathbb{R}^{n}$ and $\beta\in\mathbb{R}$.
	The Bregman projection of $x\in\mathbb{R}^{n}$ onto the halfspace $H(u,\beta)=\{x\in\mathbb{R}^{n}\mid \langle u,x\rangle\leq\beta\}$ with $u\neq 0$ is
	$$\hat{x}:=\Pi_{H(u,\beta)}^{x^*}(x)=\nabla \psi^*(x^*-\hat{t}\cdot u),$$
	where $\hat{t}\in \mathbb{R}$ is a solution of 
	$$
	\min_{t\in \mathbb{R}} \psi^*(x^*-t\cdot u)+t\cdot\beta.
	$$
\end{lemma}


\subsection{Probability results}

Let $J$ be a random set-valued map with values in $2^{[m]}$ with probability distribution $\textbf{P}$. For any index $i\in [m]$, let $p_i=\textbf{P}(i\in J)$ define the probability that the index $i\in [m]$ is included in the set $J$.
Consequently, as established in \cite{necoara2022stochastic}, for any scalars $\theta_i,~i\in [m]$, the following identity holds in expectation
\begin{eqnarray}
\label{eq3.6}
	\mathbb{E}[\sum_{i\in J} \theta_i]
	=\sum_{J\subseteq [m]}(\sum_{i\in J}\theta_i) \textbf{P}(J)
	=\sum_{i\in [m]} \theta_i \sum_{J:i\in J}\textbf{P}(J)
	=\sum_{i\in [m]} \theta_i\textbf{P}(i\in J)
	=\sum_{i\in [m]} \theta_ip_i.
	\end{eqnarray}

\section{The proposed algorithmic framework}
\label{sec3}

In this section, we propose the SBBP method based on our proposed stochastic reformulation for solving possibly inconsistent CFPs. 
Subsequently, we propose two stepsize schemes for SBBP, namely, the Polyak-like stepsize and projective stepsizes. All proofs in this section are deferred to Appendix \ref{secA1}.

\subsection{Stochastic reformulation}
\label{sec3.1}

In practical applications, one often does not know in advance whether the convex feasibility problem (\ref{eq1.1}) is consistent or not. 
To address possibly inconsistent CFPs while enabling a unique and regularized solution, we propose the following bilevel reformulation of (\ref{eq1.1}): 
\begin{equation}
\label{eq3.3}
\min_{x\in\mathbb{R}^{n}}
 \psi(x),~s.t.~x\in\bar{X}= \arg\min_{x\in\mathbb{R}^{n}}F(x),
\end{equation}
where the proximity function is defined by $F(x)=\mathbb{E}[f_i(x)]=\sum_{i=1}^{m}\textbf{P}(i)f_i(x)$. 
Here each deviation function $f_i$ quantifies the deviation of a point $x$ from the set $C_i$, and the index $i$ is sampled from $[m]$ according to the probability distribution $\textbf{P}$ with $\textbf{P}(i)>0$ and $\sum_{i\in [m]}\textbf{P}(i)=1$ for all $i\in [m]$. We impose the following assumptions.
\begin{assumption}
	\begin{enumerate}[(a)]
		\item The inner minimizer set $\bar{X}= \arg\min_{x\in\mathbb{R}^{n}}F(x)$ is nonempty.
		\item $\psi$ is strongly convex over $\mathbb{R}^n$ with modulus $\mu>0$;
		\item $f_i,~i=1,\ldots,m$ are convex and have Lipschitz continuous gradients with modulus $L_i>0$. 
	\end{enumerate}
	\label{assume.1}
\end{assumption}

Under Assumption \ref{assume.1}, the bilevel problem admits a unique solution.
This bilevel formulation (\ref{eq3.3}) is designed to simultaneously address two fundamental challenges that arise in CFPs.
First, the inner optimization relaxes the exact feasibility requirement and seeks a point that is not too far from each individual set, in the spirit of the stochastic reformulation in \cite{necoara2019randomized}.
If $C$ is nonempty, then, under Assumption \ref{assume.1}, the proximity function $F$ attains the minimum value 0 exactly on $C$, and hence the inner problem in (\ref{eq3.3}) is equivalent to the original CFP (\ref{eq1.1}). If the intersection $C$ is empty, the reformulation (\ref{eq3.3}) transforms the infeasible CFP into a feasible bilevel optimization problem.
The incorporation of the deviation function $f_i$, also known as the loss function in machine learning, not only extends
the applicability of the framework to accommodate inconsistent CFPs, but also allows us to describe the infeasibility
based on the geometry of the constraint $C_i$. In particular, when $f_i(x)=\frac{1}{2}\|x-P_{C_i}(x)\|^2$, the inner
optimization problem reduces to the stochastic smooth optimization formulation studied in
\cite{necoara2019randomized,necoara2022stochastic,polyak2001random,nedic2010random}.

Second, since the inner minimizer set $\bar{X}$ may be non-singleton, the outer level selects a unique regularized solution by minimizing the strongly convex function $\psi$ over $\bar{X}$, as in \cite{lorenz2014linearized,beck2014first}.
This viewpoint naturally motivates the use of Bregman projections with respect to $\psi$ in algorithm design \cite{lorenz2014linearized}.
With appropriate choices of the regularizer $\psi$ and the deviation function $f_i$, the bilevel optimization problem (\ref{eq3.3}) encompasses various specific problems, including sparse solutions for SFPs and linear systems \cite{lorenz2014linearized,schopfer2019linear}, and minimum-norm solutions for CFPs \cite{necoara2022stochastic}.


\subsection{The proposed method}

Building on the stochastic reformulation established in Section \ref{sec3.1}, we propose the stochastic block Bregman projection (SBBP) method, as summarized in Algorithm \ref{al1}.
To address the computational challenges in large-scale problems, we avoid employing all sets at each iteration by sampling a minibatch of indices. 
Specifically, at the $k$-th iteration, a subset of indices $J_k=\{i_k^{1},\ldots,i_k^{\tau_k}\}$ is drawn from the probability distribution $\textbf{P}$. This approach significantly reduces the per-iteration computational cost and enables parallel implementation.
Furthermore, we employ the outer function $\psi$ as a kernel function to induce Bregman projections, which is designed to steer the iterates toward the unique regularized solution of the bilevel problem (\ref{eq3.3}). 

\begin{algorithm}[h]
		\caption{Stochastic block Bregman projection method}\label{al1}
	\begin{algorithmic}[1]
		\State \textbf{Initialize}: choose $x_{0}\in \mathbb{R}^{n}$ and $x_0^*\in\partial \psi(x_0)$ 
		\State \textbf{for} $k=0,1,\ldots,N$
		\State \hspace*{2em}draw a sample $J_k\sim \textbf{P}$ with $|J_k|=\tau_k$
		\State \hspace*{2em}update the stepsize $t_k$ according to the chosen rule
		\State \hspace*{2em}update the dual variable: $$x_{k+1}^*=x_k^*-t_k\sum_{i\in J_k} w_{i,k}\nabla f_i(x_k)$$
		\State \hspace*{2em}update the primal variable: $$x_{k+1}=\nabla\psi^*(x_{k+1}^*)$$
		\State \textbf{end} 
	\end{algorithmic}
\end{algorithm}

In Algorithm \ref{al1}, the weights $w_k=(w_{i,k})_{i\in J_k}$ are chosen to be positive and sum to one, i.e., $\sum_{i\in J_k} w_{i,k}=1$. The selection of the stepsize $t_k$ is deferred to the next subsection.
Concerning the stopping criterion, a standard approach for stochastic algorithms is to fix a maximum number of
iterations, which is widely used in the literature \cite{lorenz2014linearized,schopfer2019linear}. Alternatively, one
may consider stopping criteria based on the proximity function, such as $F(x_k)-F(x_{k+1})\leq \varepsilon$ or $\|\nabla
F(x_k)\|\leq \varepsilon$ for a suitably small accuracy $\varepsilon>0$, as suggested in
\cite{combettes1994inconsistent}. However, evaluating these full quantities is often computationally prohibitive.
Therefore, we define $f_{J_k}(x)=\sum_{i\in J_k} w_{i,k} f_i(x)$ and employ a mini-batch gradient inspired by the
stopping rule in \cite{nguyen2017sarah}, i.e., $\|\nabla f_{J_k}(x_k)\|\leq \varepsilon$.

The key advantage of the SBBP framework lies in its broad applicability and verifiability. Specific examples of this framework are detailed in Example \ref{example1}.

\begin{example}
	\label{example1}
	Algorithm \ref{al1} encompasses a broad range of existing methods as special cases through appropriate selections of $\psi$ and $f_i$.
	\begin{enumerate}[(a)]
		\item When $f_i(x)= \frac{1}{2}\operatorname{dist}(A_ix,Q_i)^2$, Algorithm \ref{al1} induces a randomized Bregman projection method for split feasibility problems. From the cut-and-project perspective, it is closely aligned with the RBP method in \cite{schopfer2019linear}, but differs in the construction of the separating halfspace; see Section \ref{sec5.2}. 
		\item When $\psi(x)=\frac{1}{2}\|x\|^2$ and $f_i(x)=\frac{1}{2}\operatorname{dist}(x,C_i)^2$, Algorithm \ref{al1} reduces to the stochastic block projection (SBP) method in \cite{necoara2022stochastic} for solving the CFP, i.e.,
		$$
		x_{k+1}=x_k-t_k\sum_{i\in J_k} w_{i,k}(x_k-P_{C_i}(x_k)).
		$$
		The SBBP method generalizes SBP from the Euclidean setting to the Bregman setting, thereby enabling the recovery of solutions characterized by the regularizer $\psi$. 
		\item When $f_i(x)=\frac{1}{2}\|Ax-b\|^2_{H_i}$ with $H_i=S_i(S_i^TAA^TS_i)^{\dagger}S_i^T$, Algorithm \ref{al1} simplifies to the sketched Bregman projection method in \cite{yuan2022adaptively} for solving the linear system $Ax=b$. 
		In particular, the randomized sparse Kaczmarz method in \cite{schopfer2019linear} is a special case of the proposed SBBP method.
		\item When $J_k=[m]$ and $w_{i,k}=\frac{1}{m}$ for all $i\in [m]$ and $k\geq 0$, Algorithm \ref{al1} reduces to the Bregman projection method based on cutting halfspaces in \cite{dai2024cut}.
		In this sense, Algorithm \ref{al1} can be interpreted as a stochastic extension of this deterministic method in \cite{dai2024cut}.
	\end{enumerate}
\end{example}

\subsection{The choice of stepsize $t_k$}
\label{sec3.3}
In this subsection, we derive stepsizes $t_k$ by analyzing the SBBP method from two distinct but complementary perspectives: the optimization perspective yielding Polyak-like stepsizes, and the geometric perspective yielding projective stepsizes.

\subsubsection{Decreasing mirror stochastic Polyak stepsize via SMD}
\label{sec3.3.1}


Building on insights from \cite{azizan2019characterization,d2021stochastic,d2021stochastic2}, the SBBP method can be interpreted as an instance of the stochastic mirror descent (SMD) method
\begin{equation}
\label{eq3:1}
x_{k+1}=\arg\min_{x\in\mathbb{R}^{ n}} \{
t_k\langle \sum_{i\in J_k} w_{i,k}\nabla f_i(x_k),x-x_k\rangle+D_{\psi}^{x_k^*}(x_k,x)
\}.
\end{equation}
Under Assumption \ref{assume.1}, it follows from the definition of Bregman distance and the Lipschitz continuity of $\nabla \psi^*$ with $1/\mu$ that for $x_{k+1}^*\in\partial\psi(x_{k+1})$ and any $x\in\mathbb{R}^n$,
\begin{equation*}
\begin{aligned}
&~~~~D_{\psi}^{x_{k+1}^*}(x_{k+1},x)\\
&=\psi^*(x_{k+1}^*)+\psi(x)-\langle x_{k+1}^*,x\rangle\\
&\leq \psi^*(x_k^*)-t_k\langle \sum_{i\in J_k} w_{i,k}\nabla f_i(x_k),x_k\rangle+\frac{t_k^2}{2\mu}\|\sum_{i\in J_k} w_{i,k}\nabla f_i(x_k)\|^2+\psi(x)-\langle x_{k+1}^*,x\rangle\\
&=D_{\psi}^{x_{k}^*}(x_{k},x)-t_k\langle \sum_{i\in J_k} w_{i,k}\nabla f_i(x_k),x_k-x\rangle+\frac{t_k^2}{2\mu}\|\sum_{i\in J_k} w_{i,k}\nabla f_i(x_k)\|^2.
\end{aligned}
\end{equation*}
Using the convexity of $f_i,~i=1,\ldots,m$ that for any $x\in \mathbb{R}^n$ we have 
\begin{equation}
\label{eq4.2}
D_{\psi}^{x_{k+1}^*}(x_{k+1},x)
\leq
D_{\psi}^{x_{k}^*}(x_{k},x)
-t_k\sum_{i\in J_k} w_{i,k} (f_i(x_k)-f_i(x))
+\frac{t_k^2}{2\mu}\|\sum_{i\in J_k} w_{i,k}\nabla f_i(x_k)\|^2.
\end{equation}
Minimizing the right-hand side of the above inequality with respect to $t_k$ yields the \textsl{mirror stochastic Polyak stepsize} (mSPS) studied in \cite{loizou2021stochastic,d2021stochastic,d2021stochastic2,orvieto2022dynamics}:
\begin{equation}
\label{eq4.3}
t_k=\mu\cdot\frac{\sum_{i\in J_k} w_{i,k} (f_i(x_k)-f_i(x)) }{\|\sum_{i\in J_k} w_{i,k}\nabla f_i(x_k)\|^2}.
\end{equation}
Motivated by \cite{orvieto2022dynamics,bauschke2025boundedness}, we use a lower bound $l_{J}^{*}$ satisfying $l_{J}^{*}\leq \inf_{x\in\mathbb{R}^n} \sum_{i\in J} w_{i,k}f_i(x)$ and propose the \textsl{decreasing mirror stochastic Polyak stepsize} (DecmSPS):
\begin{equation}
\label{eq4:1}
t_k=\lambda_k \min\left\{ \mu\cdot\frac{\sum_{i\in J_k} w_{i,k} (f_i(x_k)-l_{J_k}^{*}) }{\|\sum_{i\in J_k} w_{i,k}\nabla f_i(x_k)\|^2},\frac{t_{k-1}}{\lambda_{k-1}}\right\},
\end{equation}
where $\lambda_k$ is any nonincreasing positive sequence with initialization $\lambda_{-1}=\lambda_{0}$, $t_{-1}=\gamma_b>0$. 
The DecmSPS integrates a diminishing stepsize with the adaptiveness of mirror stochastic Polyak stepsize (mSPS). 
For some special problems, the deviation function $f_i$ is non-negative hence one can pick $l_{J}^{*}=0$ for any $J\subset [m]$.

The following lemma establishes the boundedness of the DecmSPS sequence. The proof proceeds analogously to Lemma 1 in \cite{orvieto2022dynamics} and is omitted for brevity.

\begin{lemma}
	\label{lemma4.1}
	Under Assumption \ref{assume.1}, let $t_k$ be generated by DecmSPS (\ref{eq4:1}) in Algorithm \ref{al1}. Let $\lambda_{-1}=\lambda_{0},t_{-1}=\gamma_b>0$ and $\{\lambda_k\}$ be any nonincreasing positive sequence. Then we have $t_{k+1}\leq t_k$ and 
	$$
	\lambda_k\tilde{L}
	\leq t_k\leq\frac{\gamma_b}{\lambda_0}\lambda_k,
	$$	
	where $\tilde{L}=\min\left\{\frac{\mu}{2L_{\max}},\frac{\gamma_b}{\lambda_0} \right\}$ and $L_{\max}=\max_{i\in[m]}L_i$.
\end{lemma}


\subsubsection{Projective stepsize via cut-and-project}
\label{sec3.3.2}
The geometric cut-and-project perspective discussed in \cite{dai2024cut} can be naturally extended to our proposed SBBP method. 
Since each $f_i$ is convex and has a Lipschitz-continuous gradient with modulus $L_i$ for $i\in [m]$, the following inequality holds for any $x,y\in\mathbb{R}^n$:
\begin{equation}
\label{eq3.1}
\langle \nabla f_i(x)-\nabla f_i(y),x-y\rangle\geq \frac{1}{L_i}\|\nabla f_i(x)-\nabla f_i(y)\|^2.
\end{equation}
To construct the separating hyperspace, we impose the following assumptions. 

\begin{assumption}
	\label{assume.projective}
Assume that the intersection $C$ in (\ref{eq1.1}) is nonempty. In addition, $f_i(x)\geq 0$ and $f_i(x)=0$ if and only if $x\in C_i$ for all $i\in [m]$.
\end{assumption}

Under Assumption \ref{assume.projective}, for any $x\in C$, since each $f_i$ is convex, differentiable, nonnegative, and attains its minimum value $0$ on $C_i$, it holds that $\nabla f_i(x)=0$ for all $i\in [m]$. It follows from $F(x)=\mathbb{E}[f_i(x)]\geq 0$ for all $x$ that $F(x)=0$ for every $x\in C$ and $\bar{X}=\arg\min_{x\in\mathbb{R}^{n}}F(x)=C$.

Setting $y=x_k$ and $x\in\bar{X}=C$ in (\ref{eq3.1}) and summing over $i\in J_k\subset [m]$ with weights $w_{i,k}$, we construct a separating halfspace
\begin{equation}
\label{eq4.11}
H_k=\{x\in \mathbb{R}^n\mid \langle \alpha_k ,x\rangle\leq \beta_k \},
\end{equation}
where the weighted gradient $\alpha_k$ and the scalar $\beta_k$ are defined as
$$\alpha_k=\sum_{i\in J_k} w_{i,k}\nabla f_i(x_k),~\beta_k=\langle \sum_{i\in J_k} w_{i,k}\nabla f_i(x_k) ,x_k\rangle -\sum_{i\in J_k} \frac{w_{i,k}}{L_i}\|\nabla f_i(x_k)\|^2.$$ 
The separating halfspace (\ref{eq4.11}) contains the inner minimizer set $\bar{X}$ and strictly separates the iterate $x_k$ from the inner minimizer set $\bar{X}$, in the sense that $x_k\notin H_k$ and $\bar{X}\subseteq H_k$.

The SBBP method can thus be interpreted as performing the Bregman projection of $x_k$ onto $H_k$, namely,
$$x_{k+1}=\arg\min_{x\in\mathbb{R}^{ n}}D_{\psi}^{x_k^*}(x_k,x),~s.t.\langle\alpha_k ,x\rangle\leq \beta_k.$$
According to Lemma \ref{lemma2.2}, it is equivalent to the dual update
\begin{equation}
\label{eq4.5}
\begin{aligned}
x_{k+1}^*&=x_k^*-\hat{t}_k\cdot \alpha_k,\\
x_{k+1}&=\nabla\psi^*(x_{k+1}^*),
\end{aligned}
\end{equation} 
where the \textsl{exact projective stepsize} $\hat{t}_k$ is a solution of \begin{equation}
\label{eq4.6}
\min_{t\in \mathbb{R}} \psi^*(x^*-t\cdot \alpha_k)+t\cdot\beta_k.
\end{equation}
As demonstrated in \cite{lorenz2014linearized,schopfer2019linear,dai2024cut,tondji2023faster}, the update (\ref{eq4.5}) equipped with cheaper inexact stepsizes instead of exact stepsize $\hat{t}_k$ is also valid. 
For notational brevity, we define the adaptivity parameter as
\begin{equation}
\label{eq4.1}
L_{\operatorname{adapt}}^{k,\tau}=\frac{\|\sum_{i\in J_k} w_{i,k}\nabla f_i(x_k)\|^2}{\sum_{i\in J_k} w_{i,k}\|\nabla f_i(x_k)\|^2},
\end{equation}
which is well-defined provided that $\sum_{i\in J_k} w_{i,k}\|\nabla f_i(x_k)\|^2\neq 0$. The following theorem characterizes the Bregman projection and provides bounds for permissible stepsizes.


\begin{theorem}
	\label{th4.1}
Under Assumptions \ref{assume.1} and \ref{assume.projective}, let $\hat{t}_k$ be a solution to (\ref{eq4.6}). Denote $L_{\max}=\max_{i\in [m]}L_i$. Then the Bregman projection of $x_k$ onto the halfspace $H_k=\{x\in \mathbb{R}^n\mid \langle \alpha_k ,x\rangle\leq \beta_k \}$ is given by
	$$x_{k+1}=\nabla\psi^*(x_k^*-\hat{t}_k\cdot \alpha_k).$$
	Moreover, if the point $\tilde{x}_{k+1}=\nabla\psi^*(x_k^*-\tilde{t}_k\cdot \alpha_k)$ with \textsl{inexact projective stepsize} $\tilde{t}_k\in \left(0,\frac{\mu}{L_{\max}L_{\operatorname{adapt}}^{k,\tau}}\right]$ lies in the hyperplane $\{x\in \mathbb{R}^n\mid \langle \alpha_k ,x\rangle= \beta_k \}$, then it must be the unique point of the Bregman projection, that is,
	$$
	\nabla\psi^*(x_k^*-\hat{t}_k\cdot \alpha_k)=\nabla\psi^*(x_k^*-\tilde{t}_k\cdot \alpha_k)=\Pi_{H_k}^{x_k^*}(x_k).
	$$ 
\end{theorem}

The following lemma establishes that Algorithm \ref{al1} achieves a monotonic decrease in terms of Bregman distance with either exact or inexact projective stepsizes $t_k$.

\begin{lemma}[Basic descent lemma]
	\label{lemma5.1}
Under Assumptions \ref{assume.1} and \ref{assume.projective}, let $(x_k,x_k^*)$ be the $k$-th iterate generated by Algorithm \ref{al1} with the stepsize $t_k$.
	\begin{enumerate}[(a)]
	\item If the stepsize $t_k$ lies in the interval $\left(0, \frac{2\mu}{L_{\max}L_{\operatorname{adapt}}^{k,\tau}}\right)$, then for any $x\in H_k$ we have
	$$D_{\psi}^{x_{k+1}^*}(x_{k+1},x)
		\leq 
		D_{\psi}^{x_{k}^*}(x_{k},x)
		-\left(\frac{t_k}{L_{\max}}-\frac{t_k^2}{2\mu}L_{\operatorname{adapt}}^{k,\tau}\right) \sum_{i\in J_k} w_{i,k}\|\nabla f_i(x_k)\|^2.$$
	\item If $t_k$ is the \textsl{exact projective stepsize} given by (\ref{eq4.6}) or the \textsl{block adaptive extrapolated stepsize}
	\begin{equation}
	\label{eq6.9}
	t_k=\frac{\mu}{L_{\max}L_{\operatorname{adapt}}^{k,\tau}},
	\end{equation}
	then for any $x\in H_k$ we have
	$$D_{\psi}^{x_{k+1}^*}(x_{k+1},x)
	\leq 
	D_{\psi}^{x_{k}^*}(x_{k},x)
	-\frac{\mu}{2L_{\max}^2L_{\operatorname{adapt}}^{k,\tau}} \sum_{i\in J_k} w_{i,k}\|\nabla f_i(x_k)\|^2.$$
	\end{enumerate}
\end{lemma}



\section{Convergence analysis}
\label{sec4}
In this section, we present the main convergence results for Algorithm \ref{al1} equipped with either the \textsl{decreasing mirror stochastic Polyak stepsize} or \textsl{projective stepsizes}.
All detailed proofs omitted in this section are deferred to Appendix \ref{secA2}.

\subsection{The key assumption}
\label{sec4.1}
We now outline the fundamental assumptions, definitions and the sampling method used in the convergence analysis.
In order to establish linear convergence rates, we impose the following growth condition on the inner function $F$ in terms of the Bregman distance generated by $\psi$.

\begin{assumption}[Bregman distance growth condition (BDGC) \cite{zhang2021proximal,bauschke2019linear}]
	\label{assume.2}
	Let $\bar{X}=\arg\min_{x\in\mathbb{R}^n}F(x)$ and $\bar{F}=\min_{x\in\mathbb{R}^n}F(x)$. Assume there exists a constant $\gamma>0$ such that for every $x\in \operatorname{int} \operatorname{dom}\psi$ and $x^*\in\partial\psi(x)$,
	$$F(x)-\overline{F}
	\geq
	\gamma \inf_{z\in \bar{X}}D_{\psi}^{x^*}(x,z).$$
\end{assumption}


\begin{remark}
We relate BDGC to several regularity conditions commonly used in the convergence analysis of first-order algorithms.
\begin{enumerate}[(a)]
\item Under Assumption \ref{assume.1}, the BDGC generalizes the quadratic growth condition (QGC) in \cite{necoara2019linear,zhang2020new}. Specifically, if $\psi$ is $\mu$-strongly convex, the BDGC implies that $F$ satisfies the QGC with constant $\gamma\mu$, i.e., for $x^*\in\partial \psi(x)$,
\begin{equation*}
\begin{aligned}
F(x)-\overline{F}
&\geq
\gamma \inf_{z\in \bar{X}}D_{\psi}^{x^*}(x,z)\geq \frac{\gamma\mu}{2}\operatorname{dist}(x,\bar{X})^2.
\end{aligned}
\end{equation*}
\item Under Assumption \ref{assume.1}, if the convex function $F$ satisfies the BDGC, then it also satisfies the Polyak-\L{}ojasiewicz (PL) inequality:
\begin{equation}
\label{eq4:9}
\sqrt{\frac{\gamma\mu}{2}}\cdot\sqrt{F(x)-\bar{F}}\leq \|\nabla F(x)\|,~\forall x\in\mathbb{R}^n,
\end{equation}
which has been discussed in Theorem 4.4 of \cite{dai2024cut}.
\item In the Euclidean setting where $\psi(x)=\frac{1}{2}\|x\|^2,~F(x)=\frac{1}{2}\mathbb{E}[\operatorname{dist}(x,C_i)^2]$, if the intersection $C$ is nonempty, the BDGC reduces to the classical linear regularity assumption in \cite{necoara2022stochastic}: $$\mathbb{E}[\operatorname{dist}(x,C_i)^2]\geq \gamma \operatorname{dist}(x,C)^2.$$
Crucially, the BDGC imposed on $F$ does not require the intersection $C$ to be nonempty, which allows the SBBP method to address inconsistent CFPs. 
\end{enumerate}
\end{remark}


\begin{example}
	\label{example2}
A special instance of the parameter $\gamma$ was investigated in \cite{schopfer2019linear} for consistent linear systems $Ax=b$ with $A\in\mathbb{R}^{m\times n}$ and $b\in\mathbb{R}^{m}$, where the matrix $A$ is assumed to be normalized by rows. Consider the sampling method $\textbf{P}(i)=\frac{1}{m}$, $f_i(x)=\frac{1}{2}(a_i^Tx-b_i)^2$ and $F(x)=\mathbb{E}[f_i(x)]=\frac{1}{2m}\|Ax-b\|_2^2$. The stochastic reformulation (\ref{eq3.3}) then specializes to the regularized Basis Pursuit problem
\begin{equation}
\label{eq4:20}
\min_{x\in\mathbb{R}^{n}}
\psi(x)=\lambda\|x\|_1+\frac{1}{2}\|x\|^2,~s.t.~Ax=b,
\end{equation}
see e.g. \cite{cai2009convergence,schopfer2019linear}.
Let $\hat{x}$ be the solution of the linear system $Ax=b$, and let $\tilde{\sigma}_{\min}(A)=\min\{\sigma_{\min}(A_{:,J})\mid J\subset \{1,\ldots,n\},A_{:,J}\neq 0\}$ and $|\hat{x}|_{\min}=\min\{|\hat{x}_j|\mid j\in \operatorname{supp}(\hat{x})\}$. As established in Lemma 3.1 of \cite{schopfer2019linear}, for all $x\in\mathbb{R}^n$ with $\partial\psi(x)\cap\mathcal{R}(A^T)\neq \emptyset$ and all $x^*=A^T y\in\partial\psi(x)\cap\mathcal{R}(A^T)$, it holds that
$\|Ax-b\|_2^2\geq \eta\cdot D_{\psi}^{x^*}(x,\hat{x})$
with
\begin{equation}
\label{eq6:13}
\eta=
\tilde{\sigma}^2_{\min}(A)\cdot\frac{|\hat{x}|_{\min}}{|\hat{x}|_{\min}+2\lambda}.
\end{equation}
The above inequality implies that $F$ satisfies BDGC with constant $\gamma=\frac{\eta}{2m}$.
\end{example}

To characterize the variance of the stochastic gradients, we introduce two regularity constants motivated by the analysis in \cite{necoara2022stochastic}.
We then specify the sampling rule and introduce an interpolation proxy used in the convergence analysis.

\begin{definition}[stochastic smooth regularity constant]
	\label{def1}
Let the index $i$ be sampled from the probability distribution $\textbf{P}$. Define 
	$$
	L=\sup_{x\in\mathbb{R}^{n}}\frac{\|\mathbb{E}[\nabla f_i(x)]\|^2}{\mathbb{E}[\|\nabla f_i(x)\|^2]}
	$$
	with the convention $0/0=0$.
In particular, Jensen's inequality yields $L\in (0,1]$.
\end{definition}

\begin{definition}[block smooth regularity constant]
	\label{def2}
Let $\{x_k\}_{k\geq 0}$ be generated by Algorithm \ref{al1}.
Let the mini-batch $J_k$ be sampled from the block distribution $\textbf{P}$ with $|J_k|=\tau$, and let weights $w_{i,k}$ satisfy $\sum_{i\in J_k}w_{i,k}=1$ and $w_{i,k}>0$. Define
$$
L_{\tau}^{\operatorname{block}}=\sup_{k\geq 0}\frac{\|\sum_{i\in J_k} w_{i,k}\nabla f_i(x_k)\|^2}{\sum_{i\in J_k} w_{i,k}\|\nabla f_i(x_k)\|^2}
$$	
with the convention $0/0=0$.
In particular, Jensen's inequality yields $L_{\tau}^{\operatorname{block}}\in (0,1]$.
\end{definition}
 
The constant $L_{\tau}^{\operatorname{block}}$ serves as a mini-batch analogue of $L$ and provides an upper bound for $L_{\operatorname{adapt}}^{k,\tau}$ defined in (\ref{eq4.1}), ensuring $L_{\operatorname{adapt}}^{k,\tau}\leq  L_{\tau}^{\operatorname{block}}\leq 1$ for all $k\geq 0$.


\begin{assumption}[Sampling rule]\label{assume.sampling}
Fix a batch size $\tau\in [1,m]$. At each iteration $k\geq 0$, a mini-batch $J_k\subset [m]$ is sampled with $|J_k|=\tau$ and the following probability distribution
$$
\textbf{P}(i)=\frac{1}{m},~ p_i=\textbf{P}(i\in J_k)=\frac{\tau}{m},~ \forall i\in [m],
$$
and we assign uniform weights $w_{i,k}=\frac{1}{\tau}$ for all $i\in J_k$.
The sampling distribution of $J_k$ is independent of the past $\mathcal{F}_k=\{J_0,\ldots,J_{k-1}\}$.
\end{assumption}
Assumption \ref{assume.sampling} covers several sampling schemes, including uniform sampling without replacement and partition sampling \cite{necoara2022stochastic}, and it guarantees the unbiasedness required in the following convergence analysis
$$
\mathbb{E}\!\left[\sum_{i\in J_k} w_{i,k} f_i(x)\right]=F(x),~ \forall x\in\mathbb{R}^n.
$$

\begin{definition}\label{def:sigma}
Let $l_J^*$ be any lower bound satisfying $l_J^*\leq \inf_{x\in\mathbb{R}^n}\sum_{i\in J}w_if_i(x)$. Fix any $\hat{x}\in\bar{X}$, and define
\begin{equation}
\label{eq4:8}
\sigma_{\tau}^2=\mathbb{E}\left[\sum_{i\in J}w_{i}f_i(\hat{x})-l_{J}^*\right].
\end{equation}
\end{definition}
As stated in \cite{loizou2021stochastic,d2021stochastic,d2021stochastic2}, we say that the interpolation condition holds when there exists $\hat{x}\in\bar{X}$ such that $f_i(\hat{x})=\inf_{x\in\mathbb{R}^n} f_i(x)$ for all $i\in [m]$. In this case, we have $\sigma_{\tau}^2=0$.

\subsection{Convergence of SBBP with DecmSPS}
\label{sec4.2}

We first establish the sublinear convergence of SBBP with DecmSPS on the average function value under Assumption \ref{assume.1}.

\begin{theorem}[sublinear convergence]
	\label{th4:1}
Let $\{(x_{k},x_{k}^*)\}_{k\geq 0}$ be generated by Algorithm \ref{al1} with $t_k$ given by DecmSPS (\ref{eq4:1}). Let $\{\lambda_k\}$ be nonincreasing and satisfy $0<\lambda_k\leq 1$. Under Assumption \ref{assume.1} and Assumption \ref{assume.sampling}, for any $\hat{x}\in\bar{X}$ we have
	\begin{equation*}
	\mathbb{E}\left[F(\bar{x}_{K})-\bar{F}\right]
	\leq
	\frac{1}{K}\left(\frac{2D}{\lambda_{K-1}\tilde{L}}+\sigma_{\tau}^2 \sum_{k=0}^{K-1}\lambda_k\right),
	\end{equation*}
	where $\bar{x}_{K}=\frac{1}{K}\sum_{k=0}^{K-1}x_k$, $D=\max_{k\in [K-1]} D_{\psi}^{x_{k}^*}(x_{k},\hat{x})<\infty$, and $\tilde{L}=\min\left\{\frac{\mu}{2L_{\max}},\frac{\gamma_b}{\lambda_0} \right\}$.
\end{theorem}

\begin{remark}
\begin{enumerate}[(a)]
\item While we assume $0<\lambda_k\leq 1$ for brevity, Theorem \ref{th4:1} remains valid under the relaxed condition $0<\lambda_k< 2$.
		\item In the special case of $\psi(x)=\frac{1}{2}\|x\|^2$, SBBP reduces to the stochastic gradient descent (SGD) method, and DecmSPS reduces to the decreasing stochastic Polyak stepsize (DecSPS). Furthermore, Theorem \ref{th4:1} recovers the sublinear convergence of SGD with decreasing stochastic Polyak stepsize of Theorem 3 in \cite{orvieto2022dynamics}. Therefore, Theorem \ref{th4:1} extends the results of the stochastic gradient descent method in \cite{orvieto2022dynamics} to the more general stochastic mirror descent method.
\end{enumerate}
\end{remark}


\begin{corollary}
	\label{coro4:1}
	Under the assumptions of Theorem \ref{th4:1}, for $\hat{x}\in\bar{X}$ we have the following results.
	\begin{enumerate}[(a)]
	\item In the interpolated case ($\sigma_{\tau}^2=0$), let $\lambda_k=\frac{1}{c},k\in\mathbb{N}$ with $c\geq 1$, yielding an $\mathcal{O}(1/K)$ convergence rate
	$$ \mathbb{E}[F(\bar{x}_{K})-\bar{F}]
	\leq \frac{2Dc}{K\tilde{L}}.$$ 
	\item In the non-interpolated case ($\sigma_{\tau}^2>0$), let $\lambda_k=\frac{1}{c\sqrt{k+1}},k\in\mathbb{N}$ with $c>0$ yielding an $\mathcal{O}(1/\sqrt{K})$ convergence rate
	$$\mathbb{E}[F(\bar{x}_{K})-\bar{F}]
	\leq
	\frac{2cD/\tilde{L}+2\sigma_{\tau}^2/c}{\sqrt{K}}.$$
	\end{enumerate}
\end{corollary}

The validity of Theorem \ref{th4:1} and Corollary \ref{coro4:1} relies on the boundedness of the iterates $D<\infty$. Without additional assumptions beyond Assumption \ref{assume.1}, this requirement is standard in the convergence analysis of adaptive stepsize schemes, see e.g. \cite{orvieto2022dynamics,jiang2023adaptive}.
By imposing the Bregman distance growth condition on $F$, we obtain linear convergence up to an additive noise term; in particular, under interpolation ($\sigma_\tau^2=0$) this yields linear convergence in expectation to $\bar{X}$.

\begin{theorem}[linear convergence]
	\label{th4:2}
	Let $\{(x_{k},x_{k}^*)\}_{k\geq 0}$ be generated by Algorithm \ref{al1} with DecmSPS (\ref{eq4:1}). Suppose Assumption \ref{assume.1}, Assumption \ref{assume.2}, and Assumption \ref{assume.sampling} hold. Then for $\tilde{L}=\min\left\{\frac{\mu}{2L_{\max}},\frac{\gamma_b}{\lambda_0} \right\}$ and $0<\lambda_k<2$ we have
	\begin{equation*}
	\mathbb{E}\left[\operatorname{dist}_{\psi}^{x_{k+1}^*}(x_{k+1},\bar{X})^2\right]
	\leq
	\left[1-\lambda_k\left(1-\frac{\lambda_k}{2}\right)\gamma \tilde{L}\right]\mathbb{E}\left[\operatorname{dist}_{\psi}^{x_{k}^*}(x_{k},\bar{X})^2\right]
	+\lambda_k\frac{\gamma_b}{\lambda_0}\sigma_{\tau}^2.
	\end{equation*}
\end{theorem}


\begin{corollary}
	\label{coro5}
Under the setting of Theorem \ref{th4:2}, we obtain the following convergence result.
	\begin{enumerate}[(a)]
		\item In the interpolated case ($\sigma_{\tau}^2=0$), let $\lambda_k=\frac{1}{c},k\in\mathbb{N}$ with $c>\frac{1}{2}$, yielding a linear convergence rate 
	$$
	\mathbb{E}\left[\operatorname{dist}_{\psi}^{x_{k}^*}(x_{k},\bar{X})^2\right]
	\leq
	\left[1-\frac{1}{c}\left(1-\frac{1}{2c}\right)\gamma\tilde{L}\right]^k\operatorname{dist}_{\psi}^{x_{0}^*}(x_{0},\bar{X})^2.
	$$
		\item In the non-interpolated case ($\sigma_{\tau}^2>0$), let $\lambda_k=\frac{1}{c\sqrt{k+1}},k\in\mathbb{N}$ with $c>0$, yielding convergence to the set $\bar{X}$ 
	$$
	\begin{aligned}
	&~~~~\mathbb{E}\left[\operatorname{dist}_{\psi}^{x_{k+1}^*}(x_{k+1},\bar{X})^2\right]\\
	&\leq
	\left[1-\frac{1}{c\sqrt{k+1}}\left(1-\frac{1}{2c\sqrt{k+1}}\right) \gamma \tilde{L} \right]\mathbb{E}\left[\operatorname{dist}_{\psi}^{x_{k}^*}(x_{k},\bar{X})^2\right]
	+\frac{\gamma_b}{\lambda_0}\cdot\frac{\sigma_{\tau}^2}{c\sqrt{k+1}}.
	\end{aligned}
	$$
 	\end{enumerate}
\end{corollary}

\begin{remark}
Unlike SMD with mirror stochastic Polyak stepsize, which guarantees convergence only to a neighborhood of the solution in Theorem 7 and Theorem 5 of \cite{d2021stochastic2} in the case of non-interpolation, SMD with our proposed DecmSPS yields a contraction up to an additive noise term; under interpolation, it converges linearly in expectation to $\bar{X}$.
\end{remark}

In the setting of Corollary \ref{coro5} (a), DecmSPS simplifies as
\begin{equation}
\label{eq7:6}
t_k=\min\left\{\frac{\mu}{c}\cdot\frac{\sum_{i\in J_k} w_{i,k} (f_i(x_k)-l_{J_k}^{*}) }{\|\sum_{i\in J_k} w_{i,k}\nabla f_i(x_k)\|^2},t_{k-1}\right\},k\geq 0,
\end{equation}
where $t_{-1}=\gamma_b>0$. 
According to Lemma \ref{lemma4.1}, DecmSPS is bounded by
$$\tilde{L}=\min\left\{\frac{\mu}{2cL_{\max}} ,\gamma_b\right\}\leq t_k\leq \gamma_b.$$

\begin{remark}
\label{remark4.1}
The performance of SBBP is affected by the choice of $\gamma_b$ and $c$. In practice, these parameters may require careful tuning.
\begin{enumerate}[(a)]
	\item On the one hand, in the non-interpolated case, Theorem \ref{th4:2} indicates that the radius of the convergence neighborhood is proportional to $\gamma_b$; thus, $\gamma_b$ should not be chosen too large to ensure convergence to a small neighborhood. 
	On the other hand, to prevent the stepsize from degenerating to a constant stepsize, the condition $\gamma_b> \frac{\mu}{2cL_{\max}}$ is required; otherwise, the adaptive stepsize (\ref{eq7:6}) reduces to a constant stepsize $t_k\equiv\gamma_b$.
\item	The parameter $c>0$ controls the overall scale of the Polyak stepsize.
The choice of $c$ involves a trade-off between the convergence rate and the size of the neighborhood. 
We will derive the relationship between $c$ and $\gamma_b$ for linear feasibility problems in Corollary \ref{corol5.3}.
\end{enumerate}
\end{remark}

\subsection{Convergence of SBBP with projective stepsizes}

We first establish a sublinear convergence result for Algorithm \ref{al1} with projective stepsizes under Assumptions \ref{assume.1}, \ref{assume.projective}, and \ref{assume.sampling}.

\begin{theorem}[sublinear convergence]
	\label{th3}
Under Assumptions \ref{assume.1}, \ref{assume.projective}, and \ref{assume.sampling}, let $L_{\max}=\max_{i=1}^m L_i$ and $\{x_k\}_{k\geq 0}$ be generated by Algorithm \ref{al1} with the \textsl{exact projective stepsize} $\hat{t}_k$ in (\ref{eq4.6}) or the \textsl{block adaptive extrapolated stepsize} in (\ref{eq6.9}). 
\begin{enumerate}[(a)]
\item For every integer $K\geq 1$, it holds that
$$
\min_{0\leq k\leq K-1}\mathbb{E}\|\nabla F(x_k)\|^2
\leq
\frac{2LL_{\max}^2L_{\tau}^{\operatorname{block}}}{\mu K}
\operatorname{dist}_{\psi}^{x_{0}^*}(x_{0},\bar{X})^2.
$$
\item Assuming further that $F$ satisfies the PL inequality (\ref{eq4:9}), the average $\bar{x}_K=\frac{1}{K}\sum_{k=0}^{K-1}x_k$ satisfies
$$
\mathbb{E}\left[F(\bar{x}_{K})-\bar{F}\right]
\leq
\frac{4LL_{\max}^2L_{\tau}^{\operatorname{block}}}{\mu^2\gamma K}
\operatorname{dist}_{\psi}^{x_{0}^*}(x_{0},\bar{X})^2.
$$
\end{enumerate}
\end{theorem}

We additionally assume that $F$ satisfies the Bregman distance growth condition to derive the linear convergence of Algorithm \ref{al1} with projective stepsizes.

\begin{theorem}[linear convergence]
	\label{th5.5}
Under Assumptions \ref{assume.1}, \ref{assume.projective}, \ref{assume.2}, and \ref{assume.sampling}, let $L_{\max}=\max_{i=1}^m L_i$ and $\{x_k\}_{k\geq 0}$ be generated by Algorithm \ref{al1} with \textsl{exact projective stepsize} $\hat{t}_k$ (\ref{eq4.6}) or \textsl{inexact projective stepsize} (\ref{eq6.9}).
	The sequence $\{x_k\}_{k\geq 0}$ converges to the set $\bar{X}$ linearly in expectation
	$$ \mathbb{E}[\operatorname{dist}_{\psi}^{x_{k}^*}(x_{k},\bar{X})^2]
	\leq
	\left(1-\frac{\gamma^2\mu^2}{4L_{\max}^2L_{\tau}^{\operatorname{block}}L}\right)^k\operatorname{dist}_{\psi}^{x_{0}^*}(x_{0},\bar{X})^2.$$
\end{theorem}

Theorems \ref{th4:2} and \ref{th5.5} establish linear convergence in expectation of Algorithm \ref{al1} to the inner minimizer set $\bar{X}$ in terms of the Bregman distance, providing a broader generalization of the convergence results for stochastic block projection method in \cite{necoara2019randomized,necoara2022stochastic}. However, these results do not guarantee convergence to the solution of the bilevel problem (\ref{eq3.3}). Under additional assumptions, we further establish linear convergence in expectation to the unique solution of (\ref{eq3.3}) in the consistent case, by adapting the proof of Theorem 1 in \cite{akhtiamov2026implicit} to our setting, as stated below. 

\begin{proposition}
	\label{th5.6}
	Under Assumptions \ref{assume.1}--\ref{assume.projective}, assume further that $\psi$ has an $L_{\psi}$-Lipschitz continuous gradient and each $f_i$ is $\mu_f$-strongly convex with modulus $\mu_f>0$ for $i\in [m]$. Let $\hat{x}$ be the unique solution of the bilevel problem (\ref{eq3.3}).
	Let $L_{\max}=\max_{i=1}^m L_i$ and $\{x_k\}_{k\geq 0}$ be generated by Algorithm \ref{al1} with the stepsize $t_k\in (0,\frac{\mu}{L_{\max}}]$.
	Then the sequence $\{x_k\}_{k\geq 0}$ converges in expectation to the unique solution $\hat{x}$ of the bilevel optimization problem (\ref{eq3.3}) with a linear rate
	$$\mathbb{E}D_{\psi}(x_{k+1},\hat{x})\leq
	\left(1- \frac{\mu_f}{L_{\psi}}t_k\right)\mathbb{E}D_{\psi}(x_{k},\hat{x}).$$
\end{proposition}

In Proposition \ref{th5.6}, by the definitions of smoothness and strong convexity for $\psi$ and $f_i$, we have $L_{\psi}>\mu$ and $L_i> \mu_f$. Hence, for any $t_k\in (0,\frac{\mu}{L_{\max}}]$, it holds that $1- \frac{\mu_f}{L_{\psi}}t_k>0$.
For inconsistent CFPs, we leave a rigorous convergence analysis showing that SBBP converges to a solution of the bilevel problem (\ref{eq3.3}) for future work.

\begin{table}[h]
	\caption{The key convergence results for SBBP with different choices of block sizes and stepsizes under Assumption \ref{assume.1} and Assumption \ref{assume.2}. For the projective stepsizes, Assumption \ref{assume.projective} is additionally required. Assume that $F$ is $L_F$-smooth. Let $\mathcal{E}_k=\lambda_k\left(1-\frac{\lambda_k}{2}\right)$ with $\lambda_k=\frac{1}{c\sqrt{k+1}},~c>0$ and $\mathcal{D}_k=\operatorname{dist}_{\psi}^{x_{k}^*}(x_{k},\bar{X})^2$ for simplicity.
	}
	\label{table3}
	\begin{tabular*}{\textwidth}{@{\extracolsep{\fill}}cccc@{\extracolsep{\fill}}}
		\toprule
		Block size   & Stepsize & Allows $C=\emptyset$? & Convergence rates 
		\\
		\midrule
		$|J_k|=m$ \cite{dai2024cut}  & $t_k=\frac{\mu}{L_{F}}$ & \ding{55} &	$  \mathcal{D}_k
		\leq
		\left(1-\frac{\gamma^2\mu^2}{4L_{F}^2}\right)^{k}\displaystyle\mathcal{D}_{0}$ 
		\\
		\hline
		\multirow{2}{*}{$|J_k|=\tau$}	& (\ref{eq4:1}) & \ding{51} &$
		\mathbb{E}[\mathcal{D}_{k+1}]
		\leq
		(1-\mathcal{E}_k\gamma \tilde{L})\mathbb{E}[\mathcal{D}_k]
		+\frac{\gamma_b\sigma_{\tau}^2}{\lambda_0 c}\cdot\frac{1}{\sqrt{k+1}}
		$ 
		\\
		\cmidrule{ 2 - 4}
		&  (\ref{eq4.6}), (\ref{eq6.9}) & \ding{55} &$ \mathbb{E}[\mathcal{D}_k]
		\leq
		\left(1-\frac{\gamma^2\mu^2}{4L_{\max}^2L_{\tau}^{\operatorname{block}}L}\right)^{k}\mathcal{D}_{0}$
		\\
		\hline
		$|J_k|=1$ & $t_k=\frac{\mu}{L_{i_k}}$ & \ding{55} & $\mathbb{E}[\mathcal{D}_k]\leq\left(1-\frac{\gamma^2\mu^2}{4L_{\max}^2L}\right)^k\mathcal{D}_0$ 
		\\
		\botrule
	\end{tabular*}	
\end{table}

Table \ref{table3} summarizes the convergence results of SBBP for different block sizes $|J_k|$. In the full-batch case ($J_k=[m]$), the SBBP method reduces to the deterministic Bregman projection method driven by $\nabla F(x_k)$ in \cite{dai2024cut}, and Theorem \ref{th5.5} recovers Theorem 4.4 in \cite{dai2024cut}. In the single-sample case ($J_k=\{i_k\}$), the SBBP method reduces to the stochastic Bregman projection method driven by $\nabla f_{i_k}(x_k)$. Notably, SBBP with DecmSPS (\ref{eq4:1}) extends the theory to the inconsistent case $C=\emptyset$, yielding linear convergence in expectation to the set $\bar{X}$ up to a decaying convergence horizon. Moreover, when $L_{\tau}^{\operatorname{block}}<1$ (and it may even occur that $L_{\tau}^{\operatorname{block}}\thickapprox 0$), using blocks ($|J_k|=\tau$) can be faster than the single-sample case ($|J_k|=1$).

\section{Application examples}
\label{sec5}
This section presents two representative applications to illustrate how to instantiate our algorithmic framework and apply the associated convergence results.

\subsection{The linear feasibility problems}
\label{sec5.1}

As a first application, we consider sparse solutions to linear feasibility problems
\begin{equation}
\label{eq6:6}
\text{find}~x\in C=\{x\mid A_{\mathcal{I}}x=b_{\mathcal{I}},
A_{\mathcal{J}}x\leq b_{\mathcal{J}}\},
\end{equation}
where $A_{\mathcal{I}}\in\mathbb{R}^{m_{\mathcal{I}}\times n}$ and $A_{\mathcal{J}}\in\mathbb{R}^{m_{\mathcal{J}}\times n}$. Let $A=[A_{\mathcal{I}};A_{\mathcal{J}}]\in\mathbb{R}^{m\times n}$ and $b=[b_{\mathcal{I}};b_{\mathcal{J}}]\in\mathbb{R}^{m}$, where $m=m_{\mathcal{I}}+m_{\mathcal{J}}$ and $[m]=\mathcal{I}\cup \mathcal{J}$. 
Motivated by the simplicity and efficiency of the randomized Kaczmarz method, the authors of \cite{leventhal2010randomized} proposed a variant for consistent linear inequalities and established linear convergence under the Hoffman error bound; see also \cite{briskman2015block,zhang2025block,morshed2020accelerated,de2017sampling,morshed2022sampling} for related developments.

As described in Section \ref{sec3.1}, the bilevel stochastic reformulation of (\ref{eq6:6}) is given by
\begin{equation}
\label{eq6:7}
\min_{x\in\mathbb{R}^n}\psi(x)=\lambda\|x\|_1+\frac{1}{2}\|x\|_2^2,~s.t.~ x\in\arg\min_{x\in\mathbb{R}^n}F(x)=\mathbb{E}[f_i(x)],
\end{equation}
with the regularization parameter $\lambda>0$ and the deviation function 
\begin{equation}
\label{eq6:12}
f_i(x)=\frac{1}{2}\frac{\|e(Ax-b)_{i}\|^2}{\|a_{i}\|^2}.
\end{equation}
Here the function $e:\mathbb{R}^m\rightarrow \mathbb{R}^m$ is defined componentwise by
$$e(x)_i=\left\{
\begin{aligned}
&x_i,~i\in \mathcal{I},\\
&\max\{x_i,0 \},~i\in \mathcal{J}.
\end{aligned}
\right.$$ 
Each $f_i$ is continuously differentiable with a gradient Lipschitz constant $L_i=1$ for $i\in [m]$, and the strong convexity constant of $\psi$ is $\mu=1$.
Applying the SBBP method to (\ref{eq6:7}) yields the randomized sparse Kaczmarz method with inequalities
\begin{equation}
\label{eq6:1}
\left\{
\begin{aligned}
x_{k+1}^*&=x_k^*-t_k\sum_{i\in J_k}w_{i,k}\frac{e(Ax_k-b)_{i}}{\|a_{i}\|^2}a_{i},\\
x_{k+1}&=S_{\lambda}(x_{k+1}^*),
\end{aligned}
\right.
\end{equation}
where the stepsize $t_k$ is chosen either by DecmSPS or projective stepsizes in Lemma \ref{lemma5.1}, and the soft-thresholding operator is defined as $S_{\lambda}(x)=\max\{|x|-\lambda,0\}\cdot \text{sign}(x)$.

When the system (\ref{eq6:6}) is consistent, the deviation function (\ref{eq6:12}) satisfies Assumption \ref{assume.projective}, and therefore $\bar{X}=C$.
Under the Bregman distance growth condition on $F$, Theorem \ref{th4:2} implies that the scheme (\ref{eq6:1}) with DecmSPS converges to $\bar{X}$ for both inconsistent and consistent systems. In the consistent case, Theorem \ref{th5.5} further yields linear convergence in expectation for (\ref{eq6:1}) equipped with projective stepsizes.

\begin{example}
	\label{example3}
	Consider the setting where $\psi(x)=\frac{1}{2}\|x\|^2$ and $|J_k|=1$, and we focus on the consistent linear feasibility problem (\ref{eq6:6}), i.e., $C\neq \emptyset$ and $C=\bar{X}$. 
	Then for the deviation function in (\ref{eq6:12}) we have $f_{i}(x)=\frac{1}{2}\|\nabla f_{i}(x)\|^2,~i\in [m]$, thus DecmSPS in (\ref{eq7:6}) with $\lambda_k=1/c$, $\gamma_b\geq 1/(2c)$ and $l_{i_k}^{*}=0$ simplifies to a constant stepsize $t_k=1/(2c)$ for all $k\geq 0$.
	Consequently, the scheme (\ref{eq6:1}) reduces to the randomized iterated projections method in \cite{leventhal2010randomized}.	
Assume that the rows of $A$ are normalized so that $\|a_i\|=1$ for all $i\in [m]$.
Under the sampling rule $\textbf{P}(i)=1/m,~i\in [m]$, the Bregman distance growth condition with $\gamma$ reduces to the Hoffman error bound with $\beta=1/\sqrt{m\gamma}$ of the form 
\begin{equation}
\label{eq6:10}
\operatorname{dist}(x,C)\leq \beta\|e(Ax-b)\|_2,
\end{equation}
see \cite{hoffman2003approximate_selected,leventhal2010randomized} for more references.
	In particular, by specializing the proof of Theorem \ref{th5.5} to the linear inequalities, we can obtain the linear convergence rate in Theorem 4.3 of \cite{leventhal2010randomized} for the randomized iterated projections method, i.e.,
$$\mathbb{E}[\operatorname{dist}(x_{k+1},C)^2]	\leq	\left(1-\frac{1}{\beta^2m}\right)\mathbb{E}[\operatorname{dist}(x_{k},C)^2].$$
\end{example}


\begin{remark}
	\label{remark5}
\begin{enumerate}[(a)]
	\item The linear convergence of the randomized iterated projections method in \cite{leventhal2010randomized} requires $C\neq\emptyset$. In contrast, our theoretical framework replaces the Hoffman error bound with the more general Bregman distance growth condition, which only requires $\bar{X}\neq\emptyset$ and therefore allows the scheme (\ref{eq6:1}) to handle inconsistent linear inequalities. 	
\item The new scheme (\ref{eq6:1}) is a sparse generalization of the randomized iterated projections method in \cite{leventhal2010randomized}, and its convergence is guaranteed by Theorem \ref{th4:2} and Theorem \ref{th5.5}. By incorporating Bregman projections, the scheme can approximate the solution with sparse structure for both inconsistent and consistent linear inequalities, which is a byproduct of this paper.
\end{enumerate} 
\end{remark}

\begin{example}
	\label{example4}
When $\mathcal{J}=\emptyset$, the linear feasibility problem (\ref{eq6:6}) reduces to the linear system $Ax=b$, and the scheme (\ref{eq6:1}) reduces to the randomized sparse Kaczmarz with averaging (RaSKA) method \cite{tondji2023faster} as follows
\begin{equation}
\label{eq6:14}
\left\{
\begin{aligned}
x_{k+1}^*&=x_k^*-t_k\sum_{i\in J_k}w_{i,k}\frac{a_{i}^Tx_k-b_i}{\|a_{i}\|^2}a_{i},\\
x_{k+1}&=S_{\lambda}(x_{k+1}^*).
\end{aligned}
\right.
\end{equation}
As shown in Example \ref{example2}, we have $\|Ax-b\|_2^2\geq \eta\cdot D_{\psi}^{x^*}(x,\hat{x})$ with $\eta$ given in (\ref{eq6:13}), where $\hat{x}$ is the unique solution of $\min_{x\in\mathbb{R}^n}\psi(x),~s.t.~ Ax=b$.
By Theorem \ref{th4:1}, RaSKA (\ref{eq6:14}) with DecmSPS admits the convergence guarantee for inconsistent linear
systems; related efforts include
\cite{zouzias2013randomized,epperly2025randomized,marshall2023optimal,tondji2024adaptive}. By Theorem \ref{th5.5}, we
further obtain the convergence of RaSKA with projective stepsizes for consistent linear systems, which has been studied in
\cite{tondji2023faster}. 
By specializing the proof of Theorem \ref{th5.5} to the linear systems, we can obtain the linear convergence rate in Theorem 3.2 of \cite{schopfer2019linear} for the
randomized sparse Kaczmarz (RaSK) method
$$\mathbb{E}D_{\psi}^{x_{k}^*}(x_{k},\hat{x})
\leq
\left(1-\frac{\eta}{2m}\right)^k D_{\psi}^{x_{0}^*}(x_{0},\hat{x}).$$
\end{example}

\begin{remark}
\label{remark6}
When the linear system is inconsistent, the convergence of RaSKA in \cite{tondji2023faster} is affected by the convergence horizon induced by errors. In contrast, by extending the DecSPS for SGD to DecmSPS for RaSKA, Theorem \ref{th4:2} establishes exact expected convergence in the inconsistent case.
\end{remark}

\subsubsection{The optimal block size ($\tau\in [1,m]$)}
\label{sec5.1.1}
Assume that the linear feasibility problem (\ref{eq6:6}) is consistent, which implies the intersection $C$ in (\ref{eq6:6}) is nonempty. Building on the analysis in \cite{necoara2022stochastic}, we consider partition sampling with a fixed block size $|J_k|=\tau,~\forall k\geq 0$ for the scheme (\ref{eq6:1}). Let $\kappa$ be a permutation of $[m]$, and assume that $l=m/\tau$ is a positive integer. We define a random row partition $\{J_1,\ldots,J_l\}$ of $[m]$ into $l$ blocks of size $\tau$ by
\begin{equation*}
J_i=\{\kappa(r):~r=(i-1)\tau+1,\ldots,i\tau\},\quad i=1,\ldots,l.
\end{equation*}
At each iteration, we sample an index $i_k\in[l]$ uniformly and set $J_k=J_{i_k}$. For a row-normalized matrix $A$, Lemma 4.1 in \cite{necoara2022stochastic} provides a high-probability upper bound on $\sigma_{\max}^2(A_J)$ for the blocks $J\in\{J_1,\ldots,J_l\}$ in such a random partition.

\begin{lemma}[Lemma 4.1, \cite{necoara2022stochastic}]
	\label{lemma5:1}
Let $A$ be a normalized matrix with $m$ rows and let $\{J_1,\ldots,J_l\}$ be a randomized partition of the row indices with $l\geq \sigma_{\max}^2(A)$. Then, with probability at least $1-m^{-1}$, the row partition satisfies
\begin{equation*}
\max_{J\in\{J_1,\ldots,J_l\}} \sigma_{\max}^2(A_J)\leq 6\log(1+m).
\end{equation*}
\end{lemma}


\begin{corollary}
	\label{coro5..2}
Under the assumptions of Lemma \ref{lemma5:1}, we choose the \textsl{optimal block size} as
	\begin{equation}
	\label{eq6.10}
	\tau^{*}=\left\lfloor\frac{m}{\sigma_{\max}^2(A)}\right\rfloor.
	\end{equation}
	Assume that $l=m/\tau^{*}$ is a positive integer.
	Then, with probability at least $1-m^{-1}$, we have the following convergence results.
	\begin{enumerate}[(a)]
		\item Let $\{x_k,x_k^*\}$ be generated by the scheme (\ref{eq6:1}) with DecmSPS (\ref{eq7:6}). Then for $c>\frac{1}{2}$ it holds that
		$$
		\mathbb{E}\!\left[\operatorname{dist}_{\psi}^{x_{k}^*}(x_{k},\bar{X})^2\right]
		\leq
		\left(1-\frac{1}{c}\left(1-\frac{1}{2c}\right)\gamma\tilde{L}\right)^k
		\operatorname{dist}_{\psi}^{x_{0}^*}(x_{0},\bar{X})^2,
		$$
	where $\tilde{L}=\min\left\{\frac{\tau^*}{12\log(1+m)},c\gamma_b\right\}$.
		\item Let $\{x_k,x_k^*\}$ be generated by the scheme (\ref{eq6:1}) with either the \textsl{exact projective stepsize} or the \textsl{block adaptive extrapolated stepsize}. Under Assumption \ref{assume.projective}, it holds that
		$$ \mathbb{E}\!\left[\operatorname{dist}_{\psi}^{x_{k}^*}(x_{k},\bar{X})^2\right]
		\leq
		\left(1-\frac{m\gamma}{6\sigma_{\max}^2(A)\log(1+m)}\right)^k\operatorname{dist}_{\psi}^{x_{0}^*}(x_{0},\bar{X})^2.$$
	\end{enumerate}
\end{corollary}


Corollary \ref{coro5..2} establishes linear convergence of the block Bregman projection method for linear feasibility problems and provides an explicit choice of the block size $\tau$ in (\ref{eq6.10}).
It not only extends the linear convergence result of Corollary 4.1 in \cite{necoara2022stochastic} from the block orthogonal projection method to the block Bregman projection method, but also shows that, in the Euclidean setting ($\psi(x)=\frac{1}{2}\|x\|^2$), Corollary \ref{coro5..2} (b) recovers the convergence rate in Corollary 4.1 of \cite{necoara2022stochastic} under $m\gamma=1/\beta^2$ in Example \ref{example3}.
Moreover, in view of Corollary \ref{coro5..2} (a), we obtain a principled choice of the parameters $\gamma_b$ and $c$ for DecmSPS (\ref{eq7:6}).

\begin{corollary}
	\label{corol5.3}
	To ensure that the DecmSPS rule in (\ref{eq7:6}) is adaptive, the parameters $\gamma_b$ and $c$ should satisfy
	\begin{equation}
	\label{eq7:7}
	c\gamma_b\geq \frac{\tau}{12\log(1+m)}.
	\end{equation}
\end{corollary}



\subsection{The split feasibility problems}
\label{sec5.2}
As a second application, we consider the split feasibility problem \cite{lorenz2014linearized,schopfer2019linear} of the form
\begin{equation}
\label{eq7:5}
\operatorname{find}~x\in C=\cap_{i=1}^{m}C_i:=\{x\in\mathbb{R}^n\mid A_ix\in Q_i\},
\end{equation}
where $A_i\in\mathbb{R}^{m_i\times n}$ is a matrix, $Q_i\subset \mathbb{R}^{m_i}$ is a closed convex set, and $\sum_{i=1}^{m}m_i=M$. In the special case $Q_i=\{b_i\}$ with $b_i\in\mathbb{R}^{m_i}$, the split feasibility problem simplifies to the linear system $Ax=b$ with $A=[A_1;\ldots;A_m]\in\mathbb{R}^{M\times n}$ and $b=[b_1;\ldots;b_m]\in\mathbb{R}^{M}$.

As described in Section \ref{sec3.1}, the stochastic reformulation of (\ref{eq7:5}) is given by
\begin{equation}
\label{eq5:3}
\min_{x\in\mathbb{R}^{n}}
\psi(x),~s.t.~x\in \arg\min_{x\in\mathbb{R}^{n}} F(x)=\mathbb{E}[f_i(x)],
\end{equation}
	where $f_i(x)= \frac{1}{2}\operatorname{dist}(A_i x,Q_i)^2$. The gradient Lipschitz constant of $f_i$ is $L_i=\|A_i\|_2^2$ for $i\in [m]$.
Applying the SBBP method to solve (\ref{eq7:5}) yields
\begin{equation}
\label{eq6:17}
\left\{
\begin{aligned}
x_{k+1}^*&=x_k^*-t_k\sum_{i\in J_k} w_{i,k}A_i^T( A_i x_k -P_{Q_i}(A_ix_k)),\\
x_{k+1}&=\nabla\psi^{*}(x_{k+1}^*),
\end{aligned}
\right.
\end{equation}
where the stepsize $t_k$ may be chosen as the DecmSPS in (\ref{eq4:1}), \textsl{exact projective stepsize} computed by (\ref{eq4.6}), or \textsl{inexact projective stepsize} in (\ref{eq6.9}).

The algorithmic scheme (\ref{eq6:17}) can be viewed as a variant of the randomized Bregman projection (RBP) method in \cite{lorenz2014linearized,schopfer2019linear}. From the cut-and-project perspective, the main difference between the scheme (\ref{eq6:17}) and RBP lies in the construction of the separating halfspace.
In detail, let $|J_k|=1$ and $c_k=A_{i_k}x_k-P_{Q_{i_k}}(A_{i_k}x_k)$, then the separating halfspace of (\ref{eq6:17}) simplifies to
$$H_k=\left\{x\in \mathbb{R}^n\mid \langle A_{i_k}^Tc_k ,x\rangle
\leq \langle A_{i_k}^Tc_k ,x_k\rangle -\frac{\|A_{i_k}^Tc_k\|^2}{\|A_{i_k}\|_2^2}\right\},$$
while the RBP method employs	$$\hat{H}_k=\left\{x\in \mathbb{R}^n\mid \langle A_{i_k}^Tc_k ,x\rangle
\leq \langle A_{i_k}^Tc_k ,x_k\rangle -\|c_k\|^2\right\}.$$
Since $\|A_{i_k}^Tc_k\|^2\leq \|A_{i_k}\|^2_2\|c_k\|^2$, we have $\bar{X}\subset\hat{H}_k\subset H_k$ for any $A_{i_k}\in\mathbb{R}^{m_i\times n}$, and $\hat{H}_k= H_k$ for $m_i=1$ and $A_{i_k}^T\in\mathbb{R}^{n}$.
RBP constructs a geometrically tighter separating halfspace at each iteration.
Moreover, we endow the scheme (\ref{eq6:17}) with DecmSPS and the mini-batch technique.

For the scheme (\ref{eq6:17}), Theorem \ref{th4:1} and Theorem \ref{th3} establish sublinear convergence rates in expectation; by combining Theorem \ref{th4:2} with Theorem \ref{th5.5}, we obtain linear convergence of the scheme (\ref{eq6:17}).
The convergence analysis for RBP in \cite{schopfer2019linear} and for the induced scheme (\ref{eq6:17}) are carried out in different analytical settings.
The analysis of RBP is formulated under feasibility and regularity assumptions on the underlying sets, whereas our analysis is derived from the stochastic reformulation framework and the property of the inner objective $F$. This reformulation allows consistent and inconsistent CFPs to be treated within a unified framework, which is especially relevant in practice since whether a CFP is consistent is often unknown in advance. This difference in modeling and analysis is one of the main theoretical contributions of the paper.

\section{Numerical experiments}
\label{sec6}
In this section, we conduct numerical experiments to illustrate the impact of the stochastic block Bregman projection method versus the stochastic block projection method \cite{necoara2022stochastic} and the randomized Bregman projection method \cite{schopfer2019linear}.
Section \ref{sec6.1} investigates the performance of Algorithm \ref{al1} for solving linear inequalities.
In Section \ref{sec6.2}, Algorithm \ref{al1} is adapted to split feasibility problems by means of suitable choices of deviation functions. 
In each test, we compare three different stepsizes as follows:

\begin{enumerate}[(a)]	
	\item The DecmSPS is given by $t_k=\lambda_k \min\left\{ \mu\cdot\frac{\sum_{i\in J_k} w_{i,k} (f_i(x_k)-l_{J_k}^{*}) }{\|\sum_{i\in J_k} w_{i,k}\nabla f_i(x_k)\|^2},\frac{t_{k-1}}{\lambda_{k-1}}\right\}$ for $\lambda_k$ according to Corollary \ref{coro4:1}.	
	\item The \textsl{exact projective stepsize} as described in Theorem \ref{th4.1}.	
	\item The \textsl{block adaptive extrapolated stepsize}	$t_k=\frac{\mu}{L_{\max}L_{\operatorname{adapt}}^{k,\tau}}$ with 	$L_{\operatorname{adapt}}^{k,\tau}=\frac{\|\sum_{i\in J_k} w_{i,k}\nabla f_i(x_k)\|^2}{\sum_{i\in J_k} w_{i,k}\|\nabla f_i(x_k)\|^2}$ and $L_{\max}=\max_{i\in [m]}L_i$.\end{enumerate}

Throughout this section, we set $\lambda=1,|J_k|=\tau,w_{i,k}=\frac{1}{\tau}$ and the sampling distribution $\textbf{P}$ is chosen to be the uniform distribution, that is, $\textbf{P}(i)=1/m$ for all $i\in[m]$ and $p_i=\textbf{P}(i\in J_k)=\tau/m$. All algorithms are started from the initial vector $x_0=0\in\mathbb{R}^n$.

\subsection{Linear inequalities}
\label{sec6.1}

We evaluate the performance of Algorithm \ref{al1} on consistent linear inequalities of the form $Ax\le b$, following the formulation in Section \ref{sec5.1}. The measurement matrices $A\in\mathbb{R}^{m\times n}$ are generated using the MATLAB function \texttt{randn} and normalized by rows. The sparse ground-truth vector $\hat{x}\in\mathbb{R}^{n}$ is generated similarly with sparsity $s$, and we set $b=A\hat{x}$. In each test, we report the median of the relative error $\|x_k-\hat{x}\|/\|\hat{x}\|$ and the residual $\|\max\{Ax_k-b,0\}\|$ over 50 trials.

\subsubsection{The effect of parameters $c$ and $\gamma_b$}
\label{sec6.1.1}
We examine how the parameters $c>0$ and $\gamma_b>0$ affect the behavior of SBBP with DecmSPS on linear feasibility problems.
Following Corollary \ref{coro5} (a), we use $\lambda_k=1/c$ and set $t_{-1}=\gamma_b$ in the DecmSPS. Let $s=20$ and $\tau=50$.

In the first test, we fix $\gamma_b=100$ and vary the values of $c$ in $\{0.1,0.2,0.5,1\}$. As shown in Fig.~\ref{fig2}, all choices of $c$ lead to convergence, but smaller values of $c$ yield a faster decay in terms of both the relative error and residual. In particular, the curves for $c=0.1$ and $c=0.2$ decay faster than those for $c=0.5$ and $c=1$. A similar empirical trend was also observed in \cite{loizou2021stochastic}.

\begin{figure}[H]
	\centering
	\subfigure[$A\in\mathbb{R}^{400\times 100}$]{\includegraphics[width=0.23\linewidth,height=0.22\textwidth]{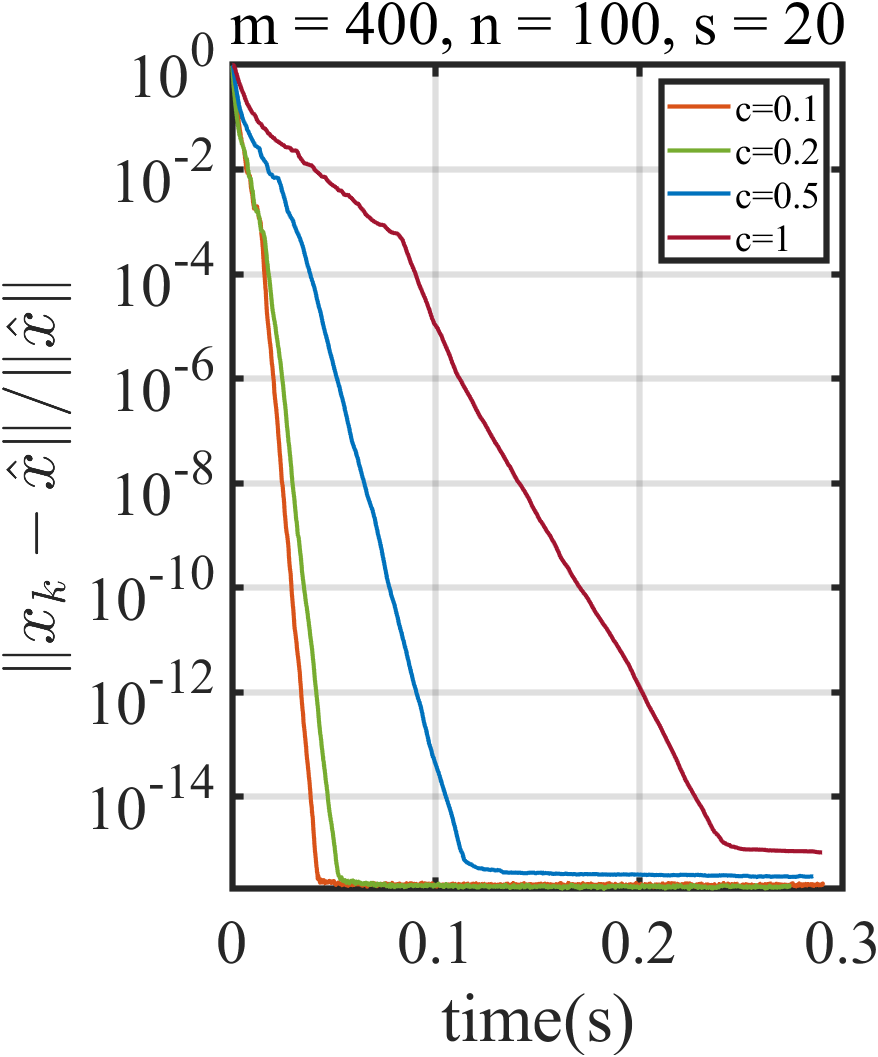}
		\includegraphics[width=0.23\linewidth,height=0.22\textwidth]{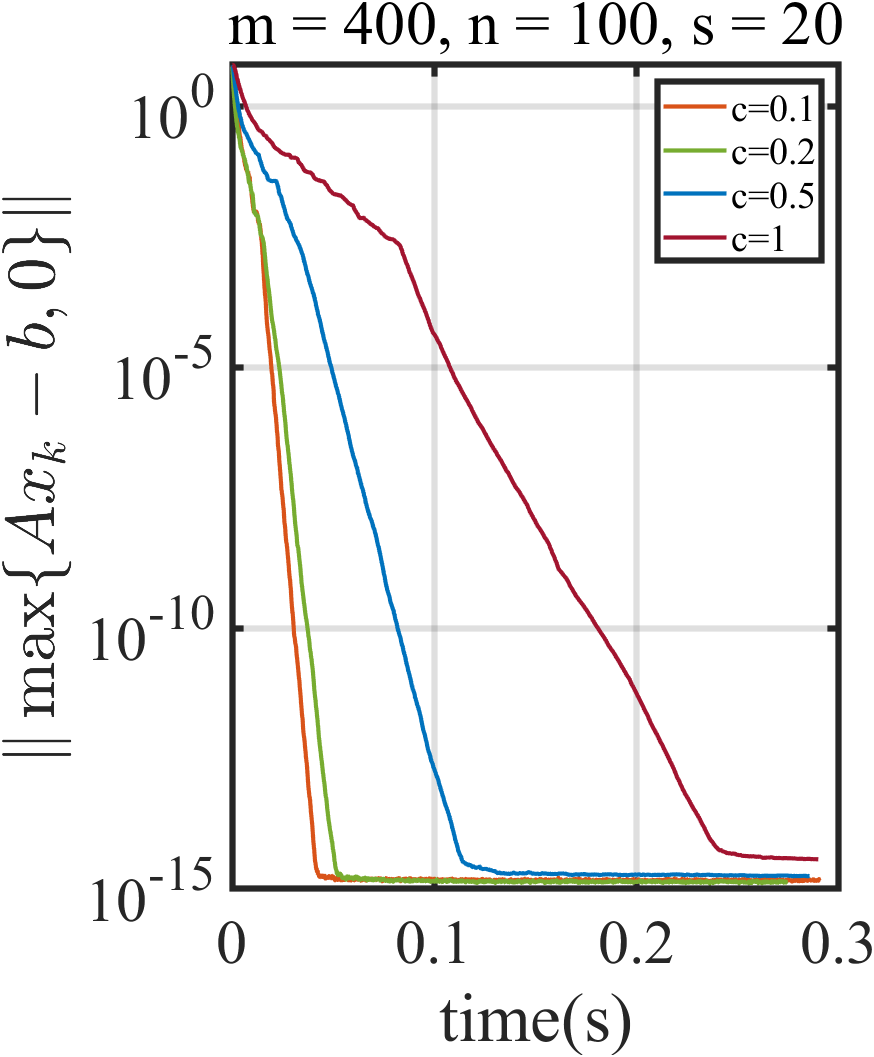}}
	\subfigure[$A\in\mathbb{R}^{800\times 200}$]{
		\includegraphics[width=0.23\linewidth,height=0.22\textwidth]{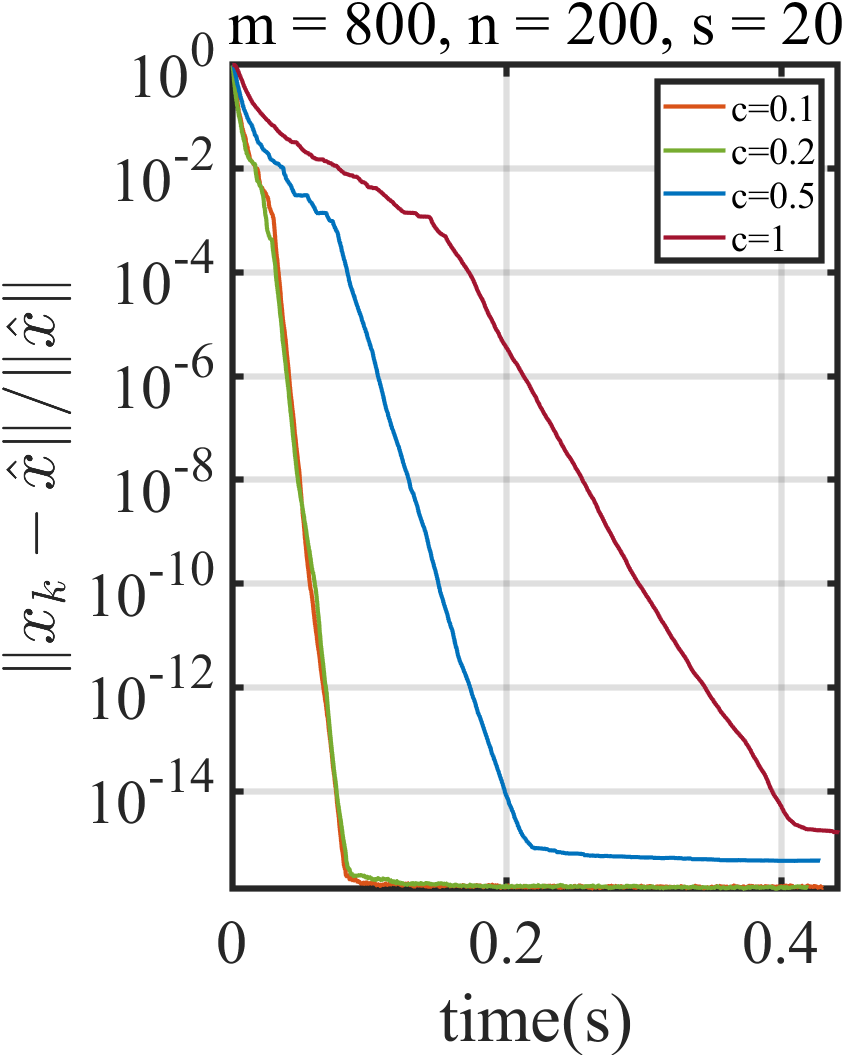}
		\includegraphics[width=0.23\linewidth,height=0.22\textwidth]{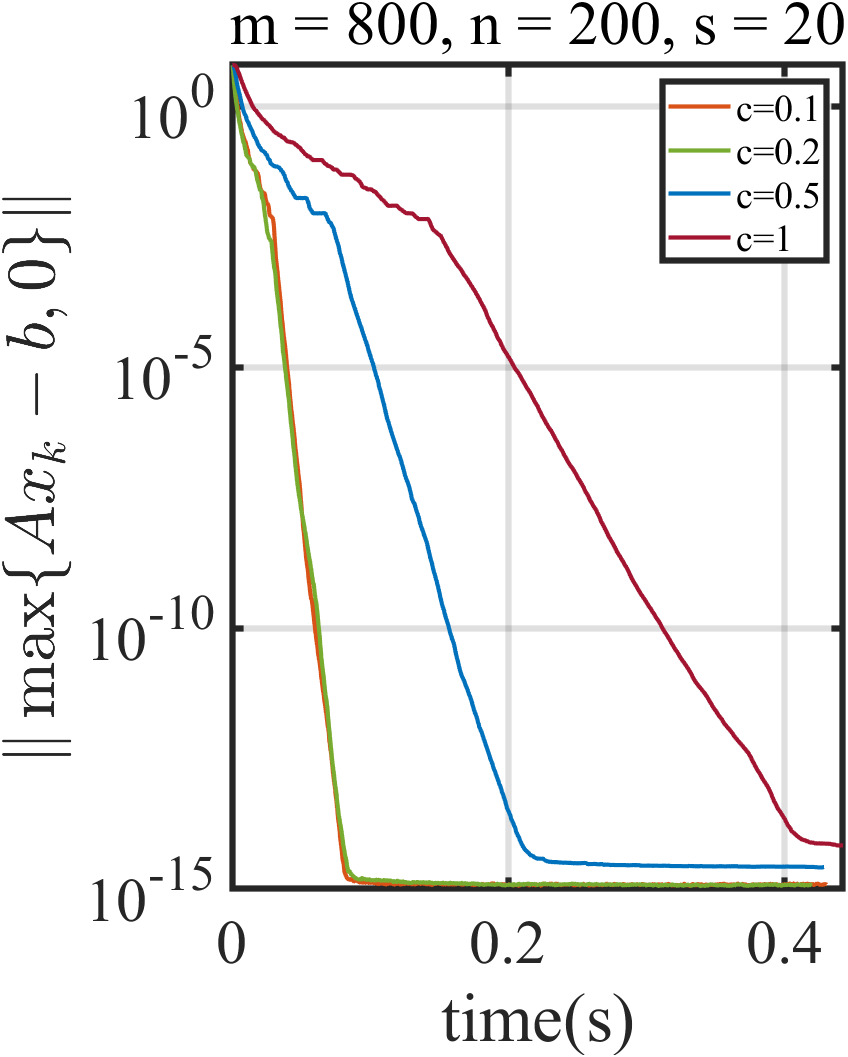}}
	\caption{The performance of SBBP with DecmSPS for $\gamma_b=100$ and varying $c\in\{0.1,0.2,0.5,1\}$}
	\label{fig2}
\end{figure}

In the second test, we fix $c=0.2$ and vary the values of $\gamma_b$ in $\{2,20,50,100\}$.
According to Fig.~\ref{fig3}, we find that larger values of $\gamma_b$ lead to faster convergence, and the convergence speed of $\gamma_b=50$ almost coincides with that of $\gamma_b=100$, while both are faster than those for $\gamma_b=2$ and $\gamma_b=20$.  
As discussed in Remark \ref{remark4.1}, if $\gamma_b$ is chosen too small, the \textsl{decreasing mirror stochastic Polyak stepsize} is truncated and degenerates to a constant stepsize. Once $\gamma_b$ is large enough to activate the stochastic Polyak stepsize, further increasing $\gamma_b$ has essentially no additional benefit. 

\begin{figure}[H]
	\centering
	\subfigure[$A\in\mathbb{R}^{400\times 100}$]{\includegraphics[width=0.23\linewidth,height=0.22\textwidth]{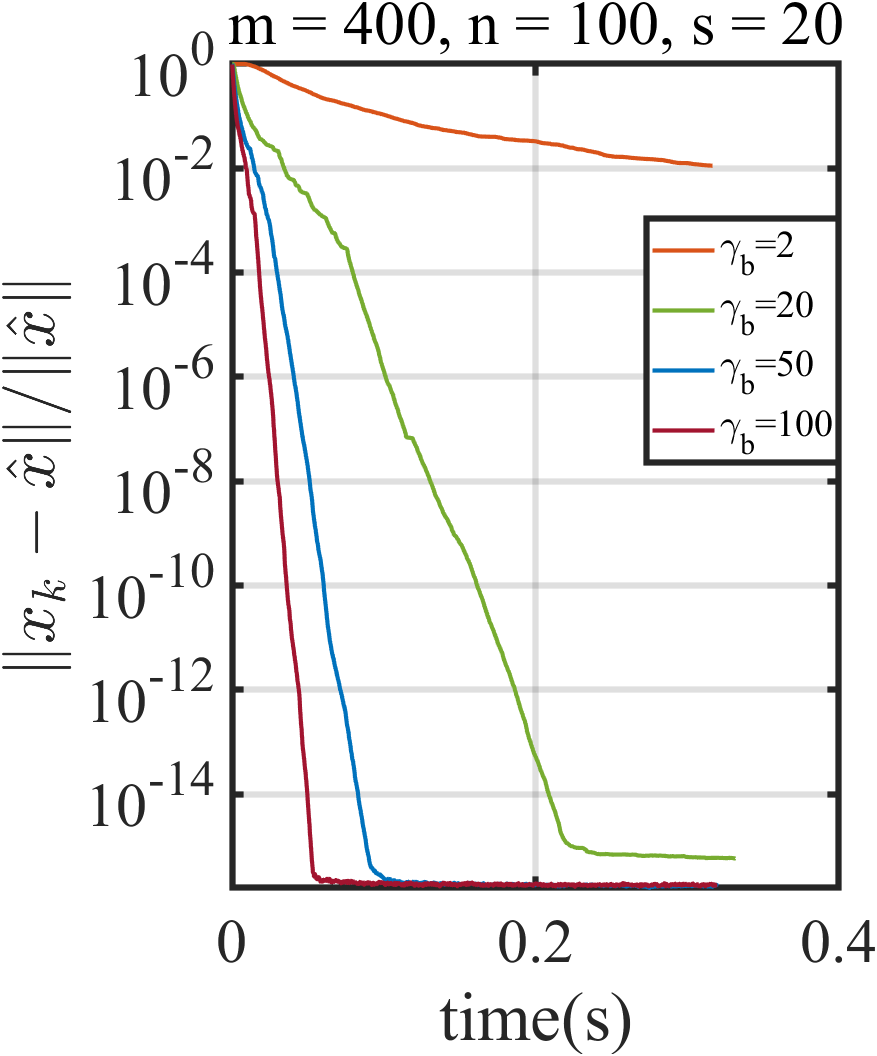}
		\includegraphics[width=0.23\linewidth,height=0.22\textwidth]{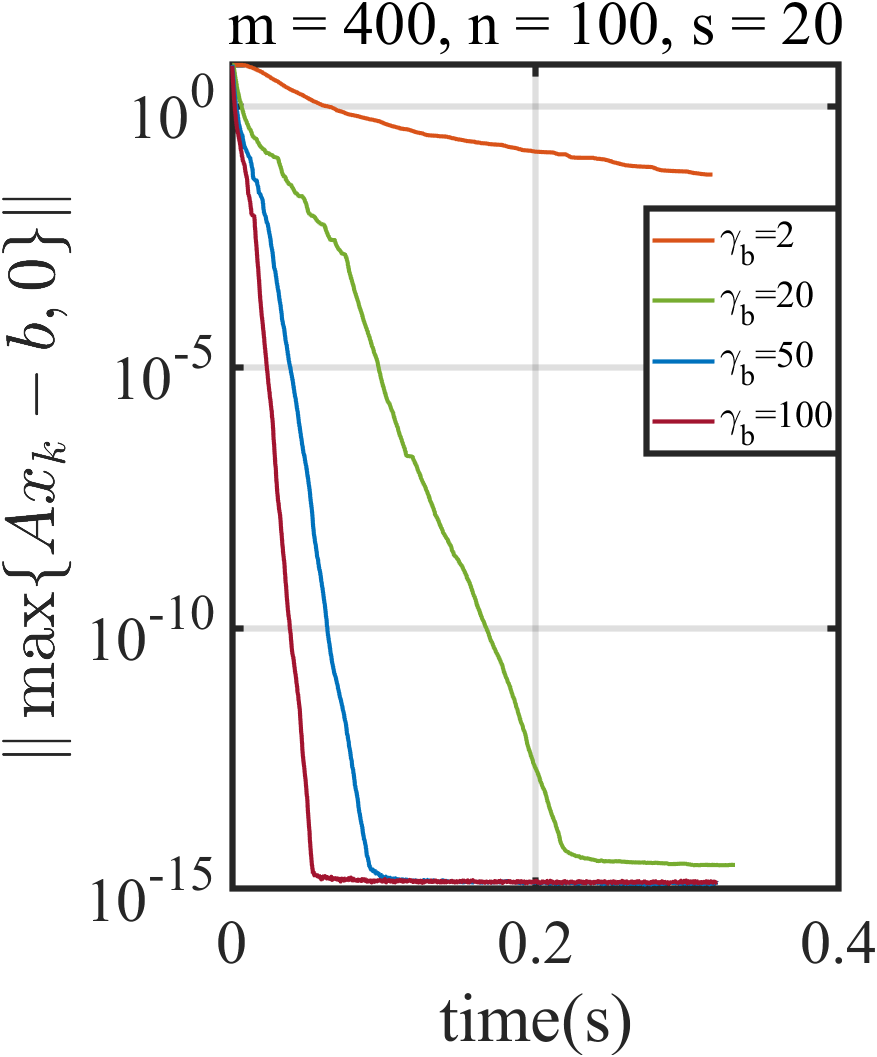}}
	\subfigure[$A\in\mathbb{R}^{800\times 200}$]{
		\includegraphics[width=0.23\linewidth,height=0.22\textwidth]{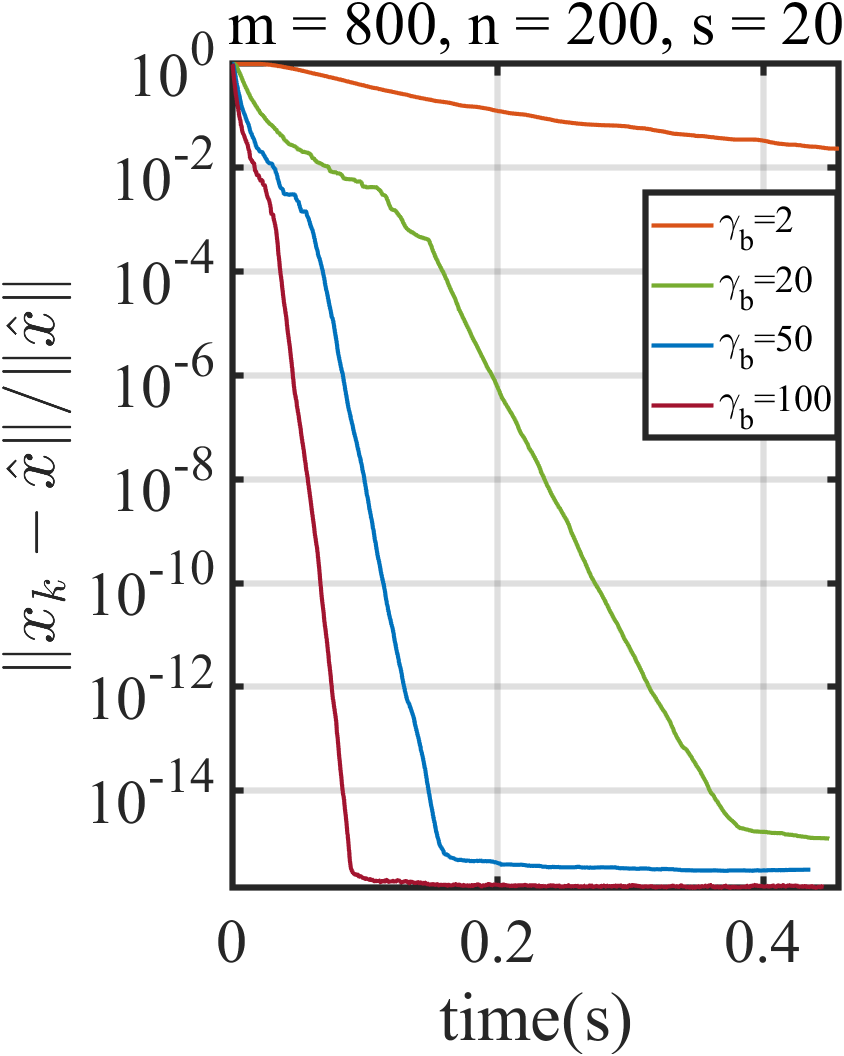}
		\includegraphics[width=0.23\linewidth,height=0.22\textwidth]{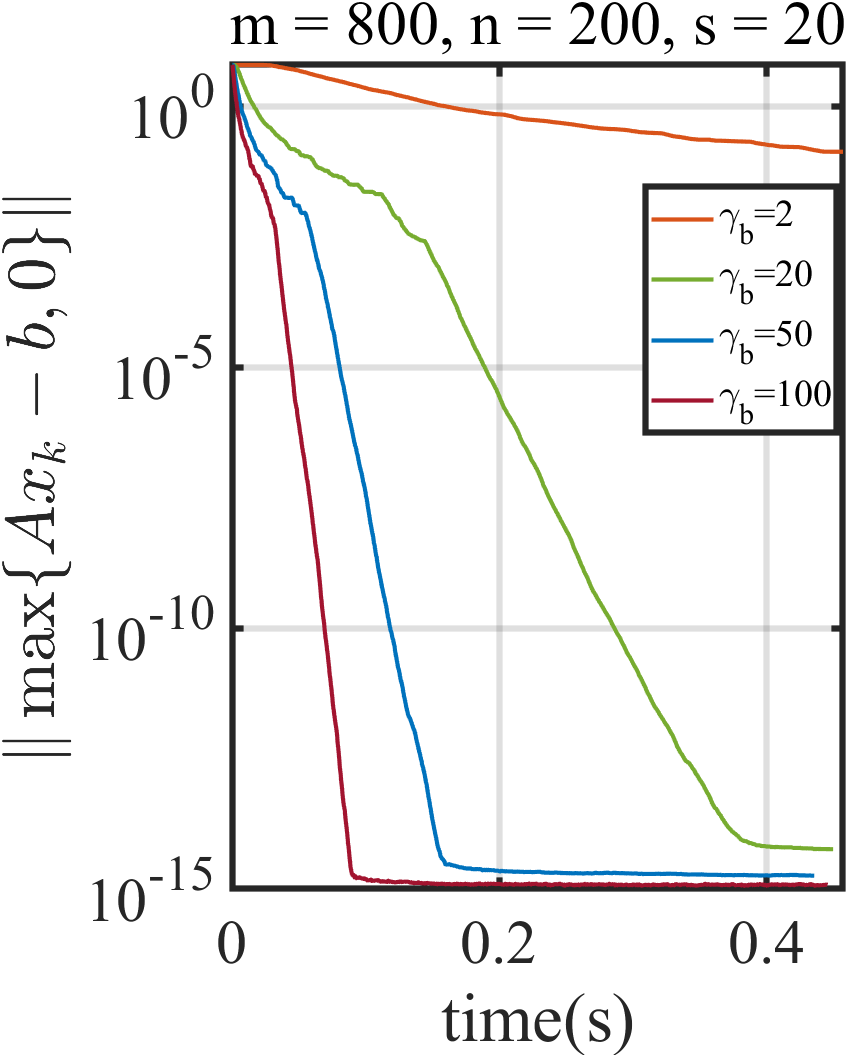}}
	\caption{The performance of SBBP with DecmSPS under $c=0.2$ and $\gamma_b\in\{2,20,50,100\}$}
	\label{fig3}
\end{figure}

\subsubsection{The effect of block size $\tau$}
\label{sec6.1.2}
Here we explore the influence of the block size $\tau$ on Algorithm \ref{al1}; in particular, we test the effectiveness of the optimal block size $\tau^*$ from Corollary \ref{coro5..2}.
For each of the three stepsizes, we run (\ref{eq6:1}) with different block sizes $\tau\in\{20,100,200\}$ together with the optimal block size $\tau^*$. Let $c=0.5$ and $\gamma_b=100$ in DecmSPS.

As shown in Fig.~\ref{fig1}, we observe that the curves for $\tau^*$ exhibit the fastest decay in terms of both relative error and residual, while both smaller and larger block sizes lead to slower convergence. The advantage of $\tau^*$ is especially pronounced for the \textsl{exact projective stepsize}, while for the \textsl{block adaptive extrapolated stepsize} the improvement is more moderate. These results provide numerical support for Corollary \ref{coro5..2} and suggests that the optimal block size $\tau^*$ is effective and feasible in practice.

\begin{figure}[H]
	\centering
	\subfigure[\textsl{decreasing mirror stochastic Polyak stepsize}]{\includegraphics[width=0.23\linewidth,height=0.22\textwidth]{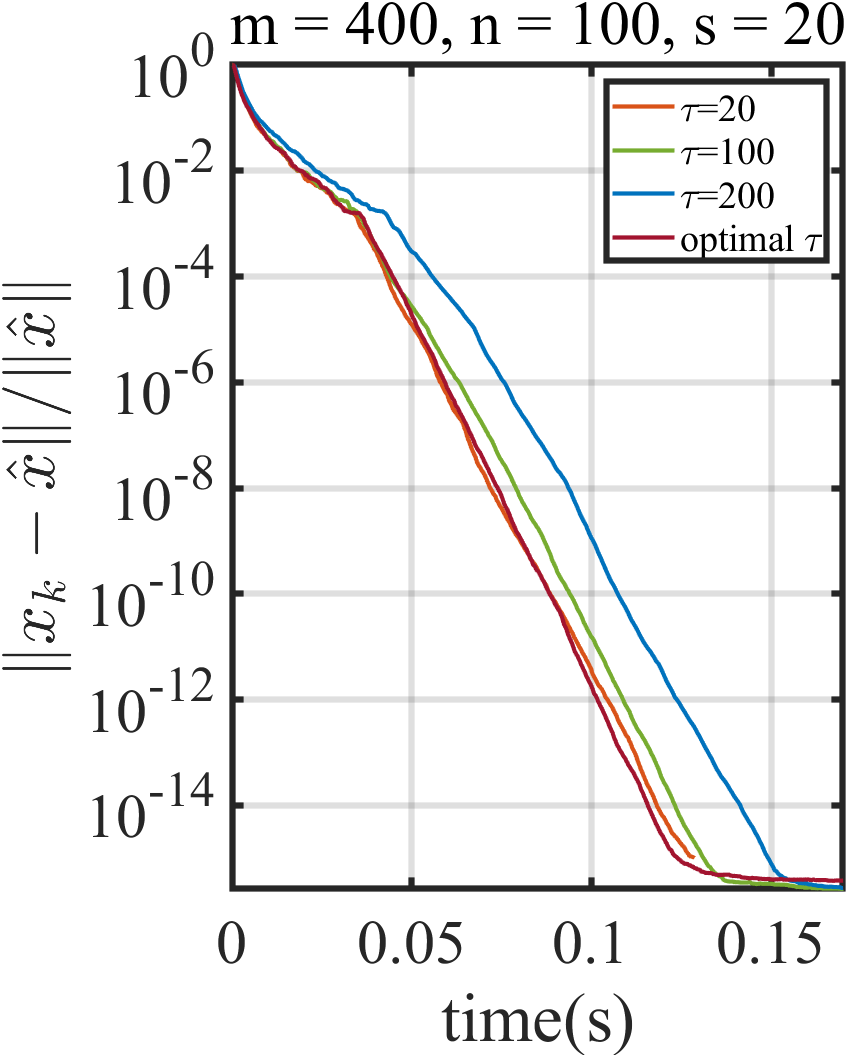}
		\includegraphics[width=0.23\linewidth,height=0.22\textwidth]{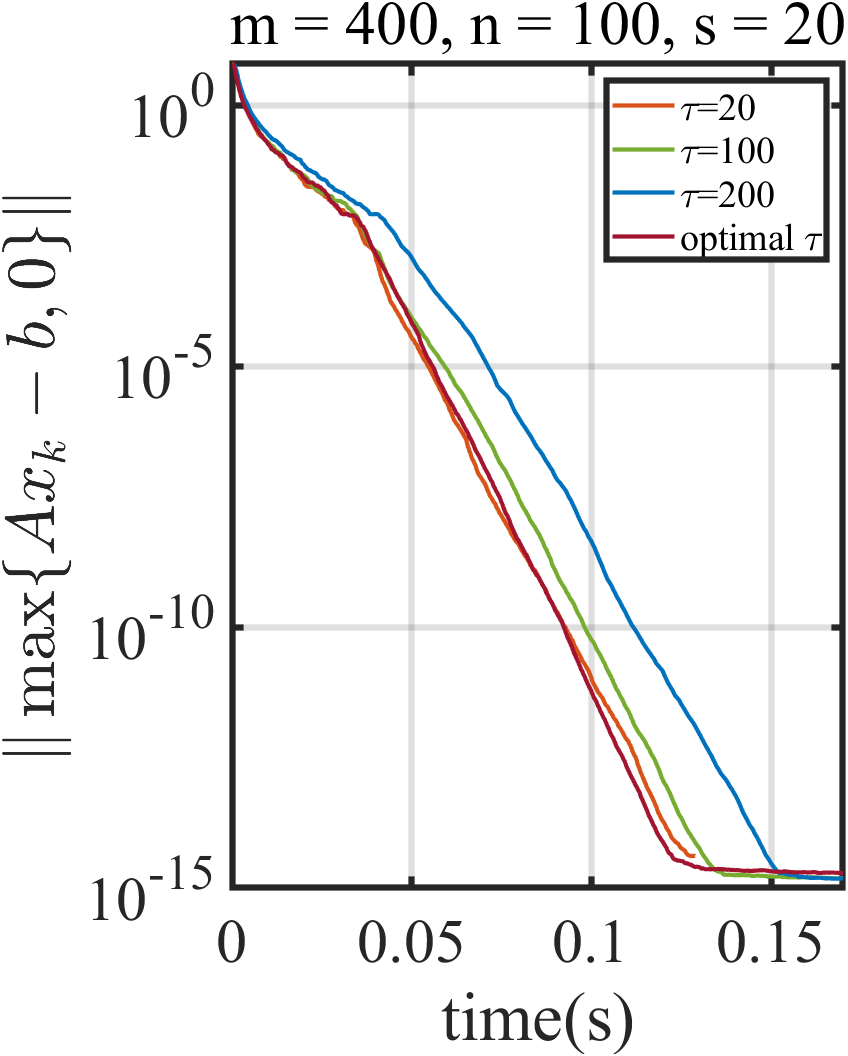}}
	\subfigure[\textsl{block adaptive extrapolated stepsize}]{
		\includegraphics[width=0.23\linewidth,height=0.22\textwidth]{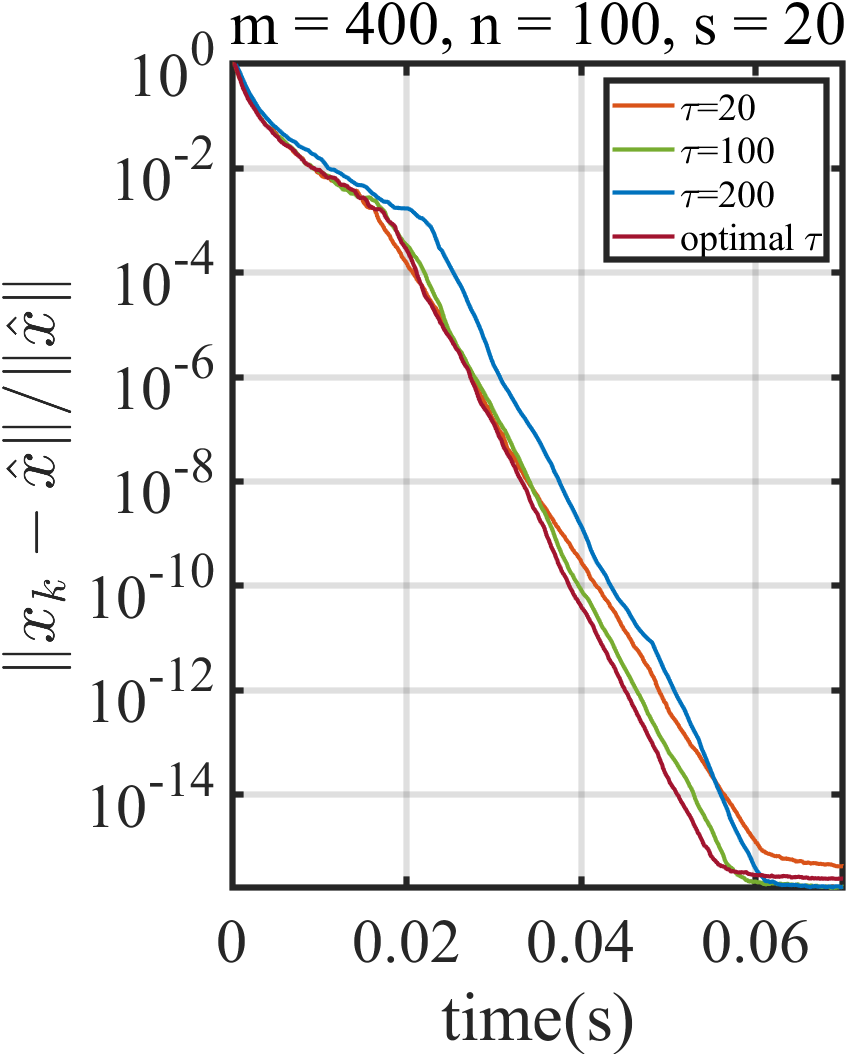}
		\includegraphics[width=0.23\linewidth,height=0.22\textwidth]{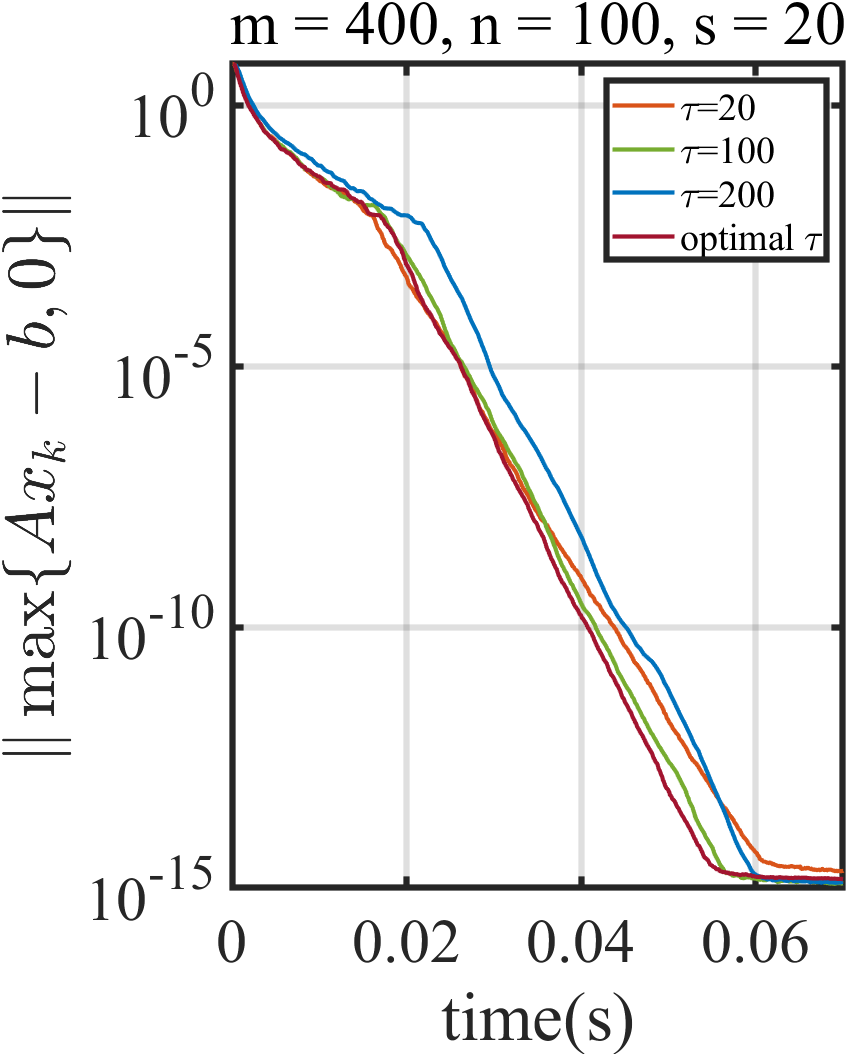}}
	\subfigure[\textsl{exact projective stepsize}]{
		\includegraphics[width=0.23\linewidth,height=0.22\textwidth]{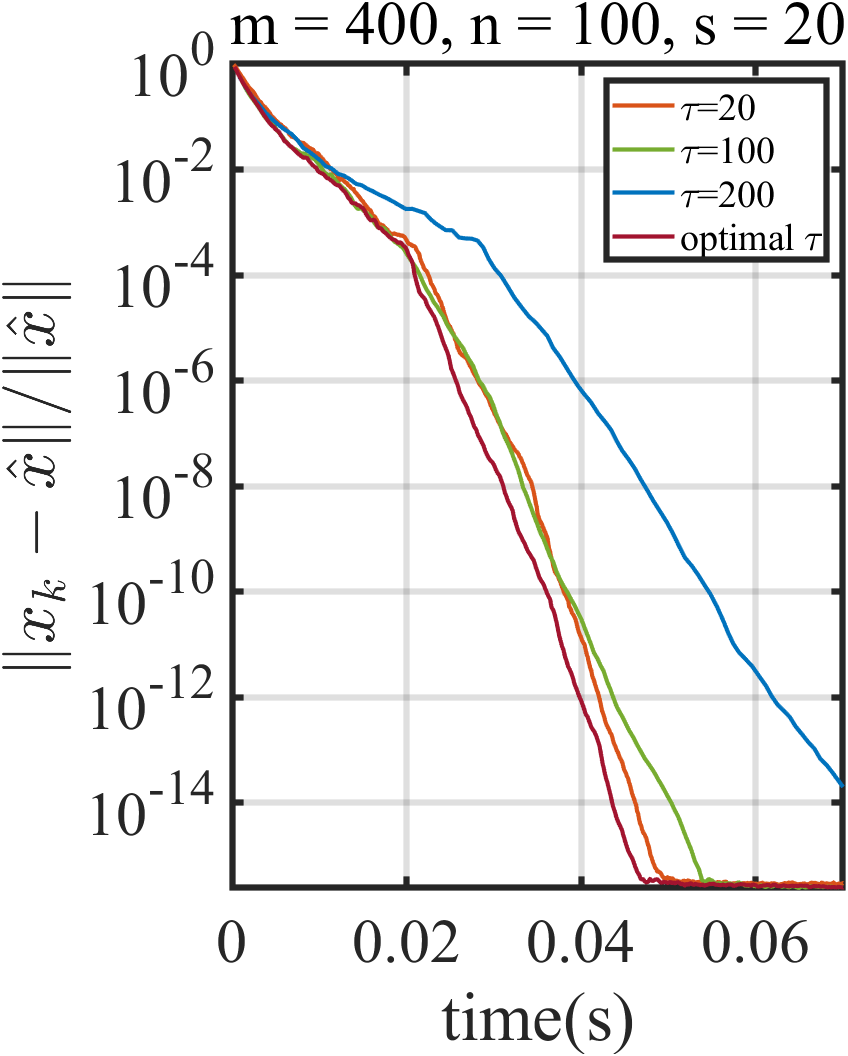}
		\includegraphics[width=0.23\linewidth,height=0.22\textwidth]{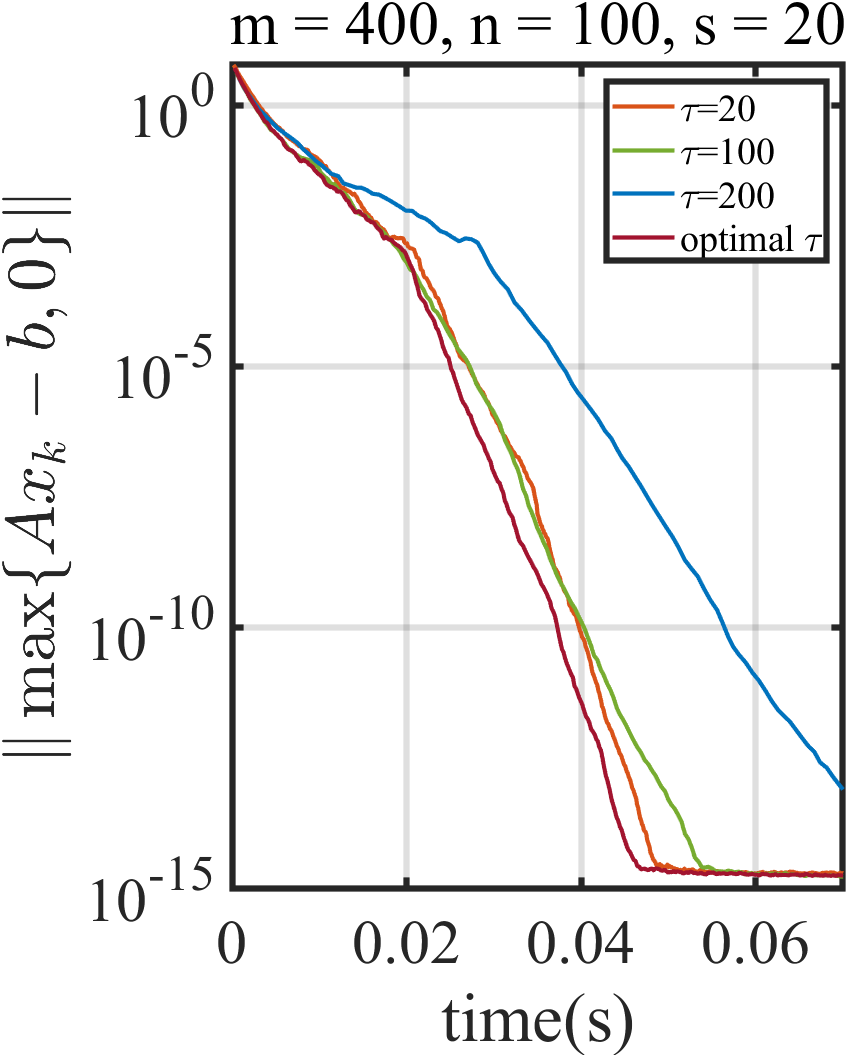}}
	\caption{The performance of SBBP with different stepsizes under $\gamma_b=100$ and $\tau\in\{20,100,200,\tau^*\}$.}
	\label{fig1}
\end{figure}

\subsubsection{The advantages of Bregman projections}

In this experiment, we set $\tau=20$ and $s=20$ and compare SBP ($\lambda=0$) in \cite{necoara2022stochastic} with SBBP ($\lambda>0$) equipped with three stepsizes. As suggested in Section \ref{sec6.1.1}, we set $c=0.2$ and $\gamma_b=100$.

Table \ref{table2} records the median number of iterations (IT) and computing time (CPU) of SBP and SBBP with three stepsizes over 50 trials. The iteration is terminated once the relative error is less than $\varepsilon=10^{-10}$ or the maximum number of iterations reaches $20000$.
To make the results more intuitive, Fig. \ref{fig6} plots the convergence behavior of SBP and SBBP. We find that SBBP with exact projective stepsize yields the fastest decrease in terms of both relative error and the residual, followed by SBBP with DecmSPS, while SBBP with \textsl{block adaptive extrapolated stepsize} is slower but still ahead of SBP with block stepsize.

\begin{table}[h]
			\caption{IT and CPU of different methods for random matrices when relative error $<10^{-10}$ (median IT and CPU over 50 trials; ``-'' indicates that the stopping criterion is not met within the iteration budget).}
			\label{table2}
			\begin{tabular*}{\textwidth}{@{\extracolsep{\fill}}lccccc@{\extracolsep{\fill}}}
				\toprule
				\multicolumn{2}{c}{ $m\times n$ }  
				&  $500\times 200$ & $1000\times 200$ & $1000\times 500$ & $2000\times 200$ 
				\\	
				\midrule
				\multirow{2}{*}{ SBP with block stepsize \cite{necoara2022stochastic}} & IT  &  -  & 3231 & - & 1601    
				\\
				\cmidrule{2-6}
				& CPU    & -  &  0.2670  & - & 0.1665  
				\\
				\midrule
				\multirow{2}{*}{ SBBP with block stepsize } & IT  & 2168 &  1151   & 3503 & 1167      
				\\
				\cmidrule{2-6}
				& CPU & 0.1636  & 0.0939 & 0.5121  & 0.1247      
				\\
				\midrule
				\multirow{2}{*}{SBBP with exact stepsize } & IT   & \textbf{771}  & \textbf{252} & \textbf{670} & \textbf{240}     
				\\
				\cmidrule{2-6}
				& CPU   & \textbf{0.0818}  &  \textbf{0.0300} & \textbf{0.1829} & \textbf{0.0358}      
				\\
				\midrule
				\multirow{2}{*}{SBBP with DecmSPS} & IT   & 1249  & 633 & 1647 & 607      
				\\
				\cmidrule{2-6}
				& CPU & 0.1171  & 0.0644 & 0.2583  & 0.0747      
				\\
				\botrule
			\end{tabular*}
\end{table}

\begin{figure}[H]
	\centering
	\subfigure[$A\in\mathbb{R}^{500\times 200}$]{
		\includegraphics[width=0.23\linewidth,height=0.22\textwidth]{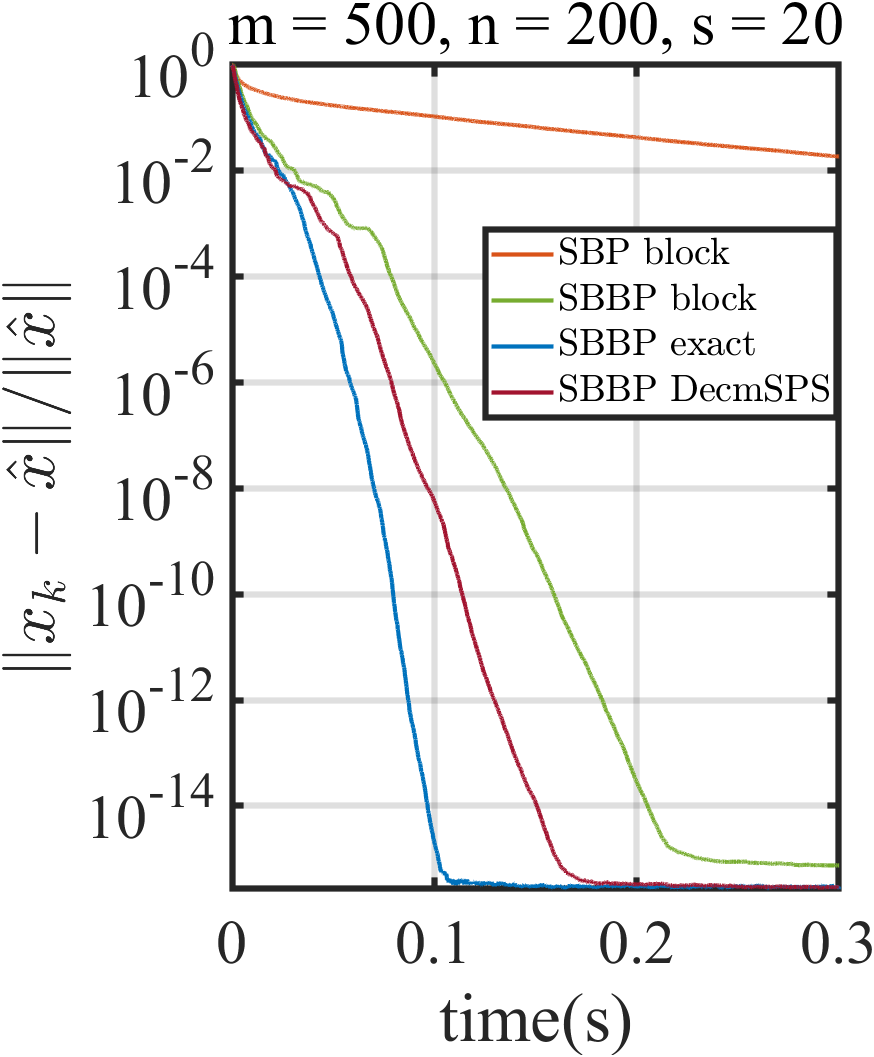}
		\includegraphics[width=0.23\linewidth,height=0.22\textwidth]{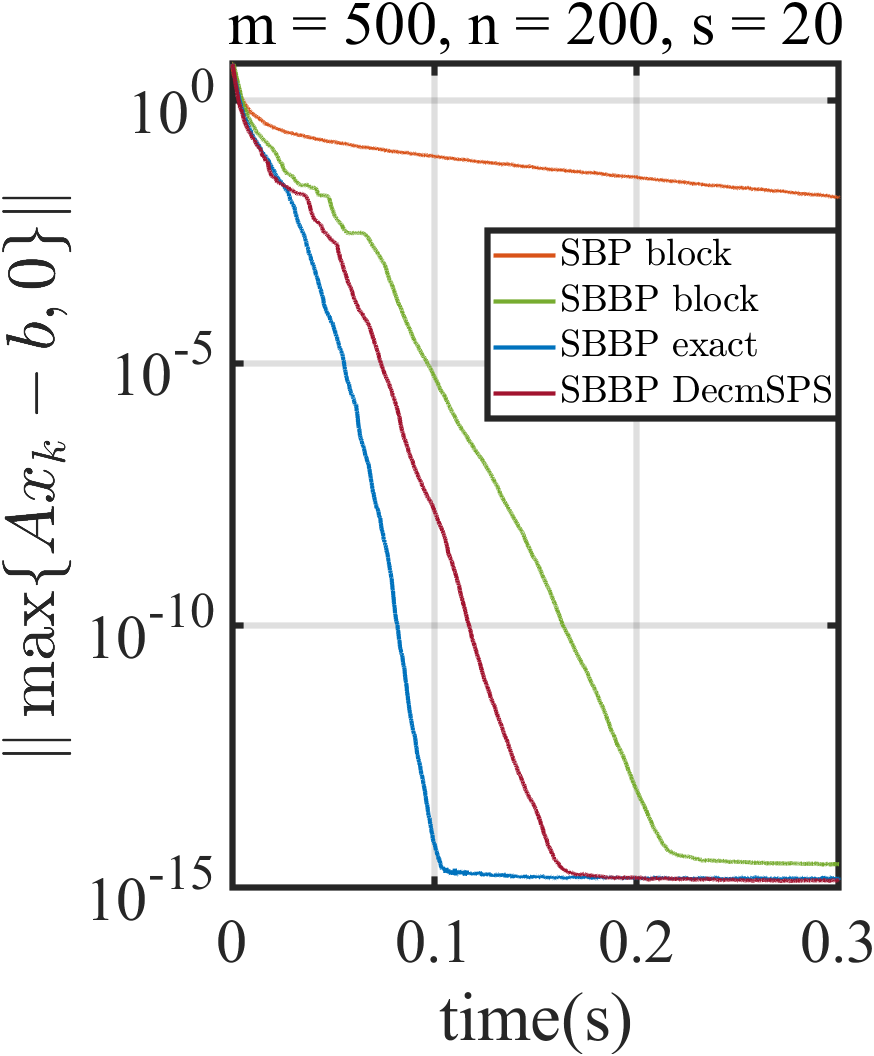}}
	\subfigure[$A\in\mathbb{R}^{1000\times 200}$]{
		\includegraphics[width=0.23\linewidth,height=0.22\textwidth]{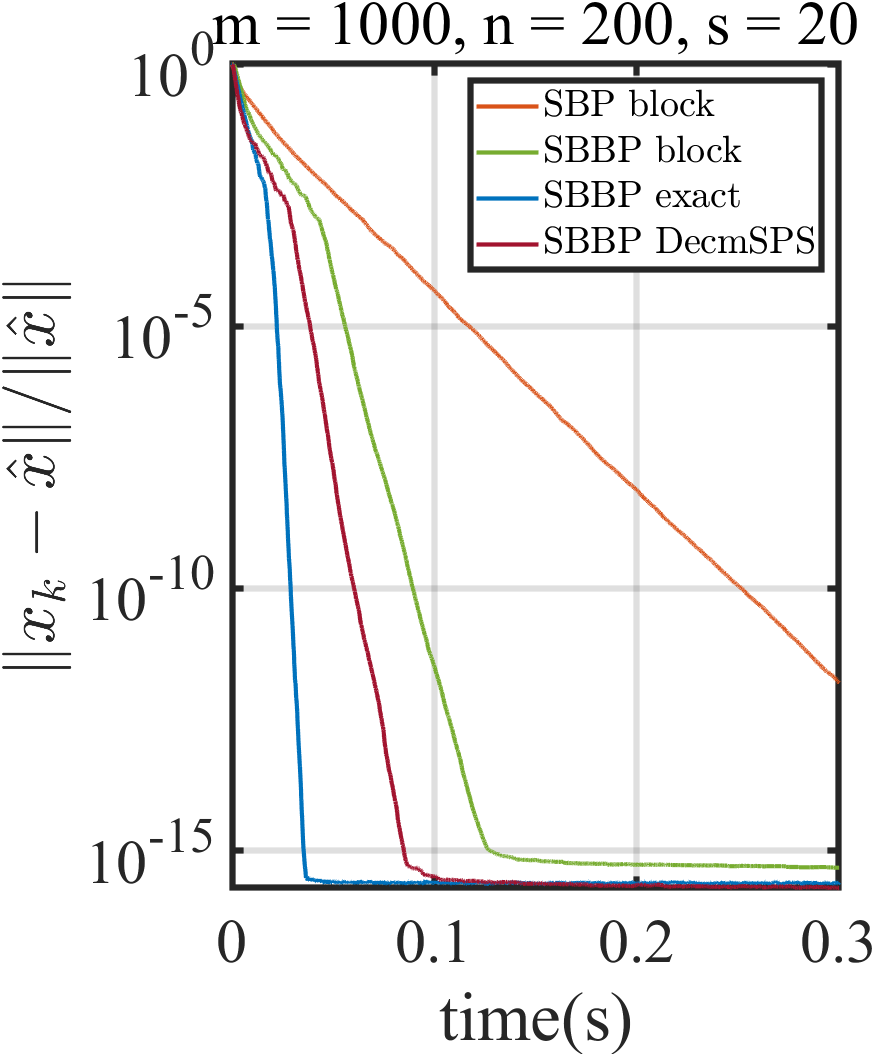}
		\includegraphics[width=0.23\linewidth,height=0.22\textwidth]{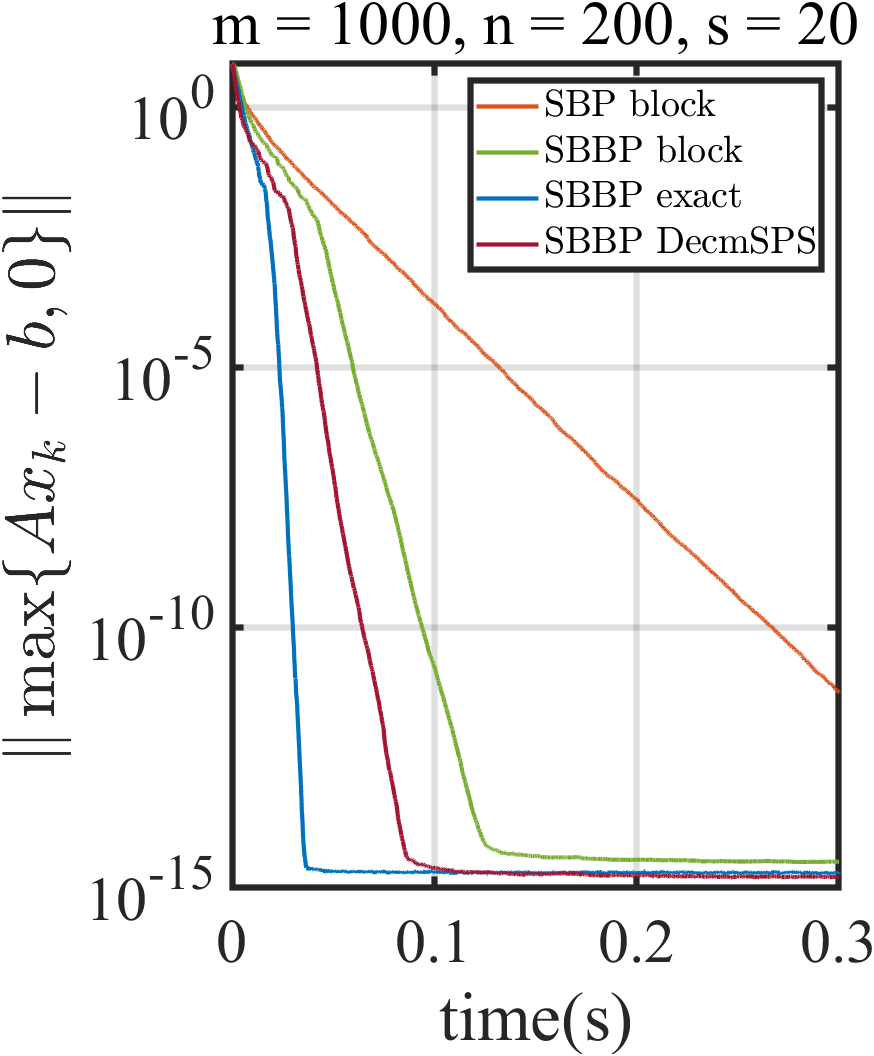}}
	\subfigure[$A\in\mathbb{R}^{1000\times 500}$]{
		\includegraphics[width=0.23\linewidth,height=0.22\textwidth]{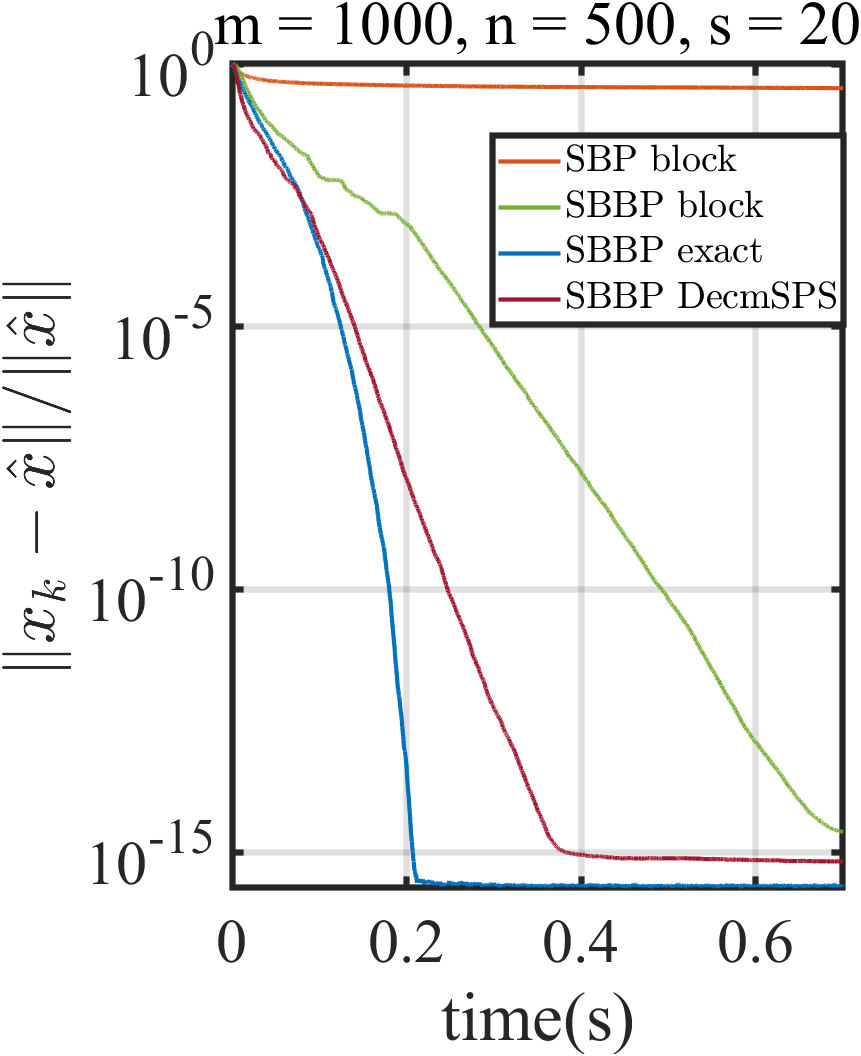}
		\includegraphics[width=0.23\linewidth,height=0.22\textwidth]{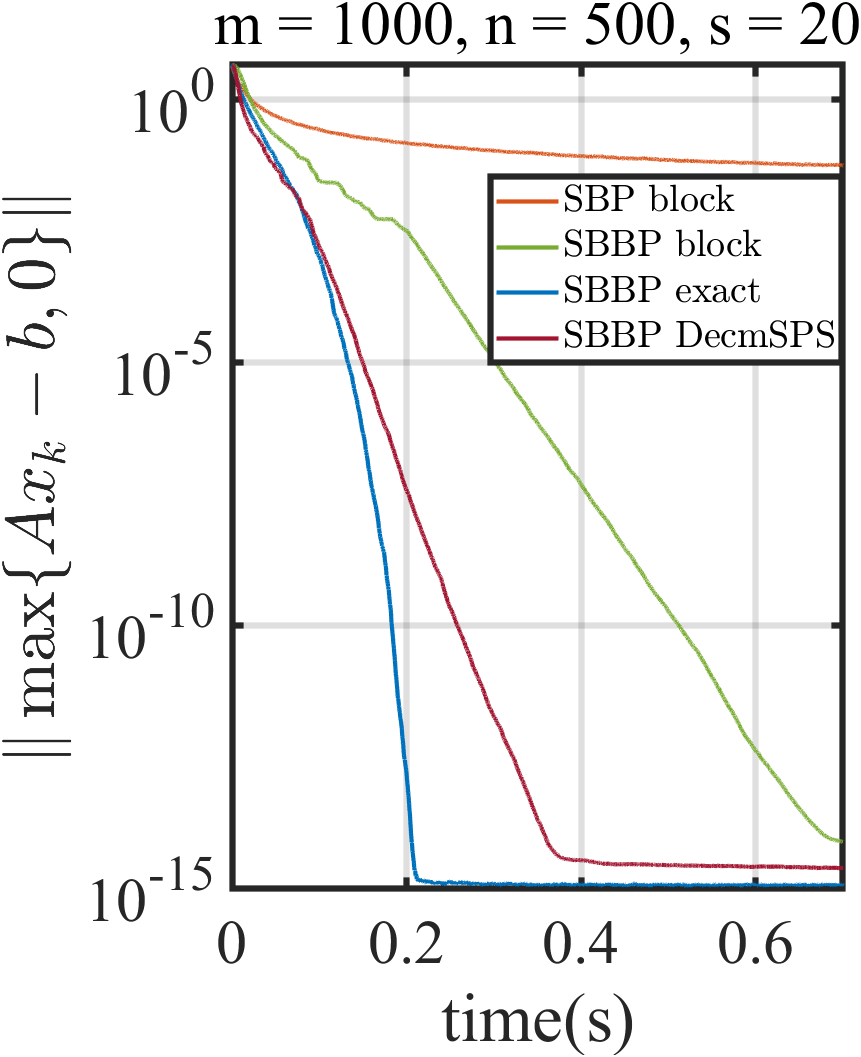}}
	\subfigure[$A\in\mathbb{R}^{2000\times 200}$]{
		\includegraphics[width=0.23\linewidth,height=0.22\textwidth]{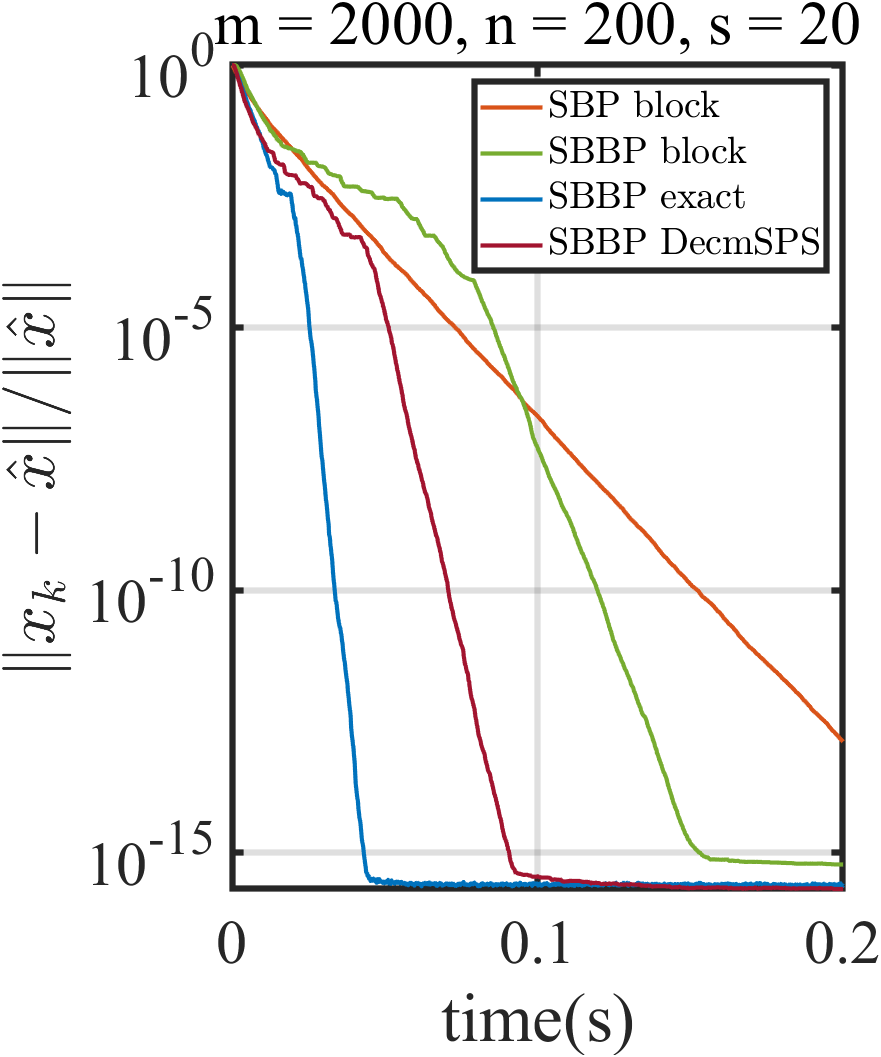}
		\includegraphics[width=0.23\linewidth,height=0.22\textwidth]{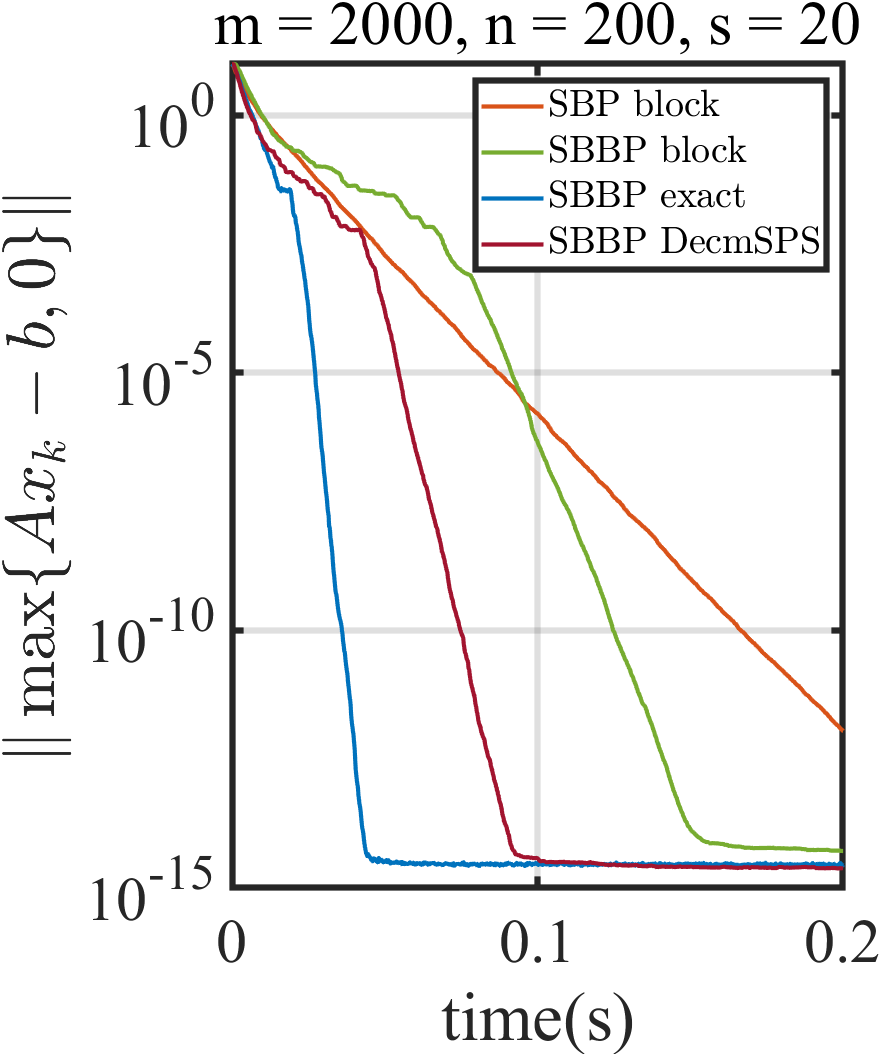}}
	\caption{The convergence of SBP and SBBP with different stepsizes for several matrix sizes}
	\label{fig6}
\end{figure}

\subsection{Split feasibility problems}
\label{sec6.2}

In this subsection, we consider sparse recovery from linear measurements corrupted by Gaussian noise and compare SBBP with the randomized Bregman projection method \cite{schopfer2019linear}.
Recall in Section \ref{sec5.2}, we use deviation functions $f_i(x)=\frac{1}{2}\|A_ix-P_{Q_i}(A_ix)\|_2^2$ and take a constant block size $m_i=\xi$, where $Q_i\subset\mathbb{R}^{\xi}$ are closed convex sets and $A_i\in\mathbb{R}^{\xi\times n}$ are row blocks of $A\in\mathbb{R}^{M\times n}$.
The matrix $A$ and the nonzero entries of an $s$-sparse ground truth $\hat{x}$ are drawn i.i.d. from $\mathcal{N}(0,1)$; the support of $\hat{x}$ is chosen uniformly at random, and $b=A\hat{x}$.
We partition $(A,b)$ into $m=M/\xi$ consecutive row blocks $(A_i,b_i)$.
For a noise level $\sigma>0$, we construct consistent and inconsistent instances as follows.
 



\textbf{Consistent SFPs.} Let $\varepsilon\sim\mathcal{N}(0,I_{M})$ and set $b^{\sigma}=b+\sigma\cdot\operatorname{sign}(\varepsilon)$.
For each block $i\in[m]$, let $b_i^{\sigma}\in\mathbb{R}^{\xi}$ be the corresponding segment of $b^{\sigma}$ and define $Q_i=\{y\in\mathbb{R}^{\xi}\mid \|y-b_i^{\sigma}\|_2\leq r\}$ with $r=\sqrt{\xi}\sigma$.
Then $\hat{x}$ is feasible and satisfies $\|A_i\hat{x}-b_i^{\sigma}\|_2=r$ for all $i\in[m]$.

\textbf{Inconsistent SFPs.} Choose a unit vector $v\in\mathcal{N}(A^T)$ and set $b^{\sigma}=b+\sigma\sqrt{M}\,v$.
Define $Q_i=\{y\in\mathbb{R}^{\xi}\mid \|y-b_i^{\sigma}\|_2\leq r\}$ with $r=\sqrt{\xi}\sigma/2$.
Since $\operatorname{dist}(b^{\sigma},\mathcal{R}(A))=\sigma\sqrt{M}>\sqrt{m}\,r$, the feasibility system is inconsistent.

Here SBBP reduces to (\ref{eq6:17}); we let $\psi(x)=\lambda\|x\|_1+\frac{1}{2}\|x\|_2^2$.
The projection onto $Q_i$ is explicit and
$$
A_ix_k-P_{Q_i}(A_ix_k)= \max\left\{0,1-\frac{r}{\|A_i x_k-b_i^{\sigma}\|_2} \right\}\cdot (A_i x_k-b_i^{\sigma}).
$$
We set $M=800,m=80,n=200,\xi=10,\tau=20$, and $s=20$, and test noise levels $\sigma\in\{0.01,0.1\}$.
We take $c=0.2,\gamma_b=100$ for the consistent case and $c=0.1,\gamma_b=100$ for the inconsistent case.
We report the median over 50 trials of the relative error $\|x_k-\hat{x}\|_2/\|\hat{x}\|_2$ and the relative residual $\|Ax_k-b\|_2/\|b\|_2$.

The recovery results of SBBP and RBP are summarized in Fig. \ref{fig8} for $\sigma=0.01$ and in Fig. \ref{fig9} for $\sigma=0.1$.
As the noise level increases, the recovery performance deteriorates for all methods; the degradation is much more pronounced for RBP in the consistent case.
As illustrated in Fig. \ref{fig8} (a) and Fig. \ref{fig9} (a), for consistent problems, SBBP with the exact projective stepsize performs best, followed by SBBP with DecmSPS. Notably, as the noise level $\sigma$ increases, the convergence rate of RBP with the exact stepsize deteriorates significantly.
This robustness highlights the advantage of the mini-batch strategy in mitigating the impact of noise.
Regarding the inconsistent problems shown in Fig. \ref{fig8} (b) and Fig. \ref{fig9} (b), SBBP with DecmSPS exhibits the superior performance.
Overall, the RBP method is more sensitive to noise than SBBP in both consistent and inconsistent cases, whereas SBBP with all three stepsizes remains robust.
In particular, SBBP with DecmSPS demonstrates the most effective balance between convergence speed and noise robustness in challenging inconsistent feasibility problems.


\begin{figure}[H]
	\centering
	\subfigure[consistent]{
		\includegraphics[width=0.23\linewidth,height=0.22\textwidth]{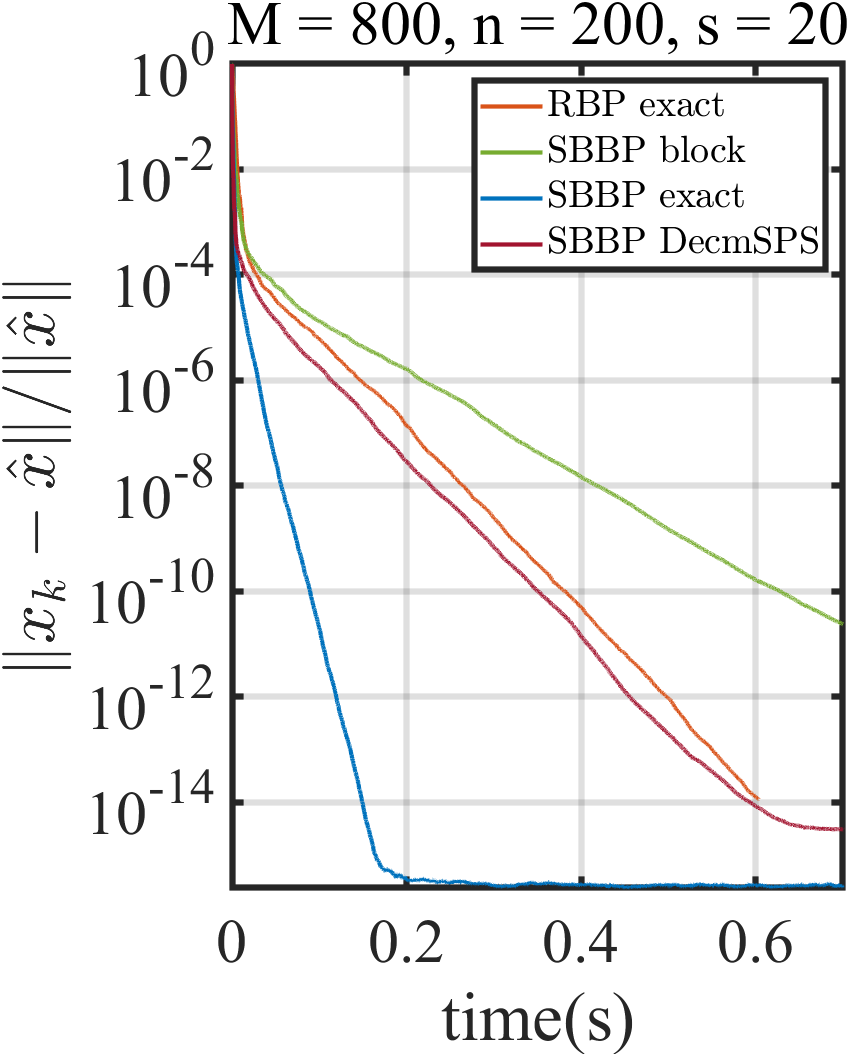}
        \includegraphics[width=0.23\linewidth,height=0.22\textwidth]{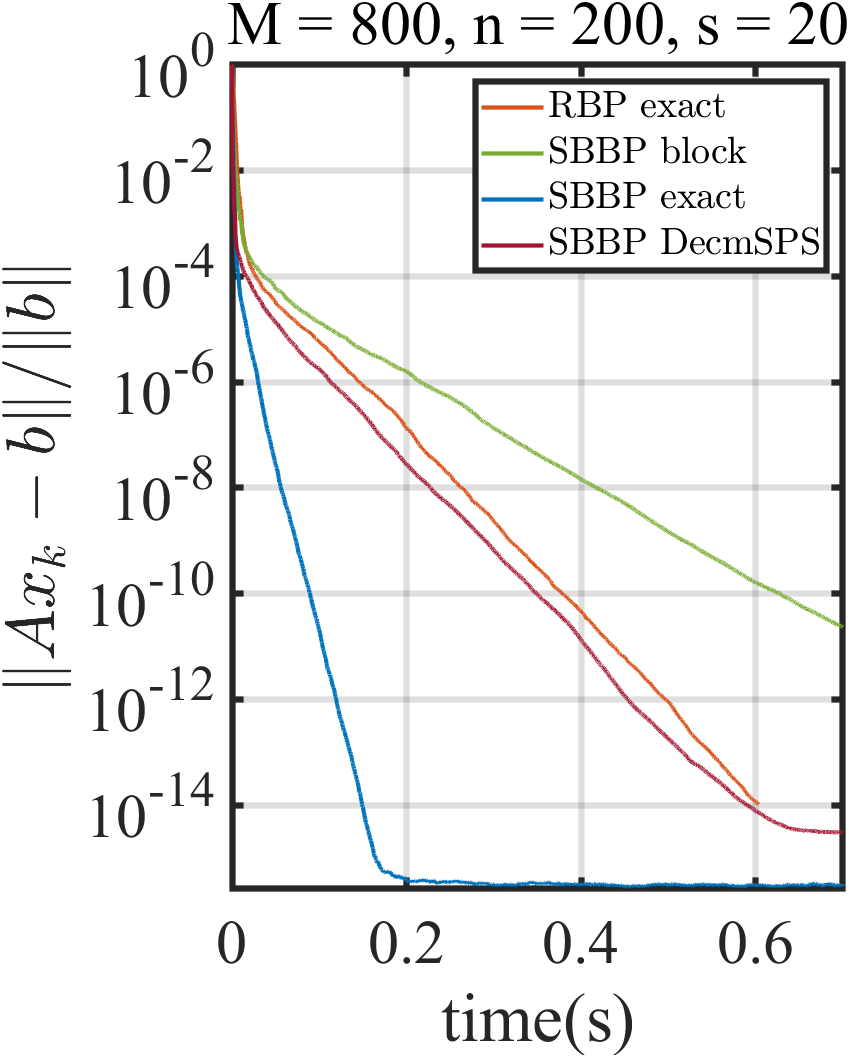}}
	\subfigure[inconsistent]{
		\includegraphics[width=0.23\linewidth,height=0.22\textwidth]{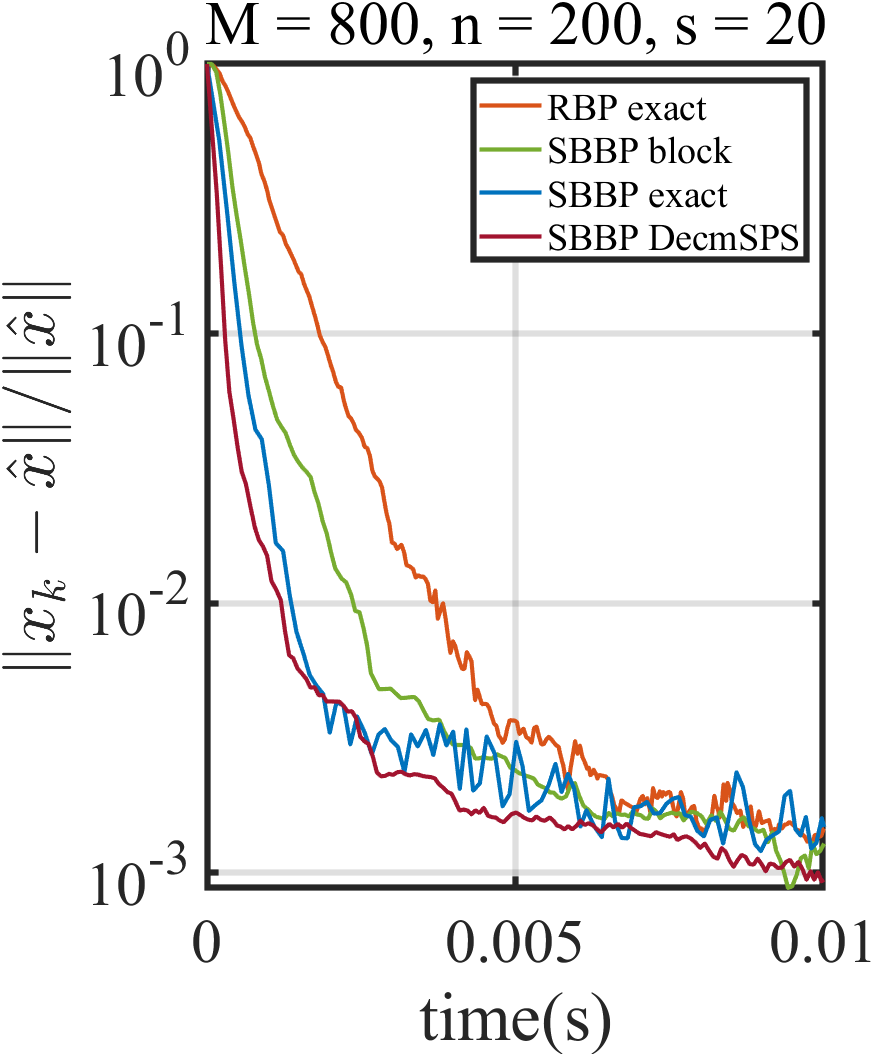}
		\includegraphics[width=0.23\linewidth,height=0.22\textwidth]{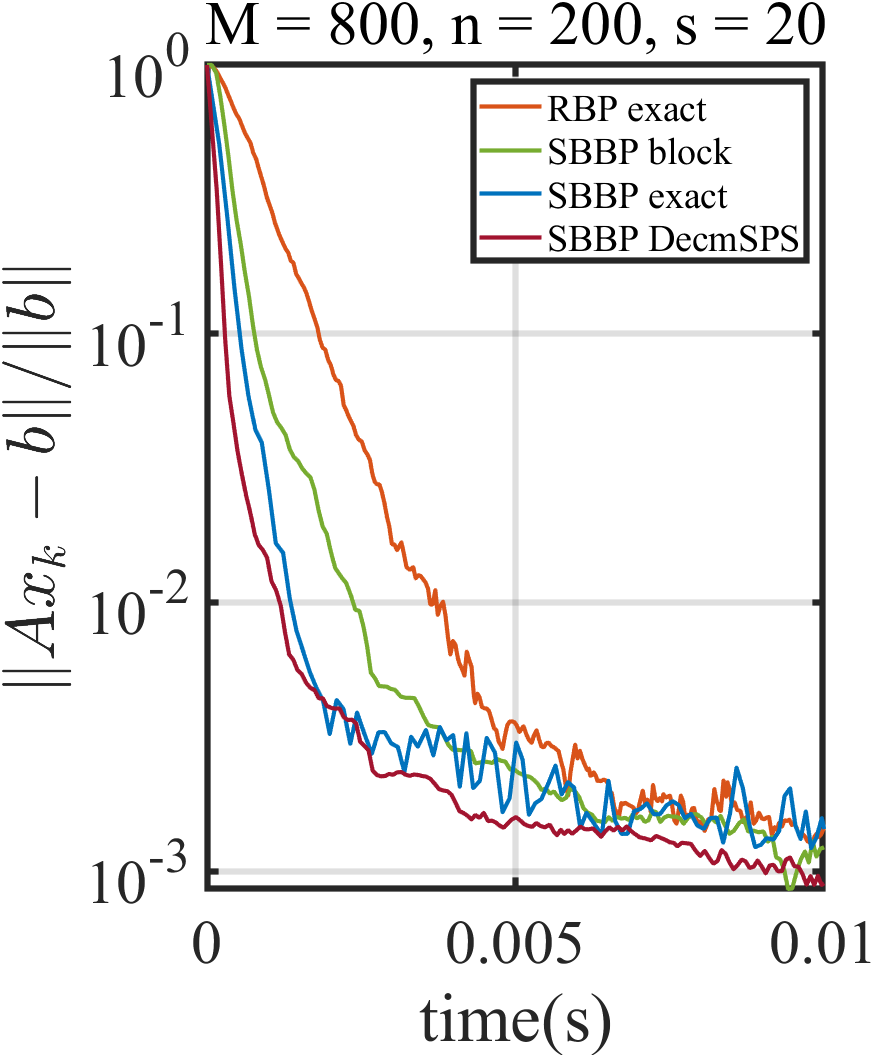}}
	\caption{The performance of RBP and SBBP with different stepsizes with $\sigma=0.01$}
	\label{fig8}
\end{figure}

\begin{figure}[H]
	\centering
	\subfigure[consistent]{
		\includegraphics[width=0.23\linewidth,height=0.22\textwidth]{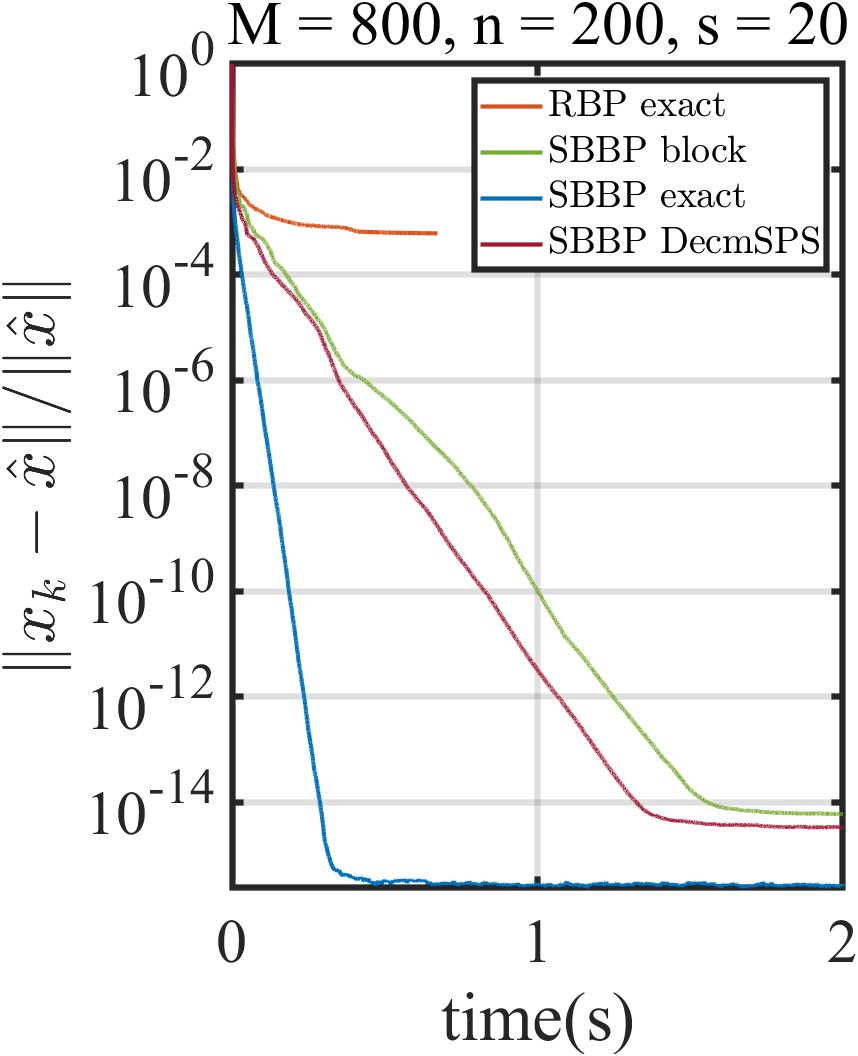}
		\includegraphics[width=0.23\linewidth,height=0.22\textwidth]{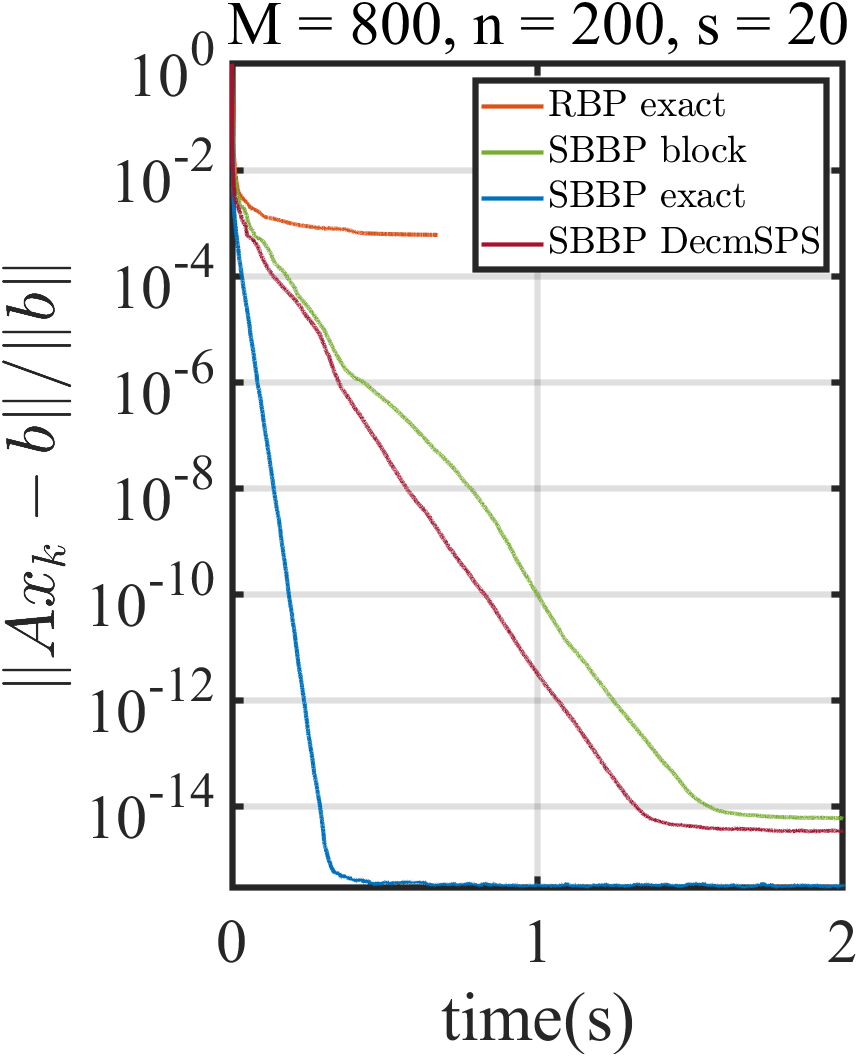}}
	\subfigure[inconsistent]{
		\includegraphics[width=0.23\linewidth,height=0.22\textwidth]{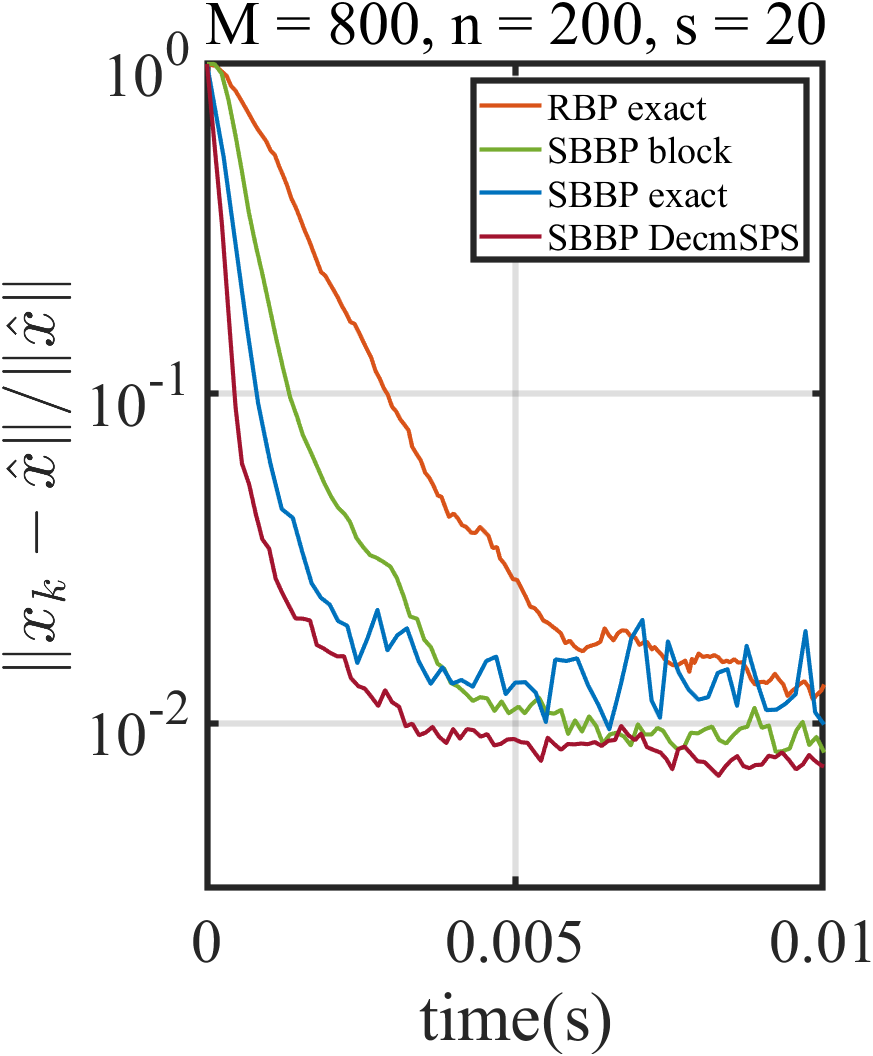}
		\includegraphics[width=0.23\linewidth,height=0.22\textwidth]{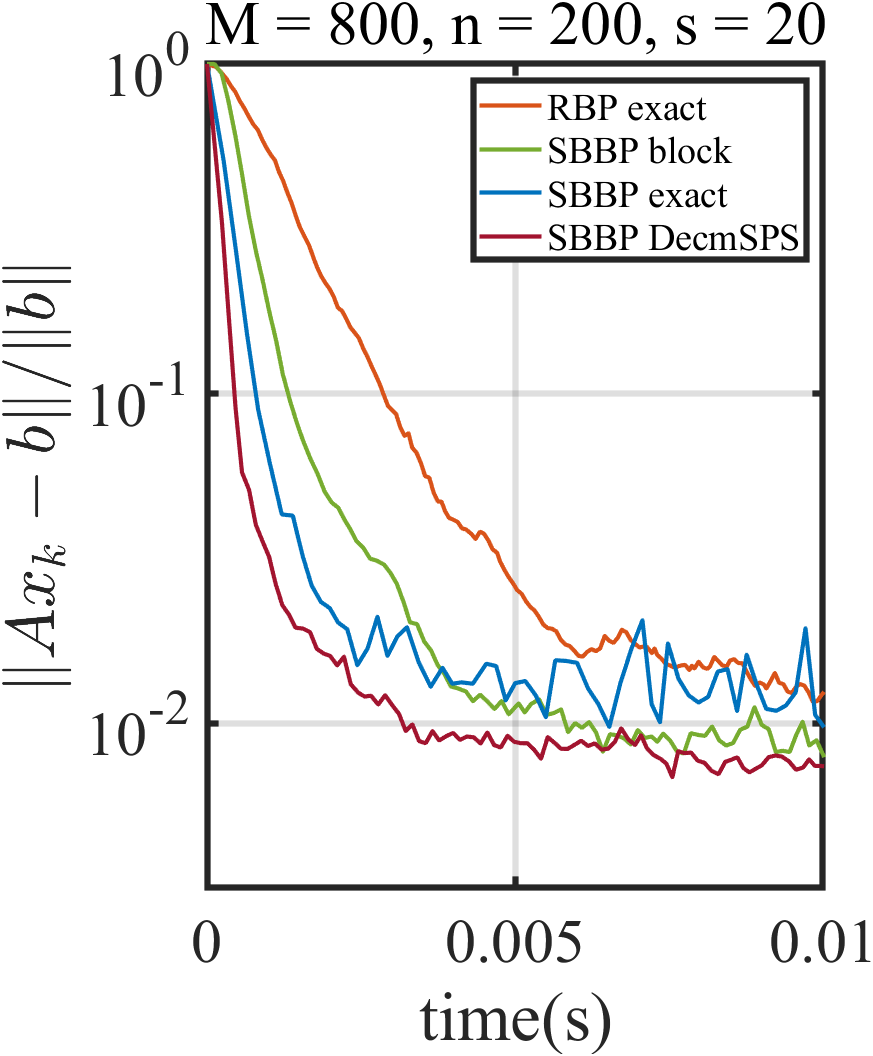}}
	\caption{The performance of RBP and SBBP with different stepsizes with $\sigma=0.1$}
	\label{fig9}
\end{figure}

\section{Conclusion}
\label{sec7}
In this paper, we developed a stochastic bilevel reformulation for convex feasibility problems, which provides a unified treatment of both consistent and inconsistent convex feasibility problems. Within this reformulation, we proposed the stochastic block Bregman projection method with Polyak-like stepsize and projective stepsizes, and established ergodic sublinear and linear convergence in expectation to the inner minimizer set $\bar{X}$ under assumptions on the inner function. The proposed method also recovers and extends several existing projection-based methods, and induces effective algorithms for linear and split feasibility problems. Numerical experiments further demonstrated the effectiveness of the proposed method.


In the future, we will focus on further accelerating the proposed framework by incorporating Nesterov's momentum \cite{nesterov1983method,nesterov2013gradient}, exploring advanced Polyak-like stepsizes such as the AdaSPS \cite{jiang2023adaptive}, and investigating efficient sampling methods. Another direction is to modify the SBBP method so as to guarantee convergence to the solution of the proposed bilevel optimization problem. Furthermore, we plan to extend the SBBP framework to the more general common fixed point problem, which involves finding a point in the intersection of the fixed point sets of a finite family of operators \cite{censor2022generalized}.

\backmatter


\bmhead{Acknowledgements}
This work was funded by the National Natural Science Foundation of China (No.12471300, No.61977065), and the China Scholarship Council program (Project ID: 202406110004).
The authors declare no competing interests. All authors contributed to the study's conception and design. The first draft of the
manuscript was written by Lu Zhang, and all authors commented on previous versions of the manuscript.
All authors read and approve the final manuscript and are all aware of the current submission.



\begin{appendices}

\section{The Proof in Section \ref{sec3}}\label{secA1}

\subsection*{The proof of Theorem \ref{th4.1}}
\begin{proof}
	For the inexact stepsize $\tilde{t}_k$, the update (\ref{eq4.5}) takes the form
	\begin{equation*}
	\label{eq4.9}
	\left\{\begin{aligned}
	\tilde{x}_{k+1}^*&=x_k^*-\tilde{t}_k\cdot \alpha_k,\\
	\tilde{x}_{k+1}&=\nabla\psi^*(\tilde{x}_{k+1}^*).
	\end{aligned}\right.
	\end{equation*} 
	Thus we have $\tilde{x}_{k+1}^*\in\partial\psi(\tilde{x}_{k+1})$. By the uniqueness of the Bregman projection, we need to show that $\tilde{x}_{k+1}$ is the Bregman projection of $x_k$ onto the halfspace $H_k$. It follows from Lemma \ref{lemma2.1} that we only need to prove 
	\begin{equation}
	\label{eq4.10}
	\langle \tilde{x}_{k+1}^*-x_k^{*},x-\tilde{x}_{k+1}\rangle \geq 0,\forall x\in H_k.
	\end{equation}
	Let $L_{\max}=\max_{i\in [m]}L_i$. For any $x\in H_k$, we have 
	\begin{equation*}
	\begin{aligned}
	&~~~~\langle \tilde{x}_{k+1}^*-x_k^{*},x-\tilde{x}_{k+1}\rangle\\
	&=\langle -\tilde{t}_k\cdot \alpha_k,x-\nabla\psi^*(x_k^*-\tilde{t}_k\cdot \alpha_k)\rangle\\
	&=\tilde{t}_k(\langle \alpha_k,\nabla\psi^*(x_k^*-\tilde{t}_k\cdot \alpha_k)\rangle-\langle \alpha_k,x\rangle)\\
	&\geq \tilde{t}_k\left(\langle \alpha_k,\nabla\psi^*(x_k^*-\tilde{t}_k\cdot \alpha_k)\rangle-\langle \alpha_k,x_k\rangle+\sum_{i\in J_k} \frac{w_{i,k}}{L_i}\|\nabla f_i(x_k)\|^2\right)\\
	&=\tilde{t}_k\left(\langle \alpha_k,\nabla\psi^*(x_k^*-\tilde{t}_k\cdot \alpha_k)-\nabla\psi^*(x_k^*)\rangle+\sum_{i\in J_k} \frac{w_{i,k}}{L_i}\|\nabla f_i(x_k)\|^2\right)\\
	&\geq \tilde{t}_k\left(
	-\|\alpha_k\|\cdot\|\nabla\psi^*(x_k^*-\tilde{t}_k\cdot \alpha_k)-\nabla\psi^*(x_k^*)\|+\sum_{i\in J_k} \frac{w_{i,k}}{L_i}\|\nabla f_i(x_k)\|^2
	\right)\\
	&\geq 
	\tilde{t}_k\left(
	-\frac{\tilde{t}_k}{\mu}\|\sum_{i\in J_k} w_{i,k}\nabla f_i(x_k)\|^2+
	\frac{1}{L_{\max}}\sum_{i\in J_k} w_{i,k}\|\nabla f_i(x_k)\|^2
	\right),
	\end{aligned}
	\end{equation*}
	where the first inequality uses $\langle \alpha_k,x\rangle \leq \beta_k,\forall x\in H_k$, the second follows from the Cauchy-Schwarz inequality, and the third from $\tilde{t}_k>0$ and the Lipschitz continuity of $\nabla \psi^*$ with $1/\mu$.
	Hence, we conclude that the condition (\ref{eq4.10}) holds for all stepsize satisfying
	$$
	0< \tilde{t}_k\leq \frac{\mu}{L_{\max}}\frac{\sum_{i\in J_k} w_{i,k}\|\nabla f_i(x_k)\|^2}{\|\sum_{i\in J_k} w_{i,k}\nabla f_i(x_k)\|^2}=\frac{\mu}{L_{\max}L_{\operatorname{adapt}}^{k,\tau}}.
	$$
	This completes the proof.
\end{proof}

\subsection*{The proof of Lemma \ref{lemma5.1}}

\begin{proof}
	Firstly, we consider the iterate sequence $\{(x_k,x_k^*)\}$ generated by Algorithm \ref{al1} with the exact stepsize $t_k=\hat{t}_k$, which is the exact solution to the minimization problem (\ref{eq4.6}). 
	It follows from the definition of Bregman distance and the update in Algorithm \ref{al1} that for $x_{k+1}^*\in\partial \psi(x_{k+1})$ and $\forall~x\in\mathbb{R}^n$ we have 
	\begin{equation*}
	\begin{aligned}
	D_{\psi}^{x_{k+1}^*}(x_{k+1},x)
	&=\psi^*(x_{k+1}^*)+\psi(x)-\langle x_{k+1}^*,x\rangle\\
	&=\psi^*(x_k^*-t_k\sum_{i\in J_k} w_{i,k}\nabla f_i(x_k))+\psi(x)-\langle x_k^*-t_k\sum_{i\in J_k} w_{i,k}\nabla f_i(x_k),x\rangle\\
	&=\psi^*(x_k^*-t_k\sum_{i\in J_k} w_{i,k}\nabla f_i(x_k))+t_k\langle \sum_{i\in J_k} w_{i,k}\nabla f_i(x_k),x\rangle
	+\psi(x)-\langle x_k^*,x\rangle.
	\end{aligned}
	\end{equation*}
	For any $x\in H_k$, we have $\langle \sum_{i\in J_k} w_{i,k}\nabla f_i(x_k),x\rangle\leq \beta_k$; combining with the fact that $t_k$ is the exact solution to the minimization problem (\ref{eq4.6}), we can further derive that for any $x\in H_k$ and $t>0$, it holds that
	\begin{equation}
	\begin{aligned}
	\label{eq5.1}
	&~~~~D_{\psi}^{x_{k+1}^*}(x_{k+1},x)\\
	&\leq \psi^*(x_k^*-t_k\sum_{i\in J_k} w_{i,k}\nabla f_i(x_k))+t_k\beta_k
	+\psi(x)-\langle x_k^*,x\rangle\\
	&\leq \psi^*(x_k^*-t\sum_{i\in J_k} w_{i,k}\nabla f_i(x_k))+t\beta_k
	+\psi(x)-\langle x_k^*,x\rangle\\
	&\leq \psi^*(x_k^*)+\langle -t\sum_{i\in J_k} w_{i,k}\nabla f_i(x_k),x_k\rangle+\frac{t^2}{2\mu}\|\sum_{i\in J_k} w_{i,k}\nabla f_i(x_k)\|^2+t\beta_k+\psi(x)-\langle x_k^*,x\rangle\\
	&=
	D_{\psi}^{x_{k}^*}(x_{k},x)
	-t \sum_{i\in J_k} \frac{w_{i,k}}{L_i}\|\nabla f_i(x_k)\|^2 +\frac{t^2}{2\mu}\|\sum_{i\in J_k} w_{i,k}\nabla f_i(x_k)\|^2.
	\end{aligned}
	\end{equation}
	where the third inequality holds due to the fact that $\psi^*$ is $1/\mu$-gradient-Lipschitz-continuous.
	It follows from
	$\|\sum_{i\in J_k} w_{i,k}\nabla f_i(x_k)\|^2=L_{\operatorname{adapt}}^{k,\tau}\sum_{i\in J_k} w_{i,k}\|\nabla f_i(x_k)\|^2$ and $L_i\leq L_{\max},i\in [m]$ that (\ref{eq5.1}) can be rewritten into
	\begin{equation}
	\begin{aligned}
	\label{eq5.2}
	D_{\psi}^{x_{k+1}^*}(x_{k+1},x)
	&\leq 
	D_{\psi}^{x_{k}^*}(x_{k},x)
	-\left(\frac{t}{L_{\max}}-\frac{t^2}{2\mu}L_{\operatorname{adapt}}^{k,\tau}\right) \sum_{i\in J_k} w_{i,k}\|\nabla f_i(x_k)\|^2, 
	\end{aligned}
	\end{equation}
	we observe that when $t\in (0,\frac{2\mu}{L_{\max}L_{\operatorname{adapt}}^{k,\tau}})$, the term $\left(\frac{t}{L_{\max}}-\frac{t^2}{2\mu}L_{\operatorname{adapt}}^{k,\tau}\right)$ is nonnegative so that (\ref{eq5.2}) is always a descent property.

	Secondly, for any $x\in H_k$ we have
	$$\begin{aligned}
	D_{\psi}^{x_{k+1}^*}(x_{k+1},x)
	&\leq 
	D_{\psi}^{x_{k}^*}(x_{k},x)
	+\inf_{t>0}\left(-\frac{t}{L_{\max}}+\frac{t^2}{2\mu}L_{\operatorname{adapt}}^{k,\tau}\right) \sum_{i\in J_k} w_{i,k}\|\nabla f_i(x_k)\|^2\\
	&=
	D_{\psi}^{x_{k}^*}(x_{k},x)
	-\frac{\mu}{2L_{\max}^2L_{\operatorname{adapt}}^{k,\tau}} \sum_{i\in J_k} w_{i,k}\|\nabla f_i(x_k)\|^2.
	\end{aligned}$$
	The descent property can also be obtained by letting $t=\frac{\mu}{L_{\max}L_{\operatorname{adapt}}^{k,\tau}}$. The proof is completed.
\end{proof}

\section{The Proof in Section \ref{sec4}}\label{secA2}

\subsection*{The proof of Theorem \ref{th4:1}}
\begin{proof}	
	It follows from the definition of $t_k$ in (\ref{eq4:1}) that we have
	\begin{equation}
	\label{eq4:2}
	t_k^2\|\sum_{i\in J_k} w_{i,k}\nabla f_i(x_k)\|^2\leq \mu t_k\lambda_k\sum_{i\in J_k} w_{i,k} (f_i(x_k)-l_{J_k}^{*}).
	\end{equation}	
	Thus, for any $x\in\mathbb{R}^n$, (\ref{eq4.2}) can be reformulated as
	\begin{equation}
	\label{eq4:3}
	\begin{aligned}
	&~~~~D_{\psi}^{x_{k+1}^*}(x_{k+1},x)\\
	&\leq
	D_{\psi}^{x_{k}^*}(x_{k},x)
	-t_k\sum_{i\in J_k} w_{i,k} (f_i(x_k)-f_i(x))
	+\frac{t_k\lambda_k}{2}\sum_{i\in J_k} w_{i,k} (f_i(x_k)-l_{J_k}^{*})\\
	&=D_{\psi}^{x_{k}^*}(x_{k},x)
	-t_k\left(1-\frac{\lambda_k}{2}\right)\sum_{i\in J_k} w_{i,k} (f_i(x_k)-f_i(x))
	+\frac{t_k\lambda_k}{2} \sum_{i\in J_k} w_{i,k} (f_i(x)-l_{J_k}^{*}).
	\end{aligned}
	\end{equation}
	By a simple reformulation with $0<\lambda_k\leq 1$, we can deduce that
	\begin{equation}
	\label{eq4:22}
\frac{t_k}{2}\sum_{i\in J_k} w_{i,k} (f_i(x_k)-f_i(x))
	\leq 
	D_{\psi}^{x_{k}^*}(x_{k},x)-D_{\psi}^{x_{k+1}^*}(x_{k+1},x)+\frac{t_k\lambda_k}{2} \sum_{i\in J_k} w_{i,k} (f_i(x)-l_{J_k}^{*}).
	\end{equation}
	Dividing both sides of (\ref{eq4:22}) by $t_k>0$ implies
	$$
	\frac{1}{2}\sum_{i\in J_k} w_{i,k} (f_i(x_k)-f_i(x))
	\leq \frac{D_{\psi}^{x_{k}^*}(x_{k},x)}{t_k}-\frac{D_{\psi}^{x_{k+1}^*}(x_{k+1},x)}{t_k}+\frac{\lambda_k}{2} \sum_{i\in J_k} w_{i,k} (f_i(x)-l_{J_k}^{*}).
	$$	
	Summing from $k=0$ to $K-1$ obtains
	\begin{equation}
	\begin{aligned}
	&~~~~\frac{1}{2}\sum_{k=0}^{K-1}\sum_{i\in J_k} w_{i,k} (f_i(x_k)-f_i(x))\\
	&\leq \sum_{k=0}^{K-1}\frac{D_{\psi}^{x_{k}^*}(x_{k},x)}{t_k}-\sum_{k=0}^{K-1}\frac{D_{\psi}^{x_{k+1}^*}(x_{k+1},x)}{t_k}+\sum_{k=0}^{K-1}\frac{\lambda_k}{2}\sum_{i\in J_k} w_{i,k} (f_i(x)-l_{J_k}^{*})\\
	&=\frac{D_{\psi}^{x_{0}^*}(x_{0},x)}{t_0}+\sum_{k=1}^{K-1}\frac{D_{\psi}^{x_{k}^*}(x_{k},x)}{t_k}
	-\sum_{k=0}^{K-2}\frac{D_{\psi}^{x_{k+1}^*}(x_{k+1},x)}{t_k}
	-\frac{D_{\psi}^{x_{K}^*}(x_{K},x)}{t_{K-1}}\\
	&~~~~~~~~~~~~~~~~~~~~
	+\sum_{k=0}^{K-1}\frac{\lambda_k}{2}\sum_{i\in J_k} w_{i,k} (f_i(x)-l_{J_k}^{*})\\
	&\leq
	\frac{D_{\psi}^{x_{0}^*}(x_{0},x)}{t_0}
	+\sum_{k=0}^{K-2}\left(\frac{1}{t_{k+1}}-\frac{1}{t_{k}}\right)D_{\psi}^{x_{k+1}^*}(x_{k+1},x)
	+\sum_{k=0}^{K-1}\frac{\lambda_k}{2}\sum_{i\in J_k} w_{i,k} (f_i(x)-l_{J_k}^{*}).
	\end{aligned}
	\end{equation}
	For a fixed $\hat{x}\in\bar{X}$, define $D=\max_{k\in [K-1]} D_{\psi}^{x_{k}^*}(x_{k},\hat{x})$, then we have
	\begin{equation}
	\label{eq4:23}
	\begin{aligned}
	&~~~~\frac{1}{2}\sum_{k=0}^{K-1}\sum_{i\in J_k} w_{i,k} (f_i(x_k)-f_i(\hat{x}))\\
	&\leq 
	D\left[\frac{1}{t_0}+\sum_{k=0}^{K-2}\left(
	\frac{1}{t_{k+1}}-\frac{1}{t_k}
	\right) \right]
	+\sum_{k=0}^{K-1}\frac{\lambda_k}{2}\sum_{i\in J_k} w_{i,k} (f_i(\hat{x})-l_{J_k}^{*})\\
	&\leq 
	\frac{D}{t_{K-1}}+\sum_{k=0}^{K-1}\frac{\lambda_k}{2}\sum_{i\in J_k} w_{i,k} (f_i(\hat{x})-l_{J_k}^{*}),
\end{aligned}
\end{equation}
where the first inequality holds since $t_{k+1}\leq t_k,~k\in\mathbb{N}$ by Lemma \ref{lemma4.1}. Moreover, it follows from Lemma \ref{lemma4.1} that we have $t_k\geq \lambda_k \tilde{L}$. Thus (\ref{eq4:23}) turns into
	\begin{equation*}
	\begin{aligned}
	~~~~\sum_{k=0}^{K-1}\sum_{i\in J_k} w_{i,k} (f_i(x_k)-f_i(\hat{x}))
	&\leq
\frac{2D}{\lambda_{K-1}\tilde{L}}+\sum_{k=0}^{K-1}\lambda_k\sum_{i\in J_k} w_{i,k} (f_i(\hat{x})-l_{J_k}^{*}).
	\end{aligned}
	\end{equation*}

	Taking expectation on both sides yields
	\begin{equation}
	\begin{aligned}
	\label{eq4:4}
	\sum_{k=0}^{K-1} \mathbb{E}[\sum_{i\in J_k} w_{i,k} (f_i(x_k)-f_i(\hat{x}))]
	\leq
	\frac{2D}{\lambda_{K-1}\tilde{L}}+\sum_{k=0}^{K-1}\lambda_k\mathbb{E}[\sum_{i\in J_k} w_{i,k} (f_i(\hat{x})-l_{J_k}^{*})].
	\end{aligned}
	\end{equation} 
	Under Assumption \ref{assume.sampling}, we have for each $k\geq 0$
	\begin{equation}
	\label{eq4:21}
\mathbb{E}\left[\sum_{i\in J_k} w_{i,k} (f_i(x_k)-f_i(\hat{x}))\right]=\mathbb{E}[F(x_k)-\bar{F}].
	\end{equation}
	Moreover, by the definition of $\sigma_{\tau}^2$ in (\ref{eq4:8}), it holds that $\mathbb{E}\left[\sum_{i\in J_k} w_{i,k} (f_i(\hat{x})-l_{J_k}^{*})\right]=\sigma_\tau^2$.
	Combining (\ref{eq4:4}) and (\ref{eq4:21}), and dividing by $K$, we get
	\begin{equation*}
	\frac{1}{K}\sum_{k=0}^{K-1}\mathbb{E}[F(x_k)-\bar{F}]
	\leq
	\frac{1}{K}\frac{2D}{\lambda_{K-1}\tilde{L}}+\frac{1}{K}\sigma_{\tau}^2\sum_{k=0}^{K-1}\lambda_k.
	\end{equation*}
	Using $\bar{x}_{K}=\frac{1}{K}\sum_{k=0}^{K-1}x_k$ and convexity of $F$, Jensen's inequality yields
	$$
\mathbb{E}\left[F(\bar{x}_{K})-\bar{F}\right]
	\leq
	\frac{1}{K}\left(\frac{2D}{\lambda_{K-1}\tilde{L}}+\sigma_{\tau}^2 \sum_{k=0}^{K-1}\lambda_k\right).
	$$
 	The proof is completed.
\end{proof}

\subsection*{The Proof of Theorem \ref{th4:2}}

\begin{proof}
	Starting from (\ref{eq4:3}), for any $x\in \bar{X}$ we obtain
	\begin{equation}
	\label{eq4.4}
	\begin{aligned}
	&~~~~D_{\psi}^{x_{k+1}^*}(x_{k+1},x)\\
	&\leq
	D_{\psi}^{x_{k}^*}(x_{k},x)
	-t_k\left(1-\frac{\lambda_k}{2}\right)\sum_{i\in J_k} w_{i,k} (f_i(x_k)-f_i(x))
	+\frac{t_k\lambda_k}{2}\sum_{i\in J_k} w_{i,k} (f_i(x)-l_{J_k}^{*})\\
	&=
	D_{\psi}^{x_{k}^*}(x_{k},x)
	-t_k\left(1-\frac{\lambda_k}{2}\right)\sum_{i\in J_k} w_{i,k} (f_i(x_k)-l_{J_k}^{*})
	+t_k\sum_{i\in J_k} w_{i,k} (f_i(x)-l_{J_k}^{*}).
	\end{aligned}
	\end{equation}
	It follows from Lemma \ref{lemma4.1} that $\lambda_k \tilde{L} \leq t_k\leq \lambda_k\frac{\gamma_b}{\lambda_0}$. Since $\sum_{i\in J_k} w_{i,k} (f_i(x_k)-l_{J_k}^{*})\geq 0$ and $\sum_{i\in J_k} w_{i,k} (f_i(x)-l_{J_k}^{*})\geq 0$, for any $x\in\bar{X}$ and $0<\lambda_k<2$ we have
	\begin{equation*}
	\begin{aligned}
	&~~~~D_{\psi}^{x_{k+1}^*}(x_{k+1},x)\\
	&\leq
	D_{\psi}^{x_{k}^*}(x_{k},x)
	-\lambda_k \tilde{L}\left(1-\frac{\lambda_k}{2}\right)\sum_{i\in J_k} w_{i,k} \big(f_i(x_k)-l_{J_k}^{*}\big)
	+\lambda_k\frac{\gamma_b}{\lambda_0}\sum_{i\in J_k} w_{i,k} \big(f_i(x)-l_{J_k}^{*}\big).
	\end{aligned}
	\end{equation*}
	It follows that
	\begin{equation*}
	\begin{aligned}
	&~~~~\inf_{x\in \bar{X}}D_{\psi}^{x_{k+1}^*}(x_{k+1},x)\\
	&\leq
	D_{\psi}^{x_{k}^*}(x_{k},x)
	-\lambda_k \tilde{L}\left(1-\frac{\lambda_k}{2}\right)\sum_{i\in J_k} w_{i,k} \big(f_i(x_k)-l_{J_k}^{*}\big)
	+\lambda_k\frac{\gamma_b}{\lambda_0}\sum_{i\in J_k} w_{i,k} \big(f_i(x)-l_{J_k}^{*}\big).
	\end{aligned}
	\end{equation*}

	Taking expectation conditional to $\mathcal{F}_k=\{J_0,\ldots,J_{k-1}\}$, for any $x\in\bar{X}$ we obtain
	\begin{equation}
	\label{eq6:2}
	\begin{aligned}
	\mathbb{E}\!\left[\inf_{x\in \bar{X}}D_{\psi}^{x_{k+1}^*}(x_{k+1},x)\mid \mathcal{F}_k\right]
	&\leq
	D_{\psi}^{x_{k}^*}(x_{k},x)
	-\lambda_k \tilde{L}\left(1-\frac{\lambda_k}{2}\right)\mathbb{E}\!\left[\sum_{i\in J_k} w_{i,k} \big(f_i(x_k)-l_{J_k}^{*}\big)\mid \mathcal{F}_k\right]\\
	&\quad\quad\quad\quad\quad\quad
	+\lambda_k\frac{\gamma_b}{\lambda_0}\mathbb{E}\!\left[\sum_{i\in J_k} w_{i,k} \big(f_i(x)-l_{J_k}^{*}\big)\mid \mathcal{F}_k\right].
	\end{aligned}
	\end{equation}
	By Assumption \ref{assume.sampling} and $x\in\bar{X}$, it holds that
	$$\begin{aligned}
	&~~~~\mathbb{E}\left[\sum_{i\in J_k} w_{i,k} \big(f_i(x_k)-l_{J_k}^{*}\big)\mid \mathcal{F}_k\right]\\
	&=\mathbb{E}\left[\sum_{i\in J_k} w_{i,k} \big(f_i(x_k)-f_i(x)\big)\mid \mathcal{F}_k\right]+\mathbb{E}\left[\sum_{i\in J_k} w_{i,k} \big(f_i(x)-l_{J_k}^{*}\big)\mid \mathcal{F}_k\right]\\
	&=F(x_k)-\bar{F}+\sigma_{\tau}^2,
	\end{aligned}$$
	and 
	$$\mathbb{E}\!\left[\sum_{i\in J_k} w_{i,k} \big(f_i(x)-l_{J_k}^{*}\big)\mid \mathcal{F}_k\right]=\sigma_{\tau}^2.$$
	Thus, for any $x\in\bar{X}$ (\ref{eq6:2}) can be rewritten as
	\begin{equation}
	\label{eq6:3}
	\begin{aligned}
	\mathbb{E}\!\left[\inf_{x\in \bar{X}}D_{\psi}^{x_{k+1}^*}(x_{k+1},x)\mid \mathcal{F}_k\right]
	&\leq
	D_{\psi}^{x_{k}^*}(x_{k},x)
	-\lambda_k \tilde{L}\left(1-\frac{\lambda_k}{2}\right)(F(x_k)-\bar{F}+\sigma_{\tau}^2)
	+\lambda_k\frac{\gamma_b}{\lambda_0}\sigma_{\tau}^2\\
	&\leq D_{\psi}^{x_{k}^*}(x_{k},x)
	-\lambda_k \tilde{L}\left(1-\frac{\lambda_k}{2}\right)(F(x_k)-\bar{F})
	+\lambda_k\frac{\gamma_b}{\lambda_0}\sigma_{\tau}^2.
	\end{aligned}
	\end{equation}
	Taking the infimum over $x\in\bar{X}$ yields
	\begin{equation}
	\label{eq6:4}
	\begin{aligned}
	\mathbb{E}\!\left[\inf_{x\in \bar{X}}D_{\psi}^{x_{k+1}^*}(x_{k+1},x)\mid \mathcal{F}_k\right]
	&\leq \inf_{x\in \bar{X}}D_{\psi}^{x_{k}^*}(x_{k},x)
	-\lambda_k \tilde{L}\left(1-\frac{\lambda_k}{2}\right)(F(x_k)-\bar{F})
	+\lambda_k\frac{\gamma_b}{\lambda_0}\sigma_{\tau}^2.
	\end{aligned}
	\end{equation}
	By Assumption \ref{assume.2}, we have $F(x_k)-\bar{F}\geq\gamma\,\operatorname{dist}_{\psi}^{x_{k}^*}(x_{k},\bar{X})^2$. Hence, taking the full expectation on (\ref{eq6:4}) yields
	\begin{equation*}
	\mathbb{E}\!\left[\operatorname{dist}_{\psi}^{x_{k+1}^*}(x_{k+1},\bar{X})^2\mid \mathcal{F}_k\right]
	\leq
	\left[1-\lambda_k\left(1-\frac{\lambda_k}{2}\right)\gamma \tilde{L}\right]\operatorname{dist}_{\psi}^{x_{k}^*}(x_{k},\bar{X})^2
	+\lambda_k\frac{\gamma_b}{\lambda_0}\sigma_{\tau}^2.
	\end{equation*}
Applying the tower property of conditional expectation yields
\begin{equation*}
\mathbb{E}\!\left[\operatorname{dist}_{\psi}^{x_{k+1}^*}(x_{k+1},\bar{X})^2\right]
\leq
\left[1-\lambda_k\left(1-\frac{\lambda_k}{2}\right)\gamma \tilde{L}\right]\mathbb{E}\!\left[\operatorname{dist}_{\psi}^{x_{k}^*}(x_{k},\bar{X})^2\right]
+\lambda_k\frac{\gamma_b}{\lambda_0}\sigma_{\tau}^2.
\end{equation*}
This completes the proof.
\end{proof}

\subsection*{The Proof of Theorem \ref{th3}}

\begin{proof}
		Starting from Lemma \ref{lemma5.1} (b), for any $x\in H_k$ it holds that
		$$
		D_{\psi}^{x_{k+1}^*}(x_{k+1},x)
		\leq
		D_{\psi}^{x_{k}^*}(x_{k},x)
		-\frac{\mu}{2L_{\max}^2L_{\operatorname{adapt}}^{k,\tau}}\sum_{i\in J_k} w_{i,k}\|\nabla f_i(x_k)\|^2.
		$$
		It follows from $L_{\operatorname{adapt}}^{k,\tau}\leq L_{\tau}^{\operatorname{block}}$ that we obtain
		\begin{equation}
		\label{eq4:12}
		D_{\psi}^{x_{k+1}^*}(x_{k+1},x)
		\leq
		D_{\psi}^{x_{k}^*}(x_{k},x)
		-\frac{\mu}{2L_{\max}^2L_{\tau}^{\operatorname{block}}}\sum_{i\in J_k} w_{i,k}\|\nabla f_i(x_k)\|^2.
		\end{equation}
 Hence, for any $x\in\bar{X}$, taking the infimum over $x\in\bar{X}$ on (\ref{eq4:12}) and then taking expectation conditional on $\mathcal{F}_k=\{J_0,\ldots,J_{k-1}\}$ yield
		\begin{equation}
		\label{eq4:10}
		\begin{aligned}
		&~~~~\mathbb{E}\!\left[\operatorname{dist}_{\psi}^{x_{k+1}^*}(x_{k+1},\bar{X})^2\mid \mathcal{F}_k \right]\\
		&\leq
		\mathbb{E}\!\left[\operatorname{dist}_{\psi}^{x_{k}^*}(x_{k},\bar{X})^2\mid \mathcal{F}_k\right]
		-\frac{\mu}{2L_{\max}^2L_{\tau}^{\operatorname{block}}}\mathbb{E}\!\left[\sum_{i\in J_k} w_{i,k}\|\nabla f_i(x_k)\|^2\mid \mathcal{F}_k\right].
		\end{aligned}
		\end{equation}
		By Assumption \ref{assume.sampling}, the definition of $L$, we have
		\begin{equation}
		\label{eq4:7}
		\begin{aligned}
		\mathbb{E}\!\left[\sum_{i\in J_k} w_{i,k}\|\nabla f_i(x_k)\|^2\mid \mathcal{F}_k\right]
		&=\mathbb{E}[\|\nabla f_i(x_k)\|^2\mid \mathcal{F}_k]
		\geq
		\frac{1}{L}\|\nabla F(x_k)\|^2.
		\end{aligned}
		\end{equation}
		Combining (\ref{eq4:10}) and (\ref{eq4:7}) and taking full expectation on both sides, we obtain
		$$\mathbb{E}\operatorname{dist}_{\psi}^{x_{k+1}^*}(x_{k+1},\bar{X})^2
		\leq
		\mathbb{E}\operatorname{dist}_{\psi}^{x_{k}^*}(x_{k},\bar{X})^2
		-\frac{\mu}{2LL_{\max}^2L_{\tau}^{\operatorname{block}}}
		\mathbb{E}\|\nabla F(x_k)\|^2.$$
		Summing the above inequality from $k=0$ to $k=K-1$ gives
		$$
		\frac{\mu}{2LL_{\max}^2L_{\tau}^{\operatorname{block}}}\sum_{k=0}^{K-1}\mathbb{E}\|\nabla F(x_k)\|^2
		\leq
		\operatorname{dist}_{\psi}^{x_{0}^*}(x_{0},\bar{X})^2.
		$$
		Dividing both sides by $K$ yields
		$$
		\min_{0\leq k\leq K-1}\mathbb{E}\|\nabla F(x_k)\|^2
		\leq
		\frac{1}{K}\sum_{k=0}^{K-1}\mathbb{E}\|\nabla F(x_k)\|^2
		\leq
		\frac{2LL_{\max}^2L_{\tau}^{\operatorname{block}}}{\mu K}
		\operatorname{dist}_{\psi}^{x_{0}^*}(x_{0},\bar{X})^2.
		$$
		
		Additionally, by (\ref{eq4:7}) and the PL inequality on $F$, we have
		\begin{equation*}
		\label{eq4:7pl}
		\begin{aligned}
		\mathbb{E}\!\left[\sum_{i\in J_k} w_{i,k}\|\nabla f_i(x_k)\|^2\mid \mathcal{F}_k\right]
		&
		\geq
		\frac{1}{L}\|\nabla F(x_k)\|^2
        \geq \frac{\gamma\mu}{2L}[F(x_k)-\bar{F}].
		\end{aligned}
		\end{equation*}
		Similarly, we can deduce that
		$$
		\mathbb{E}[F(\bar{x}_K)-\bar{F}]
		\leq 
		\frac{4LL_{\max}^2L_{\tau}^{\operatorname{block}}}{\mu^2\gamma K}
		\operatorname{dist}_{\psi}^{x_{0}^*}(x_{0},\bar{X})^2,
		$$
		which proves the sublinear convergence rate.
\end{proof}

\subsection*{The proof of Theorem \ref{th5.5}}

\begin{proof}
	 Starting from (\ref{eq4:12}), (\ref{eq4:7}), and the sampling method in Assumption \ref{assume.sampling}, taking the infimum over $\bar{X}$ and then taking the conditional expectation with respect to $J_k$ given $\mathcal{F}_k$, we obtain
	\begin{equation}
	\label{eq6:5}
	\begin{aligned}
	\mathbb{E}\left[\inf_{x\in \bar{X}}D_{\psi}^{x_{k+1}^*}(x_{k+1},x)\mid \mathcal{F}_k\right]
	&\leq
	\inf_{x\in \bar{X}}D_{\psi}^{x_{k}^*}(x_{k},x)
	-\frac{\mu}{2L_{\max}^2L_{\tau}^{\operatorname{block}}}
	\mathbb{E}\left[\|\nabla f_i(x_k)\|^2\mid \mathcal{F}_k\right].
	\end{aligned}
	\end{equation}
	By the definition of $L$, the PL inequality (\ref{eq4:9}) deduced from Assumption \ref{assume.2}, and Assumption \ref{assume.2}, we obtain
	\begin{equation}
	\label{eq6:8}
	\begin{aligned}
	\mathbb{E}\left[\|\nabla f_i(x_k)\|^2\mid \mathcal{F}_k\right]	
	&\geq 
	\frac{1}{L}\|\nabla F(x_k)\|^2
	\geq 
	\frac{\gamma\mu}{2L}(F(x_k)-\bar{F})
	\geq
	\frac{\gamma^2\mu}{2L}\inf_{x\in \bar{X}}D_{\psi}^{x_k^*}(x_k,x).
	\end{aligned}
	\end{equation}	
	Combining (\ref{eq6:5}) and (\ref{eq6:8}) leads to
	$$\begin{aligned}
	\mathbb{E}\left[\inf_{x\in\bar{X}}D_{\psi}^{x_{k+1}^*}(x_{k+1},x)\mid \mathcal{F}_k\right]
	&\leq 
	\left(1-\frac{\gamma^2\mu^2}{4L_{\max}^2L_{\tau}^{\operatorname{block}}L}\right)
	\inf_{x\in\bar{X}}D_{\psi}^{x_{k}^*}(x_{k},x).
	\end{aligned}
	$$
	 Taking full expectation on both sides yields
	$$ \mathbb{E}[\inf_{x\in\bar{X}}D_{\psi}^{x_{k+1}^*}(x_{k+1},x)]
	\leq
	\left(1-\frac{\gamma^2\mu^2}{4L_{\max}^2L_{\tau}^{\operatorname{block}}L}\right)\mathbb{E}[\inf_{x\in\bar{X}}D_{\psi}^{x_{k}^*}(x_{k},x)].
	$$
	The proof is completed.
	
\end{proof}

\subsection*{The proof of Proposition \ref{th5.6}}

\begin{proof}
	According to the three-points identity in Lemma 3.1 of \cite{chen1993convergence}, for any convex and continuously differentiable function $\psi:\mathbb{R}^n\rightarrow \mathbb{R}$ and any three points $x,y,z\in\mathbb{R}^n$, the following identity holds
	\begin{equation*}
	D_{\psi}(x,y)+D_{\psi}(y,z)-D_{\psi}(x,z)=\langle \nabla \psi(x)-\nabla \psi(y),z-y\rangle.
	\end{equation*}
	Using the three-points identity, we have for any $x\in\mathbb{R}^n$
	\begin{equation*}
	D_{\psi}(x_k,x_{k+1})+D_{\psi}(x_{k+1},x)-D_{\psi}(x_k,x)=\langle \nabla \psi(x_k)-\nabla \psi(x_{k+1}),x-x_{k+1}\rangle.
	\end{equation*}
	Since $\psi$ is differentiable, we have $x_k^*=\nabla\psi(x_k)$ for all $k$. It follows from the SBBP update that $\nabla \psi(x_k)-\nabla \psi(x_{k+1})=t_k\nabla f_{J_k}(x_k)$, and therefore we arrive at
	\begin{equation}
	\label{eq7:2}
	D_{\psi}(x_k,x_{k+1})+D_{\psi}(x_{k+1},x)-D_{\psi}(x_k,x)=\langle t_{k}\nabla f_{J_k}(x_k),x-x_{k+1}\rangle.
	\end{equation}
	From the three-points identity applied to $f_{J_k}$, we have
	\begin{equation}
	\label{eq7:3}
	t_kD_{f_{J_k}}(x_k,x_{k+1})+t_kD_{f_{J_k}}(x_{k+1},x)-t_kD_{f_{J_k}}(x_k,x)=t_k\langle \nabla f_{J_k}(x_k)-\nabla f_{J_k}(x_{k+1}),x-x_{k+1}\rangle.
	\end{equation}
	When $t_k\in (0,\frac{\mu}{L_{\max}}]$, we can verify that $\psi-t_kf_{J_k}$ is a convex function since $\psi$ is $\mu$-strongly convex and each $f_i,~i=1,\cdots,m$ is $L_i$-smooth with $L_{\max}=\max_{i\in [m]}L_i$.
	Subtracting (\ref{eq7:3}) from (\ref{eq7:2}) we obtain
	$$
	D_{\psi-t_kf_{J_k}}(x_k,x_{k+1})+D_{\psi-t_kf_{J_k}}(x_{k+1},x)-D_{\psi-t_kf_{J_k}}(x_k,x)=\langle t_{k}\nabla f_{J_k}(x_{k+1}),x-x_{k+1}\rangle,
	$$
	which can be reformulated as
	\begin{equation}
	\begin{aligned}
	\label{eq7:4}
	&~~~~D_{\psi}(x_k,x)\\
	&=t_kD_{f_{J_k}}(x_k,x)+D_{\psi}(x_{k+1},x)-t_kD_{f_{J_k}}(x_{k+1},x)+D_{\psi-t_kf_{J_k}}(x_k,x_{k+1})\\
	&\qquad\qquad\qquad\qquad -\langle t_{k}\nabla f_{J_k}(x_{k+1}),x-x_{k+1}\rangle\\
	&=t_kD_{f_{J_k}}(x_k,x)+D_{\psi}(x_{k+1},x)+D_{\psi-t_kf_{J_k}}(x_k,x_{k+1})\\
	&\qquad\qquad\qquad\qquad -t_k[D_{f_{J_k}}(x_{k+1},x)+\langle \nabla f_{J_k}(x_{k+1}),x-x_{k+1}\rangle]\\
	&=t_kD_{f_{J_k}}(x_k,x)+D_{\psi}(x_{k+1},x)+D_{\psi-t_kf_{J_k}}(x_k,x_{k+1})-t_k[f_{J_k}(x)-f_{J_k}(x_{k+1})].
	\end{aligned}
	\end{equation}
	For the consistent CFP, we assume that $\hat{x}$ is the unique solution of the bilevel optimization (\ref{eq3.3}). Then we have $f_{i}(\hat{x})=0,~i=1,\cdots,m$ and (\ref{eq7:4}) implies that
	\begin{equation}
	\label{eq7:8}
	D_{\psi}(x_{k+1},\hat{x})\leq D_{\psi}(x_k,\hat{x})- t_kD_{f_{J_k}}(x_k,\hat{x}).
	\end{equation}
	
	It follows from the $\mu_f$-strong convexity of $f_i$ that we have 
	\begin{equation}
	\label{eq7:9}
	D_{f_{J_k}}(x_k,\hat{x})\geq \frac{\mu_f}{2}\|x_k-\hat{x}\|^2.
	\end{equation}
	Combining (\ref{eq7:8}) and (\ref{eq7:9}), we have
	$$D_{\psi}(x_{k+1},\hat{x})\leq D_{\psi}(x_k,\hat{x})- \frac{\mu_f}{2}t_k\|x_k-\hat{x}\|^2.$$
	Assume that $\psi$ has an $L_{\psi}$-Lipschitz continuous gradient, then we have
	$$\frac{L_{\psi}}{2}\|x_k-\hat{x}\|^2\geq D_{\psi}(x_k,\hat{x}).$$
	Therefore, it holds that
	\begin{equation}
	\label{eq7:11}
	D_{\psi}(x_{k+1},\hat{x})\leq
	\left(1- \frac{\mu_f}{L_{\psi}}t_k\right) D_{\psi}(x_k,\hat{x}).
	\end{equation}
	Taking expectations on both sides of (\ref{eq7:11}) yields the stated result.
\end{proof}

\end{appendices}

\bibliography{ref}

\end{document}